\def\style{} % ``a'' for TAC paper
\if a\style
\documentclass[a4paper, 10pt, conference]{cssconf}      % Use this line for a4
\IEEEoverridecommandlockouts                              % This command is only
\input{tsumua}
\else
\documentclass[10pt]{article}
\def\parindentwidth{4mm}
\def\vs{18.4pt}
\newtheorem{theorem}{Theorem}[section]

\newtheorem{proposition}{Proposition}[section]
\newtheorem{lemma}{Lemma}[section]
\newtheorem{definition}{Definition}[section]
\newtheorem{note}{Note}[section]
\newtheorem{problem}{Problem}[section]

\newtheorem{example}{Example}[section]

\newtheorem{assumption}{Assumption}[section]

\newenvironment{proof}{{\parindent=0pt \par \bf Proof\quad}}{\hfill$\square$\par}

\def\sqr#1#2{{\vcenter{\hrule height.#2pt
        \hbox{\vrule width.#2pt height#1pt \kern#1pt
          \vrule width.#2pt}
        \hrule height.#2pt}}}
\def\square{\mathchoice\sqr34\sqr74\sqr{2.1}3\sqr{1.5}3}
\def\matsp#1{\hbox to 10pt {#1}}

\def\diag{{\rm diag}}
\def\trace{{\rm tr\, }}
\def\rref#1{(\ref{#1})}

\def\bunken(#1){$^{#1}$}

\def\tsection(#1)
{
\setlength{\parindent}{0pt}
\stepcounter{section}
\vspace{10pt}
\begin{centerline}
{\large\bf \thesection{.}\quad{#1}}
\end{centerline}
\vspace{10pt}
\setlength{\parindent}{\parindentwidth}
}

\def\tsubsection(#1)
{
\setlength{\parindent}{\parindentwidth}
\stepcounter{subsection}
{\bf \thesubsection\quad{#1}}
\setlength{\parindent}{\parindentwidth}
\par
}

\def\tsectionappen(#1)
{
\setlength{\parindent}{0pt}
\stepcounter{section}
\vspace{10pt}
\begin{centerline}
{\large\bf \thesection{}\quad{#1}}
\end{centerline}
\vspace{10pt}
\setlength{\parindent}{\parindentwidth}
}

\def\ttsection(#1)
{
\setlength{\parindent}{0pt}
\stepcounter{section}
\vspace{10pt}
\begin{centerline}
{\large\bf \thesection{}\quad{#1}}
\end{centerline}
\vspace{10pt}
\setlength{\parindent}{\parindentwidth}
}

\newenvironment{equation*}{\[}{\]}

\fi

\usepackage[dvips]{graphics}
\usepackage{tepsf}
\usepackage{times}
\usepackage{comment}

\excludecomment{commentout}
\newsavebox{\rdummy}
\newenvironment{r-commentout}{\global\setbox\rdummy = \vbox\bgroup}{\egroup}

\newsavebox{\tdummy}
\newenvironment{t-commentout}{\global\setbox\tdummy = \vbox\bgroup}{\egroup}
\def\rregi{\tilde{\phi}_b}
\def\rregj{\tilde{\phi}_c}
\def\rregh{\tilde{\phi}_a}
\def\rregk{\tilde{\phi}_d}

\def\reg{\tilde{\phi}_1}

\def\regk{\tilde{\phi}_k}
\def\th1{\tilde{\theta}_1}
\def\tphi{\tilde{\phi}}

\def\goptv{g_{\rm v}}
\def\goptf{g_{\rm f}}
\def\qy{{y}'}
\def\qY{{Y}'}
\def\qyone{{y}_{\langle{1}\rangle}'}
\def\qytwo{{y}_{\langle{2}\rangle}'}
\def\qythree{{y}_{\langle{3}\rangle}'}
\def\kappaY{\kappa_y}
\def\i{t}
\def\j{{\langle{j}\rangle}}
\def\S{{\cal S}^{y}}
\def\Phi{{\cal S}^{\phi}}
\def\I{{\cal S}^{\reg}}
\def\V{{\sf V}}
\def\E{{\sf E}}
\def\regjp{(\tilde{\phi}_{1})_{\j}'}
\def\Ih{{\cal S}^{\rregh}}
\def\Ii{{\cal S}^{\rregi}}
\def\Ij{{\cal S}^{\rregj}}
\def\Ik{{\cal S}^{\rregk}}
\def\II{{\Ih \times \Ii \times \Ij \times \Ik}}

\def\den{\H_0
+ 
\sum_i
\H_{i}(\tilde{\phi}_i-\tilde{\phi}_{i}')
+
\sum_{i,j} 
\H_{ij}(\tilde{\phi}_i-\tilde{\phi}_{i}')(\tilde{\phi}_j-\tilde{\phi}_{j}')
+
O((\tilde{\phi}_i-\tilde{\phi}_{i}')(\tilde{\phi}_j-\tilde{\phi}_{j}')(\tilde{\phi}_k-\tilde{\phi}_{k}'))}

\def\H{{\delta}}

\def\star{{}}
\def\plim{{\rm plim}}

\def\M{{M'}}
\def\Mo{M}

\def\figwidth{150pt}
\def\minipagewidth{7.8cm}

\def\Qopt{Q$_{\rm opt}$}

\def\mathbold(#1){\mbox{\boldmath $#1$}}

\title{
\LARGE \bf
\vspace{-1cm}
Optimal Quantization of Signals for System Identification\thanks{The
technical report/conference versions of this paper are in
\cite{Tsumura:CUED02andECC03,Tsumura:METR04-10,Tsumura:METR05-04}.}
}
\author{Koji Tsumura\thanks{Koji Tsumura is with 
Department of Information Physics and Computing, 
The University of Tokyo, 
Hongo 7--3--1, Bunkyo-ku, 
Tokyo 113--8656, Japan, 
tel: +81--3--5841--6891, fax: +81--3--5841--6886,
e-mail: {\tt\small tsumura@i.u-tokyo.ac.jp}}%
}
\begin{document}

\maketitle

\if a\style
\else
   \setlength{\baselineskip}{\vs}
\fi

\vspace{-.5cm}

%#!platex paper

\noindent
{\bf Abstract:}
In this paper, we examine the optimal quantization of signals for system
identification.  We deal with memoryless quantization for the output
signals and derive the optimal quantization schemes.  The objective
functions are the errors of least squares parameter estimation subject
to a constraint on the number of subsections of the quantized signals or
the expectation of the optimal code length for either high or low
resolution.  In the high-resolution case, the optimal quantizer is found
by solving Euler--Lagrange's equations and the solutions are simple
functions of the probability densities of the regressor vector.  In
order to clarify the minute structure of the quantization, the optimal
quantizer in the low resolution case is found by solving recursively a
minimization of a one-dimensional rational function.  The solution has
the property that it is coarse near the origin of its input and becomes
dense away from the origin in the usual situation.  Finally the required
quantity of data to decrease the total parameter estimation error,
caused by quantization and noise, is discussed.  

\noindent
{\bf Keywords:}
system identification, quantization, networked control, least squares method, FIR model, 
entropy

%#!platex paper

\section{Introduction}\label{sec:introduction}

The recent rapid improvement in the transmission capacity of computer
networks has made long-distance automatic control more realistic and
the necessity of understanding the effects of transmission limitations
on the information in control systems has become more widely accepted. In
particular, quantization of the signals to reduce the
information content of the transmitted signals in control systems has been
discussed actively by several control research groups during the last few
years and interesting results have been achieved.  

The problem of signal quantization has a long history going back to
the 1940s, and is one of main themes in the area of information theory (e.g., see
\cite{Gray:TIT98}). The problem is to attain low
distortion between the original and the quantized signals subject to
constraints on the amount of information.  Naturally, the situations and
objectives for data transmission and those for control systems are
essentially different and the need for research on the latter
case has been recognized.  
However, although elementary discussion in the control
community dates from the 1970s (e.g., see \cite{Curry:book}), rigorous analysis
did not begin until the late 1980s.
%[OLE1]  
The main difficulty of quantization in control systems lies
in their dynamics; the result by \cite{Delchamps:SCL89,Delchamps:IEEE90} is
recognized as a breakthrough, in which the behavior of control
systems and their stability or state estimation are analyzed in detail.
In the last few years, stabilization problems of quantized systems
have been actively investigated in several different situations, e.g.,
\cite{Wong:IEEE97,Wong:IEEE99,Brockett:IEEE00,Nair:scl00,Elia:IEEE01,Nair:cdc02,Tsumura:CDC03,Nair:siam04}.
Of these, a logarithmic quantizer was shown to be coarsest, in some
sense, to achieve a kind 
%form
of asymptotic stability \cite{Elia:IEEE01} and 
reveal the variations in the importance of signals, depending on their
magnitudes and the directions in the signal space, from the viewpoint
%[OLE2] 
of system control. 

With this background, our interests naturally shifted to the system
identification problem; that is, what quantization scheme is {\it
optimal} for system identification?  We expect that the answer to this
question will clarify the amount of information in the signals necessary for
parameter estimation.  Unfortunately however, compared to the research activity in the
stabilization or estimation problem, the optimal quantization
problem for system identification \cite{Gevers:book} has not been
adequately considered.  
%From this viewpoint, we deal with this problem.  
The main subject of this paper is to answer this fundamental question.  

In this paper, we consider the optimal memoryless quantization problem of output
signals that are used for parameter estimation.  The identified system
is a simple single input single output (SISO) finite impulse response
(FIR) model, in order to reveal the essential properties of
the optimal quantization in system identification and help intuitively
understanding it.  By {\it optimality} in this paper we mean the
minimization of the variance of the parameter estimation error given by
the least squares method with a constraint on the number of quantization
steps or the expectation of the code length of the optimally coded
quantized signals.  
We consider this problem for two cases: 
(1) high quantization resolution with weak assumptions on input,
%[OLE3]  
(2) low quantization resolution, 
however with some specific assumptions on input. 
The difficulty with the problem is in the complex correlation between
the input signals and the quantization errors, and solving this is the
key for the optimization problem.  

In the high resolution case (Section~\ref{section:main}), 
%we consider a generalization of the previous
%results subject to weak conditions on the distribution of input signals.  
%The straightforward extension of the approach taken in
%Section~\ref{sec:prev} is difficult for such 
%%to extend to the general 
%case because of
%the complexity of the calculation of quantization error.  
%To resolve this difficulty, 
we introduce a key concept, the density of the number of quantized
subsections, and by using calculus of variations, analytic solutions are
derived subject to the constraint on the number of quantization steps or the
optimal code length. 
The solutions are functions of the probability density of the input
signals and we can rigorously calculate the profile of the density of the
number of the optimally quantized subsections.  
Moreover, these results suggest several insights into system
identification with finite information.  We illustrate
these facts for some cases and describe the complexity of the problem
of system identification.  

The results in Section~\ref{section:main} show that the quantization
resolution around the origin of the signals relatively becomes coarse in
usual cases.  In order to clarify the minute structure of the
quantization and complement the results in Section~\ref{section:main}, 
we consider the low resolution case in Section~\ref{sec:prev}.  
We give the optimal quantizer with a condition of uniform distribution
of input signals.
The optimal quantizer is given by minimizing a one-dimensional rational
function recursively.   
In a special case, we show that the optimal quantization is not uniform and 
it is coarse near the origin of the quantized signals and becomes
dense away from the origin.  
This fundamental property is opposite to the case of stabilization in 
\cite{Elia:IEEE01} and reveals duality between system
identification and stabilization.  

Finally, in Section~\ref{sec:eval}, we compare the effects of the
resolution of quantization and the I/O data length.  
The results show that the former is more effective for decreasing
quantization error in the estimated system parameters, on the other
hand, the latter is more effective in reducing noise error.  
%[OLE4]
From this, there exists a trade-off between these two error terms  
subject to a constant amount of data and we can find an
appropriate quantizer resolution to balance them by using the results in
Section~\ref{sec:eval}.  

Note that the main purpose of this paper is to reveal the essential
properties of the optimal quantization for system identification;
therefore, the focus of this paper is on the analysis of this problem
and not on practical system identification methods.  

In this paper, 
most of the proofs of theorems, lemmas, or
propositions are collected in the appendix for ease of understanding 
the main theme and the outline of this paper.  
Refer to these in Appendix~\ref{appendex:proofs}
if necessary.  
 
\newpage

\noindent
{\bf Notation:} 

\begin{minipage}[t]{8.2cm}
\baselineskip 16pt
$d_j$: eq.~\rref{eq:part} and \rref{eq:dd}
\hfill\break
$\E[x]$: expectation of $x$, \,  
%\hfill\break
$\E_{\bullet}[x]$: eq.~\rref{eq:limitE}
\hfill\break
%$E = [e(1)\; e(2) \cdots e(N)]^{\rm T}$: quantization error vector
%\hfill\break
%$\Delta E$: quantization error term \rref{eq:defofE}
%\hfill\break
$e(\i) = \qy(\i)-y(\i)$: quantization error at $\i$
\hfill\break
$e(\reg(\i)) = e(\i)$: quantization error specified by $\reg$
\hfill\break
$f(x)$: probability density of $x$
\hfill\break
$g(\bullet)$: eq.~\rref{eq:defofg}
\hfill\break
$H(\bullet)$: entropy of $\bullet$, \, 
$H(\bullet,\bullet)$, $H_{\rm d}(\bullet)$: eq.~\rref{eq:entropy}
\hfill\break
$j$: index of quantized subsections
\hfill\break
$\Mo$: number of quantization subsections
\hfill\break
$\M$: associate number of quantization subsections \rref{eq:defofM}
\hfill\break
$N$: data length
\hfill\break
$n$: order of FIR model 
\hfill\break
$O(\bullet)$, $o(\bullet)$: orders of $\bullet$ (Landau's symbols) 
\hfill\break
${\cal P}(\bullet)$: eq.~\rref{eq:symmain}
\hfill\break
%$q(\bullet)$: memoryless quantizer \rref{eq:quant10}
%\hfill\break
\end{minipage}
\hfill
\begin{minipage}[t]{8.2cm}

\baselineskip 16pt
$r_j$, $r_j^o$: ratio or optimal ratio of $d_j$ and $d_{j+1}$ \rref{eq:rdef}
\hfill\break
${\cal S}_j^\bullet$: $j$-th subsection on the space of $\bullet$
\hfill\break
$T$: variable transformation matrix
\hfill\break
%$\i$: time index 
%\hfill\break
%$U$: input data matrix \rref{eq:defofU}, \,
%%\hfill\break
%$u(\i)$: input at $\i$
%\hfill\break
$\V[x]$: expectation of $\|x\|_2^2$, \,
%\hfill\break
$\V_{\bullet}[x]$: eq.~\rref{eq:limv}
\hfill\break
%$\V_{M}\left[\reg\cdot e(\reg)\right]$: eq.~\rref{eq:vari00}
%\hfill\break
%$w(\i)$: noise at $\i$
%\hfill\break
$y(\i) = \phi(\i) \theta$: output of FIR model at $\i$
\hfill\break
$y_o(\i)$: observed output \rref{eq:sys}
\hfill\break
$\theta \in {\cal R}^n$: parameter vector of FIR model 
\hfill\break
$\phi(\i)$: regressor vector eq.~\rref{eq:sys}
\hfill\break
$\reg$: $1$st element of $\tilde{\phi}$
\hfill\break
$\sigma(\reg)$: eq.~\rref{eq:variance101}
\hfill\break
$\bullet_i$, $(\bullet)_i$: $i$-th element of vector $\bullet$
\hfill\break
$\bullet'$: quantized number of $\bullet$
\hfill\break
$\bullet_\j'$: $j$-th quantized number for ${\cal S}_j^\bullet$
\hfill\break
$\tilde{\bullet}$: transformed vector or matrix of $\bullet$ by $T$
\hfill\break
\end{minipage}

%#!platex paper

\section{Problem Formulation}
\label{sec:formulation}
%\section{Preliminaries}\label{sec:formulation}

The objective of this paper is to show the effect of I/O
signal quantizers for parameter estimation error intuitively
understandable form as possible.  
In general, 
the quantization error has a strong correlation with the original signal, 
%
%the quantization error behaves as a random signal when a
%quantizer has a high resolution; this condition is usually 
%assumed in signal processing.  Naturally, however, 
%. 
%
therefore, %a rigorous 
analysis 
%is desirable to understand the essence 
of the quantization problem in system identification in general model is difficult 
because several types of correlation are used for parameter estimation.  
In order to derive 
%Conversely, we should also note that the derivation of 
analytic and intuitively understandable results for the quantization problem in system
identification,  
% with general models is complex. 
we should formulate the problem in feasible forms appropriately.  

From the above observations, 
in this paper, we deal with a system identification problem by
least square criterion for a simple discrete time 
SISO FIR model.  
The plant is: 
\begin{eqnarray}
&&
y_o(\i) = q(y(\i)) + w(\i), \; y(\i) = \phi(\i)\theta, 
\label{eq:sys}
\\
&&
\phi(\i) : =  \left[
\matrix{
u(\i) & u(\i-1) & \cdots & u(\i-n+1)
}
\right], 
%\nonumber\\
%&&
\;\;
\theta : =  \left[
\matrix{
\theta_1 &
\theta_2 &
\cdots &
\theta_n
}
\right]^{\rm T},
\nonumber
\\
&&
y_o, \, y, \, w, \, u \in {\cal R}, \, 
\phi \in {\cal R}^{1\times n}, \,
\theta \in {\cal R}^{n\times 1}, \,
\nonumber
\end{eqnarray}
where $w$ is random noise, $q$ is the quantized original analogue output
$y$, $y_o$ is the observed output, $\phi$ is the regressor vector, $\theta$ is a system parameter, 
$n$ is the dimension of the FIR model, $u$ is the input, and $\i$ is the time index.

We assume that $u$ and $w$ are independent.  
The input $u$ and the associated regressor vector $\phi$ are a realization of
a stochastic process with a joint density function 
$f({\phi}_1, \phi_2, \dots , {\phi}_n)$ of $\phi_1$, $\phi_2$, $\dots$ , $\phi_n$,
where $\phi_i$ denotes the $i$-th element of $\phi$. 
The class of $f({\phi}_1, \phi_2, \dots , {\phi}_n)$ considered in this
paper is described below.  

\begin{note}\label{note:noise}
\rm
We also consider noise to be 
\begin{equation}
y_o(\i) = q(y(\i)+w(\i))
\label{eq:sysnoise}
\end{equation}
in \cite{Tsumura:METR05-04} (the long version of this paper).  
The result suggests that the noise when \rref{eq:sysnoise}
increases the effect of quantization on the magnitude of the parameter
estimation error by approximately twice that of \rref{eq:sys}.   
From that result, it is enough to analyze the form of \rref{eq:sys} in
order to know the essential property of the optimal quantization.  
To avoid complicated notation and focus on the quantization effect for 
system identification, we treat the plant \rref{eq:sys} in this paper. 
\hfill{$\diamondsuit$}
\end{note}

The quantizer $q$ is a memoryless symmetric type defined by:
\begin{eqnarray}
&&
q(y): =  
%\mbox{sgn}(y) 
\qy_{\j} \;\mbox{when}\; y\in \S_j
%, \; \qy^j \geq 0
\label{eq:quant10}
\\
&&
\S_0 : =  \left\{
y = 0
\right\},
\;
\S_j : =  \left\{
y : d_{j-1} < y \leq d_j
\right\}, \; j > 0,
%\nonumber
%\\
%&&
\;\;
\S_j : =  \left\{
y : d_{j} \leq y < d_{j+1}
\right\}, \; j < 0
\label{eq:part}
\\
&&
d_0 = 0 < d_{1} < d_{2} \cdots 
%< d_{M} = \kappa
, \;\; 
%\nonumber
%\\
%&&
%\;
d_{-1} = - d_1, \; d_{-2} = - d_2, \; \dots \; , 
%\;, 
%\; d_{-M} = - d_{M} = - \kappa,
%\nonumber
%\\
%&
\label{eq:dd}
\end{eqnarray}
where $
%\mbox{sgn}(y)
\qy_\j$ is the assigned quantized value to the
subsection $\S_j$.
The quantizer $q$ is symmetrical with respect to the origin, and hereinafter we
may omit references on the negative subsections $\S_{-1}$, $\S_{-2}$,
$\dots$ if they are obvious from the
context.
Note that a form $\S_0=\{y : -d_1 \leq y \leq d_1\}$ is also possible for $\S_0$, however
it is clarified not to be optimal in Section~\ref{sec:prev} and without
loss of generality, we consider the form of \rref{eq:part} hereafter.  

Following the standard least squares method, 
we propose the estimated parameter $\hat{\theta}$ 
%using the least squares method 
with a sufficient length of I/O data, $\{u(\i)\}$ and $\{y_o(\i)\}$, as:
\begin{equation}
\hat{\theta}
= (U^{\rm T}U)^{-1}U^{\rm T}
Y_o
= (U^{\rm T}U)^{-1}U^{\rm T}\left(
\qY+W
\right)
= (U^{\rm T}U)^{-1}U^{\rm T}\left(
Y+E+W
\right),
\label{eq:thetaorig}
\end{equation}
where
\begin{eqnarray}
&&
U: = 
\left[
\matrix{
\phi(1)^{\rm T} &
\phi(2)^{\rm T} &
\cdots &
\phi(N)^{\rm T}
}
\right]^{\rm T},
%\nonumber
%\\
%&&
\;\;
W: = 
\left[
\matrix{
w(1) &
w(2) &
\cdots &
w(N)
}
\right]^{\rm T},
\nonumber
\\
&&
Y_o: = 
\left[
\matrix{
y_o(1) &
y_o(2) &
\cdots &
y_o(N)
}
\right]^{\rm T}, \;\;
Y: = 
\left[
\matrix{
y(1) &
y(2) &
\cdots &
y(N)
}
\right]^{\rm T},
\nonumber
\\
&&
\qY: = 
\left[
\matrix{
\qy(1) &
\qy(2) &
\cdots &
\qy(N)
}
\right]^{\rm T},
%\nonumber
%\\
%&&
\;\;
\qy(\i) : =  q(y(\i)),
\nonumber
\\
&&
E
: = 
\left[
\matrix{
e(1) &
e(2) &
\cdots &
e(N)
}
\right]^{\rm T}, 
\label{eq:defofU}
\\
&&
e(\i) : =  \qy(\i) - y(\i).
\label{eq:defofe}
\end{eqnarray}
and $N$ is the I/O data length. 
We call $e$ as the quantization error between $\qy$ and $y$.   
%Define the quantization error between $\qy$ and $y$ by:
%\begin{equation}
%e(\i) : =  \qy(\i) - y(\i).
%\label{eq:defofe}
%\end{equation}
The estimated parameter $\hat{\theta}$ can be also written as:
\begin{eqnarray*}
\hat{\theta}
 = 
%& = & 
(U^{\rm T}U)^{-1}U^{\rm T}(U\theta + E + W)
%\nonumber
%\\
%& = & 
 = 
\theta + \Delta E + \Delta W,
\end{eqnarray*}
\vspace{-10mm}
\begin{eqnarray}
E
: = 
%&: = &
\left[
\matrix{
e(1) &
e(2) &
\cdots &
e(N)
}
\right]^{\rm T}, 
%\nonumber
%\\
\;\;
\Delta E 
: = 
(U^{\rm T}U)^{-1}U^{\rm T} E,
%\nonumber
%\\
\;\;
\Delta W
: = 
(U^{\rm T}U)^{-1}U^{\rm T} W.
\label{eq:defofE}
\end{eqnarray}
This shows that the estimation error $\hat{\theta} - \theta$ can be
evaluated from the magnitudes of the {\it quantization error term} $\Delta E$
and the {\it noise error term} $\Delta W$.

In the quantization-free case, i.e. $e=0$,   
\rref{eq:thetaorig} is the standard least squares estimation.  
When $e\not=0$, 
%the signal $y+w$ is unknown, however, 
\rref{eq:thetaorig} is still a realistically reasonable estimation 
subject to the minimization of
\begin{equation}
\E[\|\Delta E\|_2^2]
%\E[\Delta E]=0
\label{eq:bias000}
\end{equation}
because
\begin{equation*}
\E[\|\hat{\theta} - \theta\|_2^2] = 
\E[\|\Delta E + \Delta W\|_2^2] =
\E[\|\Delta E\|_2^2] + \E[\|\Delta W\|_2^2].
%=
%\E[\hat{\theta}] = \theta + \E[\Delta E] + \E[\Delta W]
\end{equation*}
%and when $\E[\Delta E]=0$, 
%$\hat{\theta}$ is the best estimation under the situation that $y+w$ is unknown. 

%In the case that the output $y$ is quantized as \rref{eq:sys}, 
%the similar estimation 
%\begin{equation*}
%\hat{\theta}
% = (U^{\rm T}U)^{-1}U^{\rm T}\left(
%Y'+W
%\right),
%\end{equation*}
%is still reasonable 
%Define the quantization error between $\qy$ and $y$ by:
%\begin{equation}
%e(\i) : =  \qy(\i) - y(\i).
%\label{eq:defofe}
%\end{equation}

The reduction of the noise error term $\Delta W$ is the main theme of
normal system identification and 
has been well investigated.  
On the other hand, 
although the quantization error term $\Delta E$ can be reduced, in general, 
when the resolution of quantizer becomes high, 
there exists a limitation in the reduction because of the constraint of the
resolution of the quantizer and {\it good}
quantizers for reducing $\Delta E$ are expected.   
Here we show an original quantization problem in this paper
which is resolved into feasible ones in Section~\ref{section:main} and \ref{sec:prev}.  
\begin{problem}
\label{problem:original}
%{\bf (original problem)}
%????????????????????????????????????????????
Find an optimal quantizer $q(y)$:
\begin{eqnarray}
&&
\mathop{\min}_{q} {\E}[\|\Delta E\|_2^2]
%\mathop{\min}_{q} {\E}\|\Delta E(f(\phi))\|_2^2
\nonumber
\\
&&
\mbox{\rm s.t. } {\E}[\Delta E] = 0
%\mbox{\rm s.t. } {\E}[\Delta E(f(\phi))] = 0
\end{eqnarray}
under constraint on the quantization resolution.  
\end{problem}
Note that the latter condition is for bias-free of the estimated parameters.  
\begin{note}
\rm
In the field of information theory, 
the quantization problem is also one of the research themes and its objective
is reducing the distortion between the original
signal and the quantized signal subject to constraints on the 
information in the transmitted signals 
\cite{Bennett:BSTJ48,Lloyd:TIT82,Gish:TIT68,Berger:TIT72,Gersho:book}.   
The constraint on the information in signals can be given by the number of
the quantization steps or the mean code length of the associated code.  
The former is called ``fixed-rate quantization'' and the latter
``variable-rate quantization''.  
In contrast, the purpose in system identification should be the
reduction of the estimation error and this is the definitive
difference.  
\hfill{$\diamondsuit$}
\end{note}

In an ordinary probabilistic framework, 
a conventional, and reasonable, method to evaluate the noise error term
$\Delta W$ 
%using a probabilistic approach 
is to show the convergence rate of:
\begin{equation*}
N(U^{\rm T}U)^{-1}
\;
\mathop{\longrightarrow}^{N\rightarrow \infty}
\;
\frac{1}{\sigma_u^2} I,
\;
\frac{1}{N}U^{\rm T}W
\;
\mathop{\longrightarrow}^{N\rightarrow \infty}
\;
O,
\end{equation*}
where $\sigma_u^2$ is the covariance of $u$,  
by using Slutsky's theorem (see Appendix~\ref{appendex:proofs}),  
subject to an assumption of the mutual independence of the input signal $u$ and the noise $w$.
This methodology is also basically applicable to the evaluation of $\Delta E$
in the probabilistic framework.  
However, different from the case of the noise error term, 
$u$ and $e$ are not independent in general 
and the evaluation of $U^{\rm T}E$ is much more complicated.
This means the problem seems to be a vector quantization on 
$U^{\rm T}E$ with a complex multidimensional distribution.  
In general, multidimensional optimal quantization is known to be a difficult
problem for analytical solution except in special cases.  

Our idea to resolve the above difficulty is in showing that the original
problem, i.e., minimizing the cost function on the magnitude of $\Delta E$, 
can be reduced to a feasible problem;
``minimization of a functional of a weighted one-dimensional quantizer,'' 
by following two steps: 
1. finding an equivalent orthogonal quantization on the space of the regressor vector
to the original quantization of the output signals, 
2. reduction of the 
cost
%evaluation 
functions to a suitable
%[OLE5] 
form by using one of the base axes in the regressor vector space.  
%variables of that basis.  
Step 1 is described in this section and Step 2 is
described in Section~\ref{section:main} and \ref{sec:prev}. 
 
We define subsets $\Phi_j$ of the regressor vector $\phi$
associated with the subsection $\S_j$ by:
\begin{equation*}
\Phi_j : =  \left\{
\phi : y = \phi \theta \in \S_j
\right\}.
\end{equation*}
We also consider the following variable transformation:
% \cite{Tsumura:CUED02andECC03}:
\begin{eqnarray}
y & = & \phi\theta = \phi T \cdot T^{-1} \theta
 =  \tilde{\phi}\tilde{\theta},
\; \tilde{\theta}: = T^{-1} \theta  = 
\left[
\matrix{
\tilde{\theta}_1\cr
O
}
\right],
\;
\tilde{\phi}: = 
\phi T  = : 
\left[
\matrix{
\tilde{\phi}_1 &
\tilde{\phi}_2 &
\cdots &
\tilde{\phi}_n
}
\right]
\label{eq:tildephi}
\end{eqnarray}
where $T$ is an orthogonal matrix.   
Note that such $T$ always exists for any $\theta$.  
Then, $\Phi_j$ is represented by:
\begin{equation*}
\Phi_j : =  
\left\{
\matrix{
\left\{
\phi : \reg \tilde{\theta}_1 \in (d_{j-1}, \; d_{j}]
\right\}, & j > 0,
\cr
\left\{
\reg = 0
\right\}, & j = 0,
\cr
\left\{
\phi : \reg \tilde{\theta}_1 \in [d_{j}, \; d_{j+1})
\right\}, & j < 0.
}
\right.
\end{equation*}
We also define subsections on the space $\reg$:
\begin{equation*}
\I_j : =  
\left\{
\matrix{
\left\{
\reg : \reg \tilde{\theta}_1 \in (d_{j-1}, \; d_j]
\right\}, & j > 0,
\cr
\left\{
\reg  = 0
\right\}, & j = 0,
\cr
\left\{
\reg : \reg \tilde{\theta}_1 \in [d_{j}, \; d_{j+1})
\right\}, & j<0.
}
\right.
\end{equation*}
Then, subsections $\S_j$, $\Phi_j$, and $\I_j$ correspond to each other,
and the probability distribution of $y$ depends only on that of
$\reg$.
Therefore, 
the variable $\reg$ and its subsection $\I_j$ are convenient
for analyzing the probability distribution of $y$ and the error $e$.
Fig.~\ref{fig:phi1} and Fig.~\ref{fig:phi2} are
representations of the relationship between $\S_j$, $\Phi_j$, and
$\I_j$ or $y$, $\phi$, and $\reg$.  

\vspace{5mm}

\centerline{
\begin{minipage}[c]{8cm}
\centerline{
\epsfxsize = 200pt \epsfbox{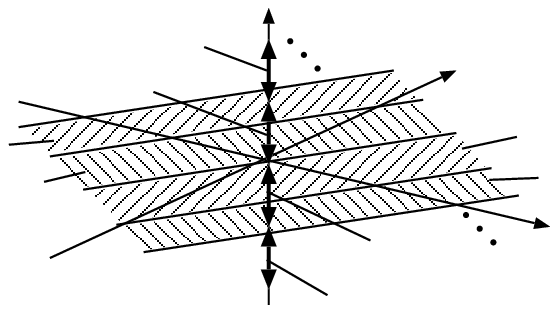}{
\tput(46,97,{${y}$})
\tput(100,28,{$\phi_1$})
\tput(79,84,{$\phi_2$})
\tput(99,45,{$\Phi_2$})
\tput(94,60,{$\Phi_1$})
\tput(-5,56,{$\Phi_{-2}$})
\tput(2,44,{$\Phi_{-1}$})
\tput(32,87,{$\S_2$})
\tput(24,76,{$\S_1$})
\tput(68,22,{$\S_{-2}$})
\tput(60,3,{$\S_{-1}$})
}
}
\refstepcounter{figure}\label{fig:phi1}
%\vspace{5mm}
%\centerline
{Fig.\,{\thefigure} \hspace{2mm} 
Diagram of the relationship between $\S_j$ and $\Phi_j$ for $n = 2$
}
\end{minipage}
}

\vspace{8mm}

%\hfill
%\vspace{1cm}
%
%\centerline{
\begin{minipage}[c]{8cm}
\centerline{
\epsfxsize = 160pt \epsfbox{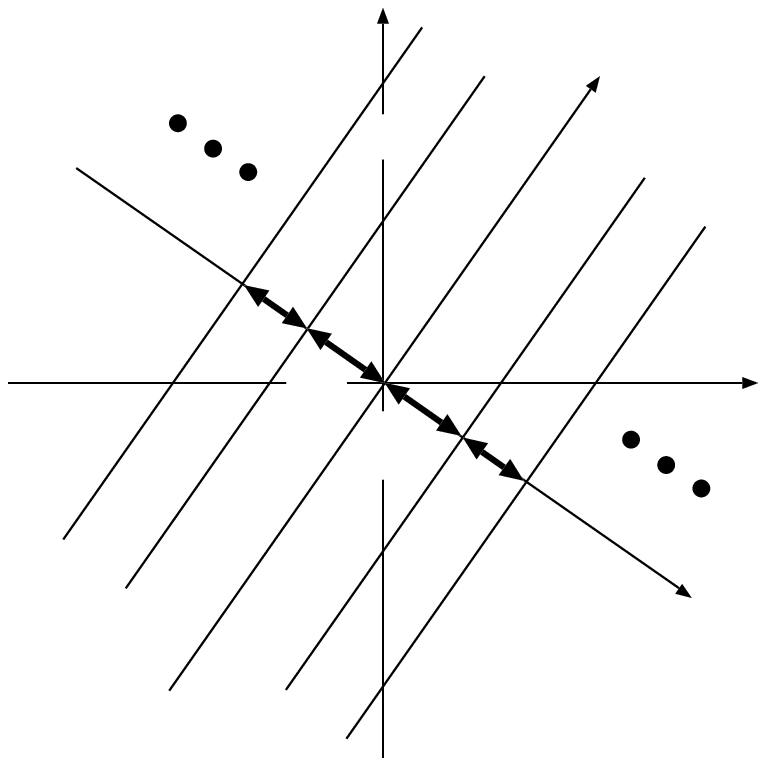}{
\tput(46,81,{$\Phi_{-2}$})
\tput(55,74,{$\Phi_{-1}$})
\tput(67,68,{$\Phi_1$})
\tput(77,62,{$\Phi_2$})
\tput(95,44,{{$\phi_1$}})
\tput(42,95,{{$\phi_2$}})
\tput(82,18,{{$\tilde{\phi}_1$}})
\tput(71,92,{{$\tilde{\phi}_2$}})
\tput(47,39,{{$\I_1$}})
\tput(56,32,{{$\I_2$}})
\tput(36,46,{{$\I_{-1}$}})
\tput(26,53,{{$\I_{-2}$}})
%\tput(54,83,{$\Phi_{-2}$})
%\tput(62,75,{$\Phi_{-1}$})
%\tput(73,67,{$\Phi_1$})
%\tput(81,59,{$\Phi_2$})
%\tput(95,44,{{$\phi_1$}})
%\tput(42,95,{{$\phi_2$}})
%\tput(76,12,{{$\tilde{\phi}_1$}})
%\tput(80,88,{{$\tilde{\phi}_2$}})
%\tput(48,40,{{$I_1$}})
%\tput(57,33,{{$I_2$}})
%\tput(37,49,{{$I_{-1}$}})
%\tput(29,57,{{$I_{-2}$}})
}
}
\refstepcounter{figure}\label{fig:phi2}
%\vspace{5mm}
%\centerline
{Fig.\,{\thefigure} \hspace{2mm} 
Diagram on the relationship between $\Phi_j$ and $\I_j$ for $n = 2$
}
\end{minipage}
%}
\hfill
\begin{minipage}[c]{8cm}
\centerline{
\epsfxsize = 160pt \epsfbox{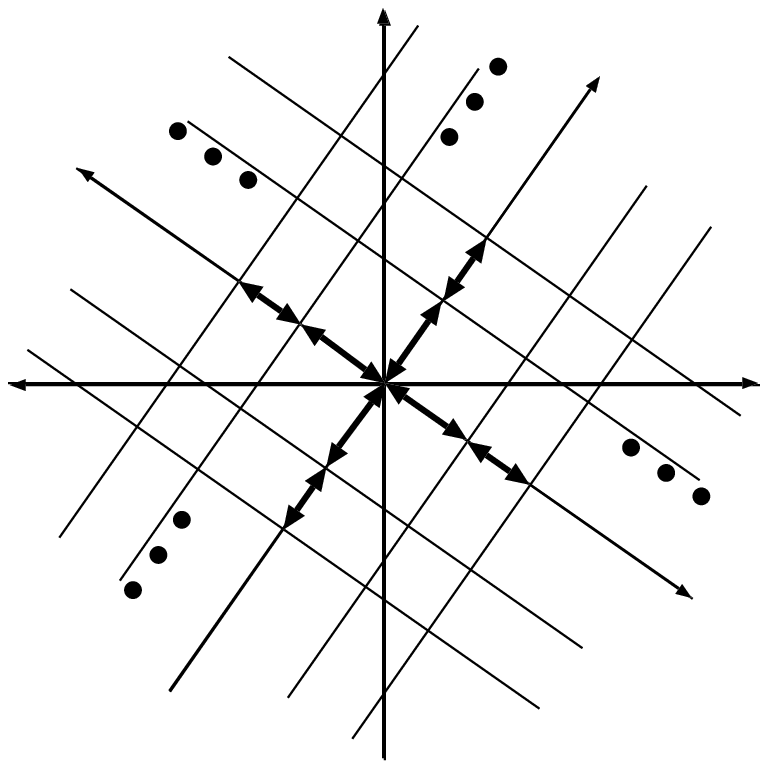}{
%\tput(46,81,{$\Phi_{-2}$})
%\tput(55,74,{$\Phi_{-1}$})
%\tput(67,68,{$\Phi_1$})
%\tput(77,62,{$\Phi_2$})
\tput(95,44,{{$\phi_1$}})
\tput(42,95,{{$\phi_2$}})
\tput(82,18,{{$\tilde{\phi}_1$}})
\tput(71,92,{{$\tilde{\phi}_2$}})
%\tput(47,39,{{$\I_1$}})
%\tput(56,32,{{$\I_2$}})
%\tput(36,46,{{$\I_{-1}$}})
%\tput(26,53,{{$\I_{-2}$}})
%\tput(54,83,{$\Phi_{-2}$})
%\tput(62,75,{$\Phi_{-1}$})
%\tput(73,67,{$\Phi_1$})
%\tput(81,59,{$\Phi_2$})
%\tput(95,44,{{$\phi_1$}})
%\tput(42,95,{{$\phi_2$}})
%\tput(76,12,{{$\tilde{\phi}_1$}})
%\tput(80,88,{{$\tilde{\phi}_2$}})
%\tput(48,40,{{$I_1$}})
%\tput(57,33,{{$I_2$}})
%\tput(37,49,{{$I_{-1}$}})
%\tput(29,57,{{$I_{-2}$}})
}
}
\refstepcounter{figure}\label{fig:phi3}
%\vspace{5mm}
%\centerline
{Fig.\,{\thefigure} \hspace{2mm} 
Quantization on $\phi$ (or $\tilde{\phi}$) for $n = 2$
}
\end{minipage}

\vspace{5mm}

Associated with $T$, the quantization error term $\Delta E$ and $U$ are also transformed to: 
\begin{equation}
\Delta \tilde{E} : =  T^{-1} \Delta E, 
\;
\tilde{U} : =  UT
\label{eq:tildedeltae}
\end{equation}
and $\Delta \tilde{E}$ can be represented as:
\begin{eqnarray}
\Delta \tilde{E}
& = &
T^{-1}({U}^{\rm T}{U})^{-1} {U}^{\rm T}{E}
%\nonumber
%\\
%& = &
 = 
(\tilde{U}^{\rm T}\tilde{U})^{-1} \tilde{U}^{\rm T}{E}
\nonumber
\\
& = &
(\tilde{U}^{\rm T}\tilde{U})^{-1}
\left[
\matrix{
\sum_{\i = 1}^N \tilde{\phi}_1(\i)e(\i)
\cr
\sum_{\i = 1}^N \tilde{\phi}_2(\i)e(\i)
\cr
\vdots
\cr
\sum_{\i = 1}^N \tilde{\phi}_{n}(\i)e(\i)
}
\right]
%\nonumber
%\\
%& = &
 = 
(\tilde{U}^{\rm T}\tilde{U})^{-1}
\left[
\matrix{
\sum_{\i = 1}^N
\tilde{\phi}_1(\i)
(q(\tilde{\phi}_1(\i)\tilde{\theta}_1)-\tilde{\phi}_1(\i)\tilde{\theta}_1)
\cr
\sum_{\i = 1}^N
\tilde{\phi}_2(\i)
(q(\tilde{\phi}_1(\i)\tilde{\theta}_1)-\tilde{\phi}_1(\i)\tilde{\theta}_1)
\cr
\vdots
\cr
\sum_{\i = 1}^N
\tilde{\phi}_n(\i)
(q(\tilde{\phi}_1(\i)\tilde{\theta}_1)-\tilde{\phi}_1(\i)\tilde{\theta}_1)
}
\right].
\label{eq:deltaetilde}
\end{eqnarray}
Note that 
$\|\Delta\tilde{E}\|_2^2 = \|\Delta E\|_2^2$ because $T$ is an orthogonal matrix.   
%
%\vspace{5mm}
%
%\input{figflow}
%
%\vspace{5mm}
%
From the above, it is known that the quantizer can be considered to be an orthogonal and 
symmetric type along each axis $\tilde{\phi}_i$ in the sense that each
axis $\tilde{\phi}_i$ is partitioned in the same rule (see
Fig.~\ref{fig:phi3}). 

%Under appropriate assumptions on input signals explained in
%Section~\ref{section:main} or \ref{sec:prev},  
%we can show 
%$
%\frac{1}{N}\tilde{U}^{\rm T}\tilde{U} {\rightarrow} \sigma_u^2 I
%$ 
%when $N\rightarrow \infty$ 
%and therefore,
%$
%\|\Delta \tilde{E}\|_2^2 \rightarrow_{N\rightarrow \infty} \frac{1}{\sigma_u^{4}N^2}
%\|\tilde{U}^{\rm T}{E}\|_2^2$
%by Slutzky theorem ??????.  
%This shows $\|\tilde{U}^{\rm T}E\|_2^2$ is a reasonable cost function, however,  
%%Note that even if we take the basis $\tilde{\phi}_1$, 
%it 
%%$\|\Delta \tilde{E}\|_2^2$ 
%still contains several cross terms, such as
%\[
%\tilde{\phi}_1(a)q(\tilde{\phi}_1(b)\tilde{\theta}_1) 
%\tilde{\phi}_1(c)q(\tilde{\phi}_1(d)\tilde{\theta}_1),   
%\]
%where $a$ -- $d$ are time indices, 
%and it is not clear whether the evaluation function can be described in
%simple form as described above.  

In Sections~\ref{section:main} and \ref{sec:prev},
% and ~\ref{sec:prev}, 
%which are main results of this paper, 
we first derive key lemmas, respectively, to show that the quantity
%$\|\tilde{U}^{\rm T}E\|_2^2$ 
%or $\|(\Delta \tilde{E}_1)\|_2^2$ 
$\|\Delta {E}\|_2^2=\|\Delta \tilde{E}\|_2^2$ 
can be represented as a functional of the
one-dimensional marginal density function $f(\reg)$ and 
the quantizer on $\reg$, subject to appropriate assumptions. 
%
%Fig.~\ref{fig:flow} is the diagram of the reduction
%from the original problem to tractable problems. 

%#!platex paper

\section{High Resolution Quantization}
\label{section:main}

%As described in Section~\ref{sec:formulation},
%%above, 
%the subject of this paper is mainly to 
%understand the essential properties for optimal quantizers for
%system identification. 
%Therefore, we make a strong assumption in this section.  
%Another purpose is for analysis on common property of optimal quantizers
%for general cases because the profile of the usual input signals'
%multidimensional probability densities in system identification, e.g.,
%normal distribution, is flat near the origin and the above assumption is
%approximately satisfied in such a region.  Therefore, optimal quantizers
%for general distributions usually have the same property in such a
%region as the optimal quantizer given in the following.  

%In Section~\ref{sec:prev}, 
%we gave the minute structure of an optimal general resolution quantizer
%subject to strong assumptions on the probability density $f(\phi)$
%or $f(\tilde{\phi})$.  

In this section, we derive optimal quantizers under considerably 
weak conditions on the probability densities $f(\phi)$ where the
quantizers are assumed to be high resolution.  
At first, we show the following assumption:
% The assumptions used in this section are: 
\begin{assumption}
\label{assumption:0}
The input $u$ and the density function $f(\phi)$
%, or $f(\reg)$, 
%and the quantizer
satisfy the following conditions: 
\begin{enumerate}
\item[1:]
%\item[\ref{section:main}-1)]
$u(\i)$,
%$u(\i) = \phi_1(\i)$, 
$\i =  \dots , 1, 2, \dots$ are mutually independent.
\item[2:]
%\item[\ref{section:main}-2)]
$f(\phi)$ is a continuous function s.t.
%\item[4)]
%\item[\ref{section:main}-4)]
$f(\tilde{\phi})$ satisfies:
\begin{eqnarray}
&&
f(\tilde{\phi}) = 
\H_0
+ 
\sum_i
\H_{i}(\tilde{\phi}_i-\tilde{\phi}_{i}^\circ)
+
\sum_{i,j} 
\H_{ij}(\tilde{\phi}_i-\tilde{\phi}_{i}^\circ)(\tilde{\phi}_j-\tilde{\phi}_{j}^\circ)
+
O((\tilde{\phi}_i-\tilde{\phi}_{i}^\circ)(\tilde{\phi}_j-\tilde{\phi}_{j}^\circ)(\tilde{\phi}_k-\tilde{\phi}_{k}^\circ)), 
\;\;
|\H_\bullet| < \infty
\nonumber
\\
&&
%f(\tilde{\phi}) = \prod_{i = 1}^n(\H_i + \K_i(\tilde{\phi}_i-\tilde{\phi}_{i}^\circ) + O((\tilde{\phi}_i-\tilde{\phi}_{i}^\circ)^2)), 
%\;\;
%|\H_i|, \; |\K_i| < \infty
\label{assumption:high}
%\\
%&&
%\nonumber
\end{eqnarray}
in the neighborhood of an arbitrary 
$\tilde{\phi}^\circ = [\tilde{\phi}_{1}^\circ \; \tilde{\phi}_{2}^\circ \; \cdots \; \tilde{\phi}_{n}^\circ] 
\in \{\tilde{\phi}\}$.  
%\item[3:]
%%and 
%The resolution of quantizer is sufficiently high.
%%$f(\tilde{\phi})$ is symmetric about each $\tilde{\phi}_i$-axis
\end{enumerate}
%where 
%$f(\tilde{\phi}_1)$ is the marginal density function on the space of
% $\tilde{\phi}_1$ defined by \rref{eq:marginal}.  
\end{assumption}
These conditions are not strong in usual setting of system
identification.  
In particular, the essence of \rref{assumption:high} is for guaranteeing the continuity
of $f(\phi)$ 
and it is usually satisfied; 
e.g., \rref{assumption:high} is satisfied when $f(\phi)$ is a
multidimensional normal distribution.  
This technical condition is used in the proof of Lemma~\ref{lemma:seisai}.  

The first Assumption~\ref{assumption:0}.1 gives 
the convergence of $\frac{1}{N}U^{\rm T}U$ or $\frac{1}{N}\tilde{U}^{\rm
T}\tilde{U}$ to $\sigma_u^2 I$, where $\sigma_u^2$ is a covariance of
$u$, at $N \rightarrow \infty$, 
and therefore,
\begin{equation}
N
\|\Delta {E}\|_2^2
\left(=N\|\Delta \tilde{E}\|_2^2\right) 
\mathop{\rightarrow}_{N\rightarrow \infty}
\trace 
\left[
\mathop{\plim}_{N\rightarrow \infty}
\left(
\frac{1}{N^2}U^{\rm T}UU^{\rm T}U
\right)^{-1}
\mathop{\plim}_{N\rightarrow \infty}
\left(
\frac{1}{N}U^{\rm T}EE^{\rm T} U
\right)
\right]
=
\frac{1}{\sigma_u^{4}}
\mathop{\plim}_{N\rightarrow \infty}
\left[
\frac{1}{N}
E^{\rm T}UU^{\rm T}E
\right]
\end{equation}
by Slutsky's theorem (see Appendix~\ref{appendex:proofs}).  
Moreover, we get:
\begin{equation}
\mathop{\plim}_{N\rightarrow \infty}
\left[
\frac{1}{N}
E^{\rm T}UU^{\rm T}E
\right]
=
\frac{1}{N}
\V
\left[
U^{\rm T}E
\right]
%
%\|{U}^{\rm T}{E}\|_2^2
\left(=
\frac{1}{N}
\V
\left[
\tilde{U}^{\rm T}{E}
\right]
\right),
\end{equation}
therefore, 
\begin{equation}
\|\Delta {E}\|_2^2
\sim
\frac{1}{\sigma_u^{4}N^2}
\V[U^{\rm T}E]
\end{equation}
at enough large $N$.  
%then $\V[U^{\rm T}E]=\V[\tilde{U}^{\rm T}E]$ is reasonable as the cost function. 
%
%When Assumption~\ref{assumption:dist} is satisfied, 
%$\frac{1}{N}U^{\rm T}U$ or $\frac{1}{N}\tilde{U}^{\rm T}\tilde{U}$ converges to
%$\sigma_u^2 I$ when $N \rightarrow \infty$ where $\sigma_u^2$ is a covariance of $u$, 
Then, 
it is reasonable to find an optimal quantizer that: 
\begin{itemize}
\item[1)] 
%1:
minimizes 
$\V
\left[
U^{\rm T}E
\right]
%$
%$( 
\left(
= \V
\left[
\tilde{U}^{\rm T}E
\right]
\right)
%)
$
%or
%$\V
%\left[
%(\tilde{U}^{\rm T}E)_1
%\right]
%$
%(i.e., the square of the first element, and this corresponds to the error in $\th1$) 
\item[2)]
%2: 
subject to constraints on the resolution of the quantizer, 
free of bias from the quantization error term,
%[OLE7]
 such as:
$\E
\left[
U^{\rm T}E
\right] = 0
$
$
\left(
\mbox{equivalently }\;
\E
\left[
\tilde{U}^{\rm T}E
\right] = 0
\right)
$.
%and so on.  
\end{itemize}
%Note that the latter is for the bias-free of the parameter estimation error.
%
The minimization of 
$\V
\left[
U^{\rm T}E
\right]
$
in arbitrary resolution cases of the quantizer 
%, 
%i.e., arbitrary coarse quantization, 
%the general resolution case for quantizers and
%input with any probability density function) 
is too complex to expect 
%the derivation of 
meaningful results, 
however, it is possible to derive the analytic solution 
%
%Therefore, we appropriately classify and formulate more tractable problems.  
%
%With this in mind, we evaluate cost functions: 
%$\V
%\left[
%(\tilde{U}^{\rm T}E)_1
%\right]
%$
%in {\bf general resolution} for quantizers in this section 
%and 
%$\V
%\left[
%U^{\rm T}E
%\right]
%$
in {\bf high resolution} as shown in the following of this section.  
%for quantizers
%%, however, 
%under weak assumptions in this section.
%Section~\ref{section:main}.  
%aaaaaaaaa
%
%With the definitions of $E$ and $U$ 
%given in \rref{eq:defofU}, \rref{eq:defofe}, and \rref{eq:defofE},

\begin{note}
\label{note:multi}
\rm

The multidimensional optimal quantization problem has been investigated 
(e.g., see \cite{Gray:TIT98,Graf:book,Scharf:book,Gersho:book}) and 
the research focus is on the derivation of analytic solutions. 
In the general resolution case, it is known to be a difficult problem and 
limited cases have been solved.  
One of these is the case of one-dimensional quantization and another is the
asymptotic case when the resolution of quantizers is sufficiently high.  
Note that cost functions are $\E[\|X - q(X)\|^r]$ in these studies.
However, we consider the cost function $\E[\|U^{\rm T}E\|_2^2]$ 
in this paper, which originates in system identification parameter estimation.  
The evaluation of the latter is much more complicated because it
contains many correlations of variables and resolving this difficulty is
one of main themes of this paper
(Note that the latter is not simple weighted square-error distortion
because of the correlation between $\reg$ and $e=\reg \th1 - q(\reg\th1)$).
The key lemmas (Lemma~\ref{lemma:seisai} and \ref{lemma:vari0}) show
that this quantity can be represented as a functional of one-dimensional
functions with one-dimensional quantization rules under appropriate
assumptions and, by using them, we
can find the optimal quantizers. 
\hfill{$\diamondsuit$}
\end{note}

On the above minimization problem, 
the bias-free condition 
$\E
\left[
U^{\rm T}E
\right] = 0
$
is equivalent to 
$\E
\left[
\tilde{U}^{\rm T}E
\right] = 0
$
from the relation $\tilde{U}^{\rm T}E = T^{\rm T}U^{\rm T}E$, 
where $T$ is nonsingular and orthogonal.  
From \rref{eq:deltaetilde}, this condition is equivalent to 
% and they can be represented as
%is directly reduced to 
%\rref{eq:biaszero} and \rref{eq:bias2}.
%\[
%\E
%\left[
%\sum_{\i = 1}^N
%{\phi}_{k}(\i)
%{e}(\i)
%\right]
% = 0, \; k = 1, 2, \dots n,
%\]
%and its equivalent condition can be written 
%with $\tilde{\phi}_k$ and the marginal density functions as:
\begin{equation}
\E
\left[
\sum_{\i = 1}^N
\tilde{\phi}_{k}(\i)
{e}(\i)
\right]
 = 
N
\cdot
\E
\left[
\tilde{\phi}_{k}
\cdot
{e}(\reg)
\right]
 = 
N
\int
\tilde{\phi}_{k}
{e}(\reg)
f(\tilde{\phi}_1, \tilde{\phi}_{k})
d\tilde{\phi}_1
d\tilde{\phi}_{k}
 = 0
\label{eq:biaszero}
\end{equation}
for $k  = 2, 3, \dots , n$
and 
\begin{equation}
\E
\left[
\sum_{\i = 1}^N
\tilde{\phi}_{1}(\i)
{e}(\i)
\right]
 = 
N
\cdot
\E
\left[
\tilde{\phi}_{1}
\cdot
{e}(\reg)
\right]
 = 
N
\int
\tilde{\phi}_{1}
{e}(\reg)
f(\tilde{\phi}_1)
d\tilde{\phi}_1
 = 0
\label{eq:bias2}
\end{equation}
for $k = 1$. 
% because 
%$\tilde{U}^{\rm T}E = T^{\rm T}U^{\rm T}E$, 
%where $T$ is nonsingular and orthogonal.
Note that we use the notation $e(\reg(t))$ when 
we intend to specify that $e(t)$ is a function of $\tilde{\phi}_1(t)$,
which can be seen from \rref{eq:deltaetilde}. 
%\rref{eq:tildephi},
The notation $f(\reg)$ represents a marginal density function:
% defined as
\begin{equation}
f(\reg) : =  \int f(
%\left[
%\matrix{
\reg,
% \;\;
\tilde{\phi}_2,
% \;\;
\dots ,
% \;\;
\tilde{\phi}_n
%}
%\right]
) 
d\tilde{\phi}_2
\cdots
d\tilde{\phi}_n.  
\label{eq:marginal}
\end{equation}
The notations $f(\tilde{\phi}_i, \tilde{\phi}_j)$, 
$f(\tilde{\phi}_i, \tilde{\phi}_j, \tilde{\phi}_k)$, $\dots$ are similarly
defined.

With the continuity condition of $f(\phi)$ in  
%only the continuity of $f(\phi)$ in 
Assumptions~\ref{assumption:0}.2,
% and \ref{assumption:0}.3,
% and \ref{assumption:0}.3, 
\rref{eq:biaszero} and \rref{eq:bias2}, i.e., the bias-free condition $\E[U^{\rm T}E] = 0$
$\left(\E[\tilde{U}^{\rm T}E] = 0\right)$, 
are asymptotically satisfied as the widths of the quantization steps
tend to 0 with the setting of ${y}_{\langle{j}\rangle}'$ at the center
of the quantization subsections.   
On the other hand, for the cost function
$\V[U^{\rm T}E] \left(= \V[\tilde{U}^{\rm T}E]\right)$, which can be
represented by 
\begin{eqnarray}
\V[U^{\rm T}E]
\left(
= \V[\tilde{U}^{\rm T}E]
\right)
& = &
\sum_{k = 1}^{n}
\E
\left[
\left(
\sum_{\i = 1}^N
\tilde{\phi}_{k}(\i)
{e}(\i)
\right)^2
\right]
 = 
\sum_{k = 1}^n
\E
\left[
\left(
\sum_{\i = 1}^N
\tilde{\phi}_{k}(\i)
{e}(\reg(\i))
\right)^2
\right],
\label{eq:star}
\end{eqnarray}
we derive the following key lemma.  
\begin{lemma}
\label{lemma:seisai}
Assume that $f(\tilde{\phi})$ satisfies 
\rref{assumption:high} in Assumption~\ref{assumption:0}.2. 
%\rref{assumption:high},
Then, 
\begin{equation}
\E\left[
\left(
%\frac{1}{N}
\sum_{\i = 1}^N 
\tilde{\phi}_k(\i)e(\reg(\i))
\right)^2
\right]
\;
\mathop{\longrightarrow}_{\Delta y_{\rm max} \rightarrow 0}
\;
%\frac{1}{N}
N
\E\left[
\tilde{\phi}_k^2e^2(\reg)
%\tilde{\phi}_k^2(i)e^2(i)
\right],
\label{eq:conv10}
\end{equation}
where $\Delta y_{\rm max}$ is the maximum width of the subsections $\S_j$
of the quantizer defined 
by $\Delta y_{\rm max}: =  \max_j |d_{j+1}-d_j|$.   
\end{lemma}
The proof of this lemma is given in Appendix~\ref{appendex:proofs}.  

From this lemma, the cost
%minimized 
function $\V[U^{\rm T}E] \left(= \V[\tilde{U}^{\rm T}E]\right)$ can be approximated by:
\begin{eqnarray}
&&
\V\left[
U^{\rm T}E
\right]
\left(
= \V\left[
\tilde{U}^{\rm T}E
\right]
\right)
%\;
\mathop{\longrightarrow}_{\Delta y_{\rm max}\rightarrow 0}
%\;
N
\sum_{k = 1}^n 
\E[\tilde{\phi}_k^2 e^2(\reg)]
 = 
N
\sum_{k = 1}^n \int \tilde{\phi}_k^2 e^2(\reg) f(\reg, \tilde{\phi}_2, \dots , \tilde{\phi}_n) 
d\reg d\tilde{\phi}_2 \cdots d\tilde{\phi}_n
\nonumber
\\
&&
 = 
N
\int
\left(
\int \sum_{k = 1}^n \tilde{\phi}_k^2 f(\reg, \tilde{\phi}_2, \dots , \tilde{\phi}_n)
d\tilde{\phi}_2 \cdots d\tilde{\phi}_n
\right)
e^2(\reg) d\reg.
% = N\int \sigma^2(\reg) e^2(\reg) f(\reg) d\reg.
\label{eq:vari1001}
\end{eqnarray} 
in the high resolution case.  
Therefore, the focus of the problem is on the calculation of 
the r.h.s. of \rref{eq:vari1001} for general $f(\phi)$ and its minimization.  
A key concept in solving this problem is the introduction of the following
quantity in the distribution of quantization subsections, 
which is a reasonable concept 
in the high resolution case.
%from Assumption~\ref{assumption:0}-2).  
\begin{definition}
The quantity $g(\tilde{\phi}_1)$, which satisfies
\begin{equation}
%\tilde{\theta}_1
g(\tilde{\phi}_1) d\tilde{\phi}_1 
 = 
\mbox{\rm number of quantized subsections in } d\reg,
%d(\tilde{\theta}_1\tilde{\phi}_1)
%\,( = \tilde{\theta}_1d\tilde{\phi}_1)
\label{eq:defofg}
\end{equation}
is called the density of the number of quantized subsections. 
\end{definition}
This quantity is the same as that introduced in
\cite{Bennett:BSTJ48,Lloyd:TIT82} 
and from this definition, $g(\reg)^{-1}$ represents the width
of the quantization step at $\reg$. 
%$\th1\reg$. 

%In Section~\ref{sec:prev}, 
%we deal with the case that the quantization resolution is not necessarily high and 
%the quantized ${y}_\j'$ for each subsection $\S_j$ is strictly assigned 
%%to satisfy 
%for the bias-free condition.   
%%that the expectation of the quantization error is zero in each subsection.   
%However, 
%it 
%%the bias-free condition 
%is asymptotically satisfied in the high resolution case 
%when ${y}_\j'$ is just selected at the midpoint of $\S_j$. 
%%assigned in $\S_j$ and 
%%the midpoint of each subsection is reasonable and simply assigned as the
%%quantized value.  
%Therefore, we fix such quantized values in what follows.
 
We also assume a form of smoothness of $f(\phi)$ and $g(\reg)$ in the following. 
\begin{assumption}
\label{assumption:smooth}
The density function $f(\phi)$ and $g(\reg)$ satisfy the following
conditions:
\begin{enumerate}
\item[1:]
%\item[\ref{section:main}-2)]
$f(\phi)$ is a continuous function s.t.
%\item[\ref{section:main}-5)]
\begin{eqnarray}
&&
\frac{d(\sigma^2(\reg)f(\reg))}{d\reg} < \infty,
\label{eq:assump42}
\\
&&
\sigma(\tilde{\phi}_1)
: = 
\left(
f(\tilde{\phi}_1)^{-1}
\int
\left(
\sum_{k = 1}^{n}
\tilde{\phi}_k^2
\right)
f(\tilde{\phi}_1, \dots , \tilde{\phi}_n)
d\tilde{\phi}_2
\cdots
d\tilde{\phi}_n
\right)^{\frac{1}{2}},
\label{eq:variance101}
\end{eqnarray}
where 
$f(\tilde{\phi}_1)$ is the marginal density function on the space of
 $\tilde{\phi}_1$ defined by \rref{eq:marginal}.  
\item[2:]
the resolution of quantizer is sufficiently high
and the density $g(\reg)$ satisfies:
\[
%\begin{eqnarray}
\frac{d g(\reg)^{-2}}{d\reg} < \infty.
%\end{eqnarray}
\]
\end{enumerate}
%\end{assumption}
\end{assumption}

\begin{note}
\rm
The essence of Assumption~\ref{assumption:smooth} is the smoothness of
$f(\tilde{\phi}_1)$ and $g(\tilde{\phi}_1)$ such as they guarantee the
approximation of \rref{eq:vari1001} in the following.  
Assumption~\ref{assumption:smooth}.1 describes a form of the continuity of
$f(\phi)$ or $f(\reg)$ and it is not a strong assumption 
in the usual situation of system identification;
e.g., $f(\phi)$ or $f(\tilde{\phi})$ in $C^1$ is enough and it is
satisfied when they are multidimensional normal distributions.  
Assumption~\ref{assumption:smooth}.2 also describes a form of the
continuity of the quantizer and $g(\tilde{\phi}_1)$ or $g(y) \in C^2$ is enough. 
Such technical conditions come from our intention to make the necessary
conditions for deriving \rref{eq:vari1001app} weak as possible.  
\hfill{$\diamondsuit$}
\end{note}

With Assumption~\ref{assumption:smooth}.2,
%and
%Assumption~
%\ref{assumption:smooth}.2, 
%and the assumption of $f(\reg)$, which is given by Assumption~\ref{assumption:0}-5), 
we can select a value 
$g_j^{-1} \sim g(\reg)^{-1}$ for the subsection $\I_j$ that satisfies
$g_j^{-1} = |\I_j|$. 
%Then, we
%define $f_j$ $\sim f(\reg)$ in $\reg \in \I_j$ 
%that satisfies: 
%\begin{equation}
%%p_j : =  
%\int_{\I_j} f(\reg) d\reg  = : f_ig_i^{-1}.
%\label{eq:fg}
%\end{equation}
Moreover, 
with $\sigma(\reg)$ of $f(\tilde{\phi})$ at $\reg$ defined in
\rref{eq:variance101}, 
Assumption~\ref{assumption:smooth}.1--2, 
%\ref{assumption:smooth}.2, 
and $\Delta \tilde{\phi}: = \max_j \th1^{-1}|d_{j+1} - d_j|$,
for the objective function \rref{eq:vari1001}, 
we calculate the following directly: 
\begin{eqnarray}
%&&
\rref{eq:vari1001}/N
=
\int
\sigma^2(\reg)
e^2(\reg)
f(\reg)
d\reg
 = 
\th1^2
\int
\frac{1}{12}g(\tilde{\phi}_1)^{-2}
\sigma^2(\tilde{\phi}_{1}) f(\tilde{\phi}_1) d\tilde{\phi}_1
+
O(\Delta \tilde{\phi}).
\label{eq:vari1001app}
\end{eqnarray}
See Appendix~\ref{appendex:proofs}
for the derivation of \rref{eq:vari1001app}.
From this, 
\begin{equation}
\th1^2
\int
\frac{1}{12}g(\reg)^{-2}
\sigma^2(\reg) f(\reg) d\reg
\label{eq:objectivef}
\end{equation}
is considered to be a reasonable cost function when
Assumption~\ref{assumption:0}
%.2, 
%Assumption~
and \ref{assumption:smooth}
%.1--2
%, and \ref{assumption:smooth}.2 
%\ref{assumption:0}-5), and Assumption~\ref{assumption:smooth}
are satisfied.   

In the following 
we assume 
Assumption~\ref{assumption:0} and Assumption~\ref{assumption:smooth} and
give the optimal quantizers, 
which minimize \rref{eq:objectivef}, 
subject to a constraint on
the number of quantization steps (Section~\ref{sec:fixedrate}) or on the
expectation of the code length, where the quantized data is optimally
encoded (Section~\ref{sec:variablerate}). 
The former case is referred to as ``fixed-rate quantization'' 
because it is identical to a ``fixed-code length'' case; 
the latter case is referred to as ``variable-rate quantization'' and
the code length is not fixed.

\subsection{Fixed-rate quantization}\label{sec:fixedrate}

From the previous derivation, 
% above, 
the original optimization problem of \rref{eq:vari1001} 
can be replaced by the minimization of \rref{eq:objectivef}
in $N \rightarrow \infty$ and the high
resolution case:
\begin{problem}
Find
\begin{eqnarray}
&&
\hspace{-1cm}
g_{\rm f}(\reg)
: =  \arg \min_g \int {\cal F}(g(\reg)
)
d\tilde{\phi}_1 
\label{eq:opt-n}
\\
\mbox{s.t.}&&
\int
%_{\reg^{\min}}^{\reg^{\max}} 
g(\reg) d\reg = M,
\label{eq:opt-n-const}
\end{eqnarray}
where
\begin{eqnarray}
{\cal F}(g(\tilde{\phi}_1)
) 
&: = &
\frac{1}{12}\th1^2
{g(\tilde{\phi}_1)}^{-2}
\sigma^2(\tilde{\phi}_1)
f(\tilde{\phi}_1).
\label{eq:defofF}
\end{eqnarray}
\end{problem}

The following theorem gives the solution of this problem:
% is derived from these preliminaries.
\begin{theorem}
\label{theorem:fixed-rate}
The solution of \rref{eq:opt-n} is: 
\begin{eqnarray}
g_{\rm f}(\tilde{\phi}_1)
& = & 
K \sigma^{\frac{2}{3}}(\tilde{\phi}_1)
f^{\frac{1}{3}}(\tilde{\phi}_1)
\label{eq:opthigh}
\\
K & = & D^{-1}M\\
D & = & \int \sigma^{\frac{2}{3}}(\reg)f^{\frac{1}{3}}(\reg) d\reg.
\end{eqnarray}
Moreover, the optimized value is given by:
\begin{equation}
\int {\cal F}(g_{\rm f}(\reg)
%G_{\rm f}(\reg)
) d\reg
 =  \frac{1}{12}\th1^2 D^3 M^{-2}.
\label{eq:minv}
\end{equation}
\end{theorem}
The minimization problem can be rigorously solved by applying the
calculus of variations.  
%The proof is in Appendix~\ref{appendex:proofs}.
See Appendix~\ref{appendex:proofs} for the proof. 

%\begin{note}
%\rm
%The corresponding optimal density $\goptf(y)$ of the number of quantized
%subsections on $y$ (denote as $g(y)$) can be also 
%derived by:
%\begin{eqnarray*}
%\goptf(y)
%& = & 
%K \sigma^{\frac{2}{3}}(y)
%f^{\frac{1}{3}}(y)
%, \;\;
%{\cal F}(g(y)
%) 
% = 
%\frac{1}{12}
%{g(y)}^{-2}
%\sigma^2(y)
%f(y),
%\\
%K & = & D^{-1}M, \;\;
%D
% = 
%\int \sigma^{\frac{2}{3}}(y)f^{\frac{1}{3}}(y) dy
%, \;\;
%\int {\cal F}(g_{\rm f}(y), \; 
%G_{\rm f}(y)) dy
% = 
%\frac{1}{12} D^3 M^{-2},
%\end{eqnarray*}
%where $f(y)$ or $\sigma(y)$ is also defined similarly to those on $\reg$. 
%% 
%\end{note}

From Theorem~\ref{theorem:fixed-rate}, the asymptotic optimal
quantization at high resolution is readily calculated analytically, or
numerically, if the marginal density functions $f(\reg)$ are known.  

\begin{note}
\rm
The optimal quantization scheme on $y$ (call it as $g_{\rm f}(y)$) 
is also given by using the above results.  
With the relation $y=\reg\th1$ and the fact that the optimal
$g_{\rm f}(\reg)$ is given only by $f(\reg)$, $g_{\rm f}(y)$ 
on $y$ is a simple scaling of $g_{\rm f}(\reg)$.  
Therefore, $g_{\rm f}(y)$ on $y$ is given by; 
(i) using the knowledge of $\th1$ and $g_{\rm f}(\reg)$, 
or (ii) $f(y)$ on $y$ such as 
$g_{\rm f}(y) = K' \sigma^{\frac{2}{3}}(y)f^{\frac{1}{3}}(y)$, 
where $f(y)$ is obtained by the observation of the output data $\{y(t)\}$.  
The situation (i) is a standard problem setting of control systems
under limitation of channel capacity, where the quantizer (encoder) is
supposed that it can fully utilize information on systems in order to
optimally compress the data.  The situation (ii) is also a natural
problem setting.  
\end{note}

\begin{example}
%\begin{note}
\label{note:number1}\rm
When $f(\tilde{\phi})$ is a multidimensional normal distribution:
\[
f(\reg, \tilde{\phi}_2, \dots , \tilde{\phi}_n)
 = 
\frac{1}{(2\pi)^{\frac{n}{2}}(\det \Gamma)^{\frac{1}{2}}}
\exp\left(
-\frac{1}{2}
\tilde{\phi}^{\rm T}
\Gamma^{-1}
\tilde{\phi}
\right), 
\;\;
\Gamma = \diag (\sigma_o, \sigma_o, \dots , \sigma_o),
\]
where $\Gamma$ is a covariance matrix of $\tilde{\phi}$, 
then 
\[
\sigma^2(\reg) = \reg^2 + (n-1) \sigma_o^2.
\]
For simplicity, 
in the case that the order $n$ of the FIR model is sufficiently large, 
%then: 
\begin{equation*}
 \sigma^2(\reg) f(\reg) \sim n\, \sigma_o^2 f(\reg).
\end{equation*}
Therefore:
\begin{eqnarray}
&&
D \sim 
n^{\frac{1}{3}}\sigma_o^{\frac{2}{3}}
\int f^{\frac{1}{3}}(\reg) d\reg,
\;\;
g_{\rm f}(\reg) \sim M \left(\int f^{\frac{1}{3}}(\reg) d\reg\right)^{-1} f^{\frac{1}{3}}(\reg),
\nonumber
\\
&&
\int {\cal F}(g_{\rm f}(\reg)
) d\reg
\sim
\frac{1}{12}\th1^2 
\left(\int f^{\frac{1}{3}}(\reg) d\reg\right)^{3}
n \sigma_o^2 
M^{-2}
 = 
\frac{1}{12}\th1^2 
6\sqrt{3} \pi
n \sigma_o^4 
M^{-2}
\sim
0.8658
\pi
\th1^2 
n \sigma_o^4 
M^{-2}.
\label{eq:gosa100}
\end{eqnarray}
%\end{note}
\hfill{$\diamondsuit$}
\end{example}
\begin{example}
%\begin{note}
\label{note:origin}
\rm
Here we consider another simple case $n=1$, where the 
cost function becomes
%modify the cost function \rref{eq:vari1001} 
\begin{eqnarray*}
\V\left[
U^{\rm T}E
%\reg^2\cdot e^2
\right]
 = 
N
\int
\reg^2
e^2(\reg)
f(\reg)
d\reg.
\end{eqnarray*}
%for comparison with the results of Section~\ref{sec:prev}.   
Then, the optimal quantization $\goptf^\star(\reg)$ 
for this is given by 
\begin{eqnarray*}
g_{\rm f}^\star(\reg)
 = 
%\sim 
K \reg^{\frac{2}{3}}
f^{\frac{1}{3}}(\reg)
, \;\;
%\\
K
 = 
%\sim
%&\sim&
D^{-1}M
, \;\;
%\\
D
 = 
%\sim
%&\sim&
\int
%_{{\cal R}\backslash [-\epsilon, \, \epsilon]} 
\reg^{\frac{2}{3}}f^{\frac{1}{3}}(\reg) d\reg.
\end{eqnarray*}
\hfill{$\diamondsuit$}
\end{example}

%\tthrule

We illustrate $\goptf^\star(\reg)$ for the cases where 
$\sigma^2(\reg) = \reg^2 + \sigma_o^2$ and 
$f(\reg)$ is the uniform distribution,
normal distribution, or power law as follows.  

%------------------------ºï½ü¤¹¤ë
%
%In Section~\ref{sec:prev}, 
%we derived the strictly optimal quantization for general 
%resolution when $f(\reg)$ is the uniform distribution.  
%Lemma~\ref{lemma:rate2} shows that the optimal quantization is coarse near the origin of 
%$\reg$ and dense near the boundary of the domain of $\reg$. 
%This property of the optimal quantization is also seen in 
%Theorem~\ref{theorem:fixed-rate} (see Fig.~\ref{fig:uniform}).  
%
%------------------------

Fig.~\ref{fig:uniform} is the case that $f(\reg)$ is the uniform distribution.  
%a simple case $\sigma(\reg) = \reg$, 
From the figure, we observe 
that the optimal quantization is coarse near the origin of 
$\reg$ and dense near the boundary of the domain of $\reg$. 
Theorem~\ref{theorem:fixed-rate} shows that the increasing rate of resolution with enough
large $\reg$ is about $\tilde{\phi}_1^{\frac{2}{3}}$.
% and this means that the optimal
%quantization is coarse near the origin of $\reg$ and dense near the
%boundary of the domain of $\reg$. 

%This property of the optimal quantization is also seen in 
%Theorem~\ref{theorem:fixed-rate} (see Fig.~\ref{fig:uniform}).  
%, 
%which is unclear from the results of the previous section. 

When $f(\reg)$ is the normal distribution, 
the profile of the density $f(\reg)$ near the origin is flat; therefore,     
the optimal quantizer must have a similar profile to that where
$\reg$ is the uniform distribution near the origin.  
%That is, the resolution increases near the origin
%; 
We can see such a profile of $\goptf^\star(\reg)$ in Fig.~\ref{fig:normal}.  
This property is, in some sense, the dual result to that of the quantization
problem for stabilization by \cite{Elia:IEEE01}; that is, the coarsest
quantization scheme for stabilization is dense near the origin and
becomes coarser as distance from the origin increases.
These observations suggest that there appears to exist a trade-off between
parameter estimation and stabilization in the quantization scheme for
a type of adaptive control system.
On the other hand, in the area of the tail of $f(\reg)$, $\goptf^\star(\reg)$
decreases.  However, contrary to our intuition, the resolution remains high,
e.g.,
$\goptf^\star(3)$ $\sim$ $0.208$ $\sim$ $45\%$ of $\max \goptf^\star(\reg)$ 
%$\goptf^\star(3)$ $\sim$ $0.201$ $\sim$ $51\%$ of $\max \goptf^\star(\reg)$ 
or 
$\goptf^\star(4)$ $\sim$ $0.0774$ $\sim$ $17\%$ of $\max \goptf^\star(\reg)$, 
%$\goptf^\star(4)$ $\sim$ $0.0758$ $\sim$ $19\%$ of $\max \goptf^\star(\reg)$, 
where $f(\reg)$ is sufficiently small.    

Finally $f(\reg)\sim \reg^{-2}$ at the tail of the
distribution is an example of a power law. 
In this case, $\goptf^\star$ is constant in the tail and it is marginal for 
the solution's existence (see Fig.~\ref{fig:power}). 
This result shows the difficulty of system identification at
sufficient accuracy by using finite information from the system when the tail
of the probability density function $f(\reg)$ is heavier than $O(\reg^{-2})$.  
That is, this 
explains the complexity of the power law from the viewpoint of
parameter estimation in system identification.

\noindent
\hfill
\begin{minipage}[c]{\minipagewidth}
%\vspace{.3cm}
\vspace{.5cm}
%\vspace{1.2cm}
\centerline{
\epsfxsize = \figwidth \epsfbox{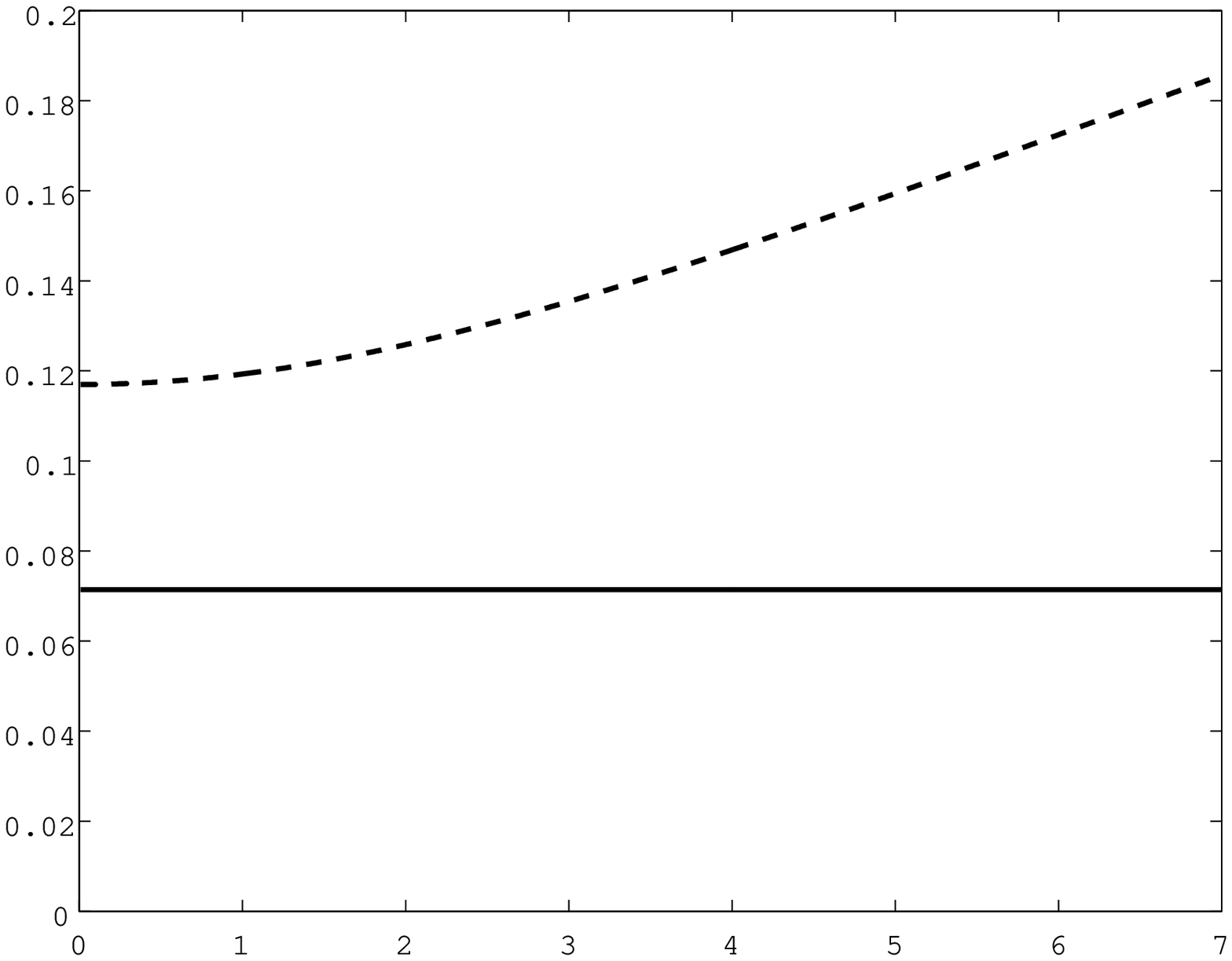}{
\tput(50,-10,{$\reg$})
\tput(55,28,{{$f(\reg)$}})
%\tput(55,31,{{$f(\reg)$}})
\tput(40,76,{{$\goptf^\star(\reg)$}})
%\tput(40,68,{{$\goptf^\star(\reg)$}})
}
\refstepcounter{figure}\label{fig:uniform}
}
\vspace{.6cm}
{Fig.~\ref{fig:uniform}: 
Probability density
 $f(\reg)$ of the regressor (solid line) 
in uniform distribution and 
the density function of the number of the optimally quantized subsections
$\goptf^\star(\reg)$ (dashed line) when $\sigma^2(\reg) = \reg^2 + \sigma_o^2$
%$\goptf^\star(\reg)$ (dashed line) when $\sigma(\reg) = \reg$
}
\vspace{.3cm}
\end{minipage}
\hfill
\begin{minipage}[c]{\minipagewidth}
\vspace{.5cm}
\centerline{
\epsfxsize = \figwidth \epsfbox{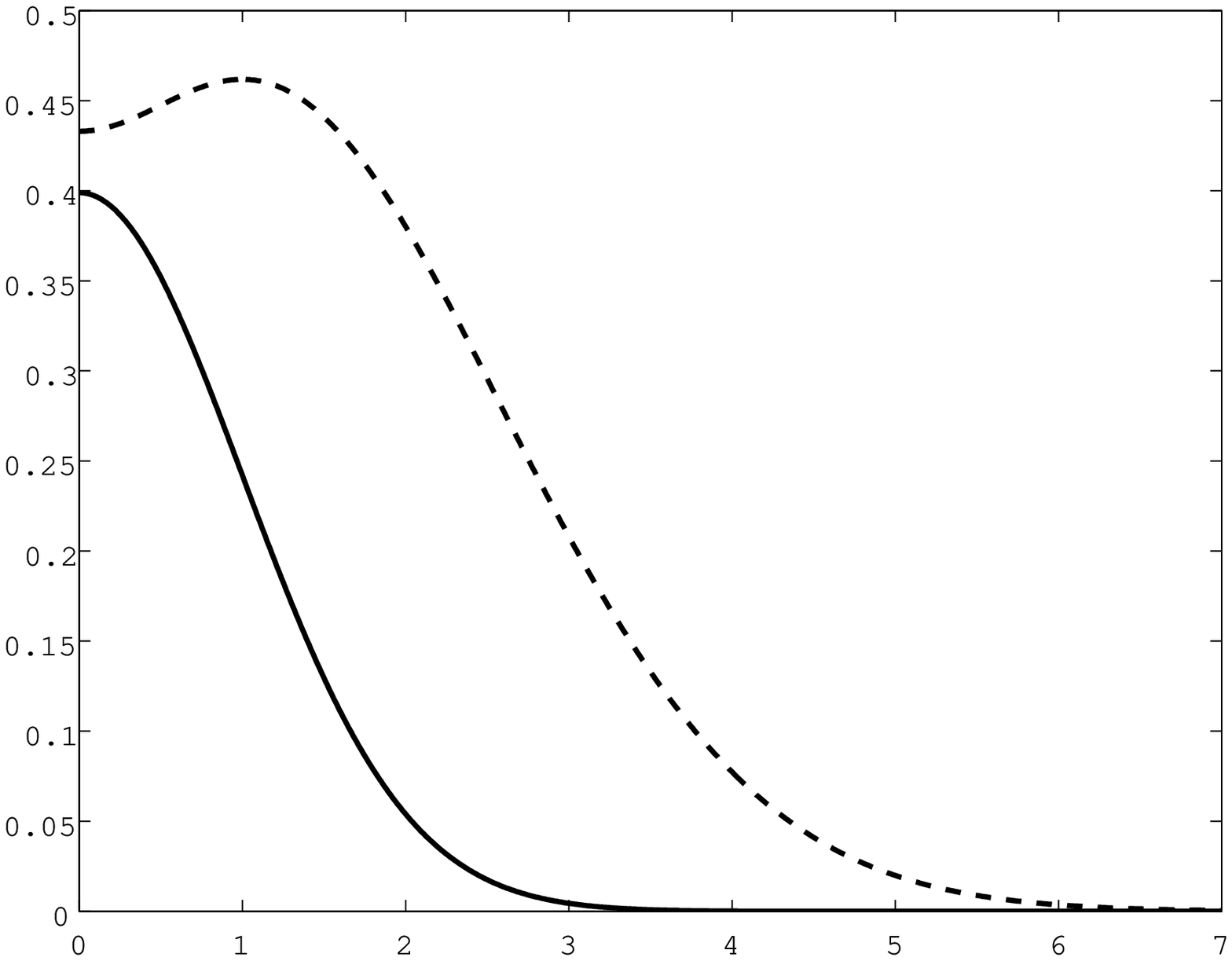}{
\tput(50,-10,{$\reg$})
\tput(29,30,{{$f(\reg)$}})
\tput(52,37,{{$\goptf^\star(\reg)$}})
%\tput(30,30,{{$f(\reg)$}})
%\tput(55,37,{{$\goptf^\star(\reg)$}})
}
\refstepcounter{figure}\label{fig:normal}
}
\vspace{.6cm}
{Fig.~\ref{fig:normal}: 
Probability density $f(\reg)$ of the regressor (solid line) 
in normal distribution and 
the density function of the number of the optimally quantized subsections
$\goptf^\star(\reg)$ (dashed line) when $\sigma^2(\reg) = \reg^2 + \sigma_o^2$
%$\goptf^\star(\reg)$ (dashed line) when $\sigma(\reg) = \reg$
}
\vspace{.3cm}
\end{minipage}
\hfill\break

\noindent
\hfill
\begin{minipage}[c]{\minipagewidth}
\vspace{.5cm}
\centerline{
\epsfxsize = \figwidth \epsfbox{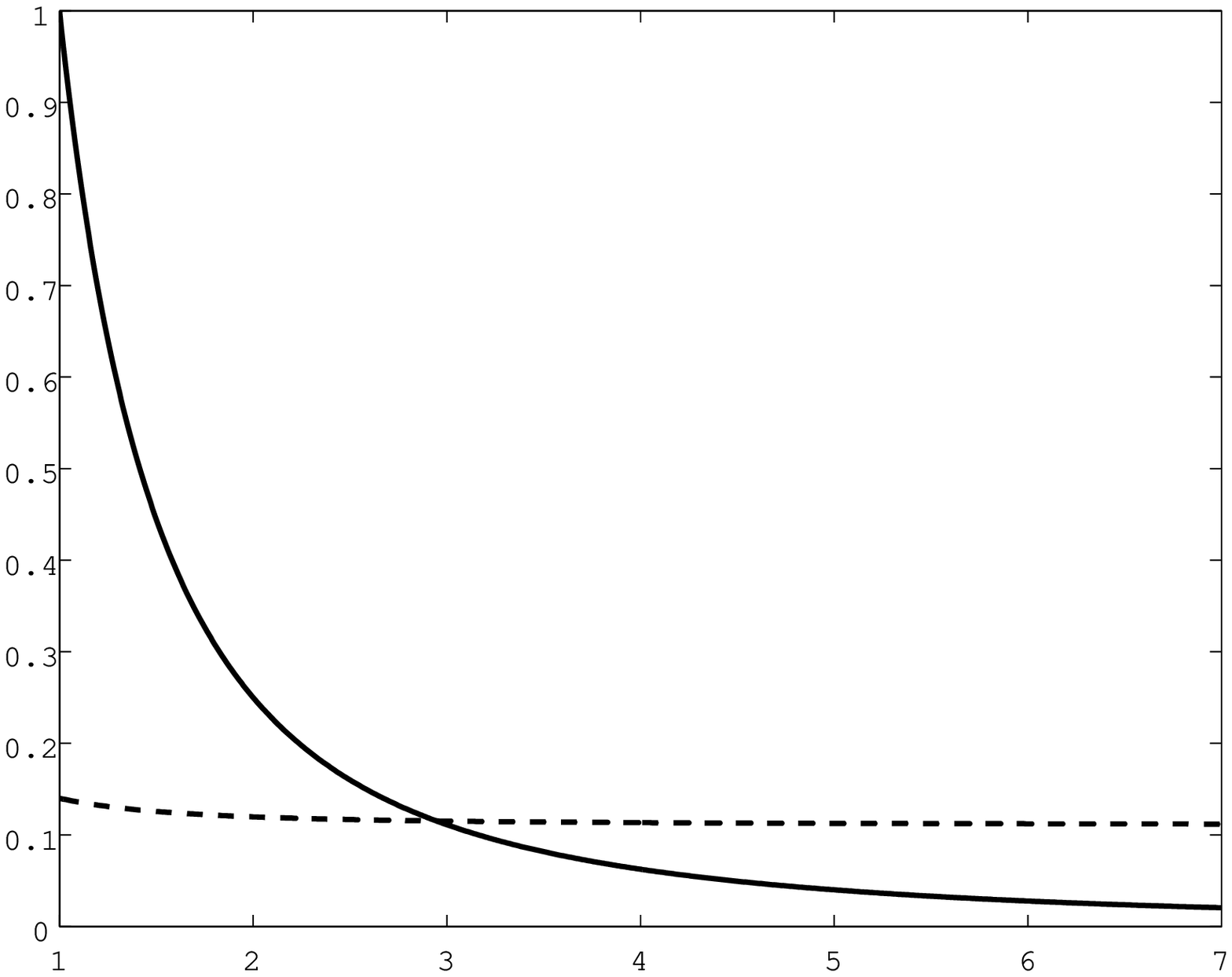}{
\tput(50,-10,{$\reg$})
\tput(20,35,{{$f(\reg)$}})
\tput(62,22,{{$\goptf^\star(\reg)$}})
}
\refstepcounter{figure}\label{fig:power}
}
\vspace{.8cm}
{Fig.~\ref{fig:power}: Power law ($O(\tilde{\phi}_1^{-2}$)) 
$f(\reg)$ of the regressor (solid line) and
the density function of the number of the optimally quantized subsections
$\goptf^\star(\reg)$ (dashed line) when $\sigma^2(\reg) = \reg^2 + \sigma_o^2$
%$\goptf^\star(\reg)$ (dashed line) when $\sigma(\reg) = \reg$
}
\vspace{.3cm}
\end{minipage}
\hfill\break

\begin{note}
\label{note:reason}
\rm
%Note that in this case
%$g_{\rm f}^\star(\reg)$ is defined only for the 
%region ${\cal R}\backslash [-\epsilon, \epsilon]$, $0 < \epsilon \ll 1$, 
%because $\goptf(\reg) = K \reg^{\frac{2}{3}}f^{\frac{1}{3}}(\reg)$ 
%is too small in $[-\epsilon, \epsilon]$ to apply the approximation of
%high resolution.  
As known from Fig.~\ref{fig:uniform} and Fig.~\ref{fig:normal}, 
%On the other hand, 
when $f(\tilde{\phi})$ is the
normal distribution, uniform distribution or other probable
distributions in usual situation of system identification,  
the marginal density $f(\reg)$ is approximately flat near the origin
and the quantization becomes coarse in such subsection.  
%$[-\epsilon, \epsilon]$.  
Therefore, in order to clarify the minute structure of the optimal
quantizer around the origin, 
we should consider the problem in the coarse resolution with a flat
marginal density $f(\reg)$.  
Such case is rigorously analyzed in Section~\ref{sec:prev}.  
%
%the optimal quantization in the region $[-\epsilon, \epsilon]$ 
%is similar to the solution derived in Section~\ref{sec:prev}.
%For this reason, the result in Section~\ref{sec:prev} can be used for 
%a complement. It is also indispensable for constructing the optimal
% quantization for \rref{eq:vari1001} in the high resolution case. 
%\end{note}
\hfill{$\diamondsuit$}
\end{note}

\subsection{Variable-rate quantization}\label{sec:variablerate}

The previous subsection presents the optimal quantizer to minimize
the identification error 
\rref{eq:vari1001} (i.e. \rref{eq:objectivef})
subject to a constraint on the number of quantization
steps, i.e., fixed-rate quantization, with high resolution.  
On the other hand, 
to reduce the
information in the observed data, it is 
reasonable to apply variable-rate coding for the quantized signals and evaluate the
mean code length from the information theoretic viewpoint.  
From this observation, 
we consider the minimization problem of
\rref{eq:vari1001} (i.e., \rref{eq:objectivef})
subject to a constraint of the expectation of the optimal
code length in this subsection, that is, variable-rate quantization, with high resolution.  

Let $C(\cdot)$ be an encoder that is a mapping from source alphabets to
code alphabets and $l(\cdot)$ be the code length.  
We regard the quantized output $q(\reg)$ as the corresponding source alphabets, then, 
$l(C(q(\reg)))$ represents the code length of $q(\reg)$. 
The expectation of the optimal variable-rate code length for a quantized
signal is related to the entropy of the source alphabets by the
following well-known source coding theorem.  
\begin{proposition}{\rm \cite{Shannon:bell48,Cover:book}}
Let $x$ be source alphabets, then: 
\begin{equation*}
\E\left[l(C(x))\right] \geq H(x),
\end{equation*}
where $H(x)$ represents the entropy of $x$. 
\end{proposition}
With this proposition, the optimization problem of the quantizer
for the code length is reduced to the minimization problem 
of \rref{eq:vari1001} (i.e., \rref{eq:objectivef})
subject to a constraint on the entropy of the quantized signals.  

The basic concept for representing the quantizer with high resolution
is the same as that of the previous subsection. 
That is, subject to 
Assumption~\ref{assumption:smooth}.1 and \ref{assumption:smooth}.2,
%Assumption~\ref{assumption:0}-5) and Assumption~\ref{assumption:smooth}, 
we obtain the asymptotic approximation of the entropy of
the quantized signal: 
\begin{eqnarray}
%H(f,g) &: = & 
\sum_j - p_j \log p_j 
& \sim & 
\sum_j - \int_{\I_j} f(\reg) d\reg \log f_jg_j^{-1} 
%\nonumber\\
%&\sim& 
\sim
\int - f(\reg) \log \left(f(\reg) g^{-1}(\reg) \right)
d\tilde{\phi}_1 
\nonumber
\\
& = &
H_{\rm d}(f) + \int - f(\reg) \log \left(g^{-1}(\reg) \right)
d\reg
%\nonumber
%\\
%&=:& 
=: H(f, g),
\label{eq:entropy}
\end{eqnarray}
where $H_{\rm d}(f): = \int -f(\reg) \log f(\reg) d\reg$.   
By using this asymptotic approximation of the entropy \rref{eq:entropy}, 
we consider the following problem.
\begin{problem}
Find
\begin{eqnarray}
\lefteqn{
g_{\rm v}(\tilde{\phi}_1)
: = 
\arg \min_g \int {\cal F}(g(\tilde{\phi}_1)
)
d\tilde{\phi}_1 
}
\label{eq:opt-e}
\\
&&
%\hspace{2cm}
\mbox{s.t. }
%\\&&
H(f,g) = \log M,
\hspace{3cm}
\label{eq:opt-e-const}
\end{eqnarray}
where ${\cal F}(\cdot)$ is defined in \rref{eq:defofF}. 
\end{problem}
Note that $M$ is the expected number of quantization steps in the sense
of \rref{eq:opt-e-const}. 
We can derive the following theorem:
\begin{theorem}
\label{theorem:variable}
The solution of \rref{eq:opt-e} is:
\begin{eqnarray}
g_{\rm v}(\tilde{\phi}_1)
& = & 
K M \sigma(\reg)
\label{eq:gentroopt}
\\
K & = & \exp L\\
L&: = & -H_{\rm d}(f)-\int f\log \sigma(\reg) d\reg 
 = 
\int f(\tilde{\phi}_1) \log\frac{f(\tilde{\phi}_1)}{\sigma(\reg)} d\tilde{\phi}_1. 
\end{eqnarray}
Moreover, the optimized value is: 
\begin{equation}
\int {\cal F}(g_{\rm v}(\tilde{\phi}_1)
) d\tilde{\phi}_1
 =  \frac{1}{12}\th1^2
K^{-2} M^{-2}.
\label{eq:mine}
\end{equation}
\end{theorem}
The proof is in Appendix~\ref{appendex:proofs}.

\begin{note}\rm
It is interesting that 
the optimal $\goptv$ is a simple linear function of $\sigma(\reg)$. 
The constant coefficient is also linear with respect to 
the number of expected quantization steps $M$.  
On the other hand, 
the convergence rate of the minimized cost function 
%variance of the quantization error term 
is $M^{-2}$; this is in common with the fixed-rate quantization.  
\hfill{$\diamondsuit$}
\end{note}
\begin{example}
\rm
When $f_{\tilde{\phi}}$ is the density function in a multidimensional normal distribution 
and $n$ is sufficiently large, as described in Example~\ref{note:number1},
%by using \rref{eq:sigmaapp}, 
\begin{eqnarray}
g_{\rm v}(\reg) & = & KM\sigma(\reg)
%\\
%&\sim& 
\sim
%M \cdot \exp(-H(f))\cdot (\sigma_o n^{\frac{1}{2}})^{-1}
%\cdot \sigma_o n^{\frac{1}{2}}
%%\\
%%& = &
% = 
M \cdot \exp(-H_{\rm d}(f))
\nonumber
\\
%\end{eqnarray*}
%\begin{eqnarray}
%\lefteqn{
\int {\cal F}(g_{\rm v}(\tilde{\phi}_1)
%, \; 
%G_{\rm v}(\tilde{\phi}_1)
) d\tilde{\phi}_1
%}
%\nonumber
%\\
&\sim&
\frac{1}{12}\th1^2
\exp(2H_{\rm d}(f)) n \sigma_o^2 M^{-2}
%\nonumber\\
%& = &
 = 
\frac{1}{12}\th1^2
2e \pi n \sigma_o^4 M^{-2}
\sim
0.4533
%5.44 
\pi \th1^2 n \sigma_o^4 M^{-2}.
\label{eq:gosa200}
\end{eqnarray}
By comparison with \rref{eq:gosa100} and \rref{eq:gosa200}, 
it can be seen 
%it is known 
that variable-rate optimal coding achieves approximately half the magnitude of the square of
 the quantization error 
compared with $\goptf$ for fixed-rate quantization. 
\hfill{$\diamondsuit$}
\end{example}

%#!platex paper

\section{
Quantization in Coarse Resolution 
}
\label{sec:prev}

%\subsection{Explicit optimal quantization scheme}

In the previous section, we give the optimal quantization in high
resolution for general probability densities of input signals.  
The results are enough for understanding the profile of the optimal
quantization, however, 
as explained in Note~\ref{note:reason}, 
its minute structure around the origin is not clear in the case of coarse quantization.  
In this section, we do not necessarily suppose high resolution of quantization
and derive the optimal quantization, however, under limited assumption as follows.   

%Note that the profile of the usual input signals' multidimensional
%probability densities in system identification, e.g., normal
%distribution, is high and flat near the origin and the above assumption is
%approximately satisfied in such a region.  Therefore, optimal quantizers
%for general distributions usually have the same property in such a
%region as the optimal quantizer given in the following.  
%
%First, we state the problem formulation considered in this section, 
%which originates in Problem~\ref{problem:original}.  
%It is assumed:  
\begin{assumption}
\label{assumption:dist}
$f(\phi)$ is a probability density function 
such that $f(\tilde{\phi})$ is uniform distribution in $\tilde{\phi}_i \in
[-\kappa, \, \kappa]$ with a given $\kappa$ $(\in {\cal R}) > 0$.
\end{assumption}

The optimization problem under this assumption has clear significance
for the following cases: 
(1) to clarify the minute quantization scheme around the origin of $y$
because the profile of the multidimensional probability densities 
of usual input signals in system identification, e.g., normal
distribution, is flat around the origin.  In such subsection, 
the quantization is comparatively coarse and the probability density can
be approximated as a uniform distribution.  
The important fact is that such property of the flatness of the
probability density around the origin does not depend on the choice of the
base in the space of $\phi$.  This means the condition of
Assumption~\ref{assumption:dist} is always satisfied around the origin 
in usual situation of system identification.  
(2) to consider the first order systems where input signals obey a
uniform distribution.  In this case, 
%This assumption considerably simplifies the cost
%function and 
%the problem is tractable and 
the analytic optimal solution in coarse quantization can be given and 
it is enough for the main subject of this paper to clarify the essential
properties of the optimal quantizers for parameter estimation.

When Assumption~\ref{assumption:dist} is satisfied, 
as similar to the case of Section~\ref{section:main},
$\frac{1}{N}U^{\rm T}U$ and $\frac{1}{N}\tilde{U}^{\rm T}\tilde{U}$ also converge to
$\sigma_u^2 I$ when $N \rightarrow \infty$, 
% where $\sigma_u^2$ is a covariance of $u$, 
then 
%it is reasonable to find 
the optimal quantization problem is also reduced to
% that: 
%\begin{itemize}
%\item[1)] 
%%1:
minimize 
$\V
\left[
U^{\rm T}E
\right]
%$
%$( 
\left(
= \V
\left[
\tilde{U}^{\rm T}E
\right]
\right)
%)
$ of \rref{eq:star}
%or
%$\V
%\left[
%(\tilde{U}^{\rm T}E)_1
%\right]
%$
%(i.e., the square of the first element, and this corresponds to the error in $\th1$) 
%\item[2)]
%2: 
subject to a bias free condition:
%constraints on the resolution of the quantizer, 
%free of bias from the quantization error term,
%[OLE7]
% such as:
$\E
\left[
U^{\rm T}E
\right] = 0
$
$
\left(
\mbox{equivalently }\;
\E
\left[
\tilde{U}^{\rm T}E
\right] = 0
\right)$, i.e. \rref{eq:biaszero} and \rref{eq:bias2}.

Under Assumption~\ref{assumption:dist}, it is obvious that 
%we can show 
\begin{equation}
\int \tilde{\phi}_{k} f(\tilde{\phi}_1, \tilde{\phi}_{k})
d\tilde{\phi}_{k} = 0
\label{eq:bias100}
\end{equation}
for $k\not = 1$, then, \rref{eq:biaszero} is automatically satisfied. 
Therefore, the bias-free condition is reduced to \rref{eq:bias2}.   
Moreover, 
%note that 
\rref{eq:bias2} 
%simultaneously 
means 
\begin{equation}
\int \reg e(\reg)f(\tphi_1, \tphi_2, \dots , \tphi_n) d\reg = 0
\label{eq:sim}
\end{equation}
under Assumption~\ref{assumption:dist}. 
A sufficient condition for \rref{eq:bias2} is 
\begin{eqnarray}
\E_{\I_j}
\left[
\tilde{\phi}_{1}
{e}(\reg)
\right]
: = 
\int_{\reg\in \I_j}
\reg e(\reg) f(\reg) d\reg
 = 
\int_{\reg\in \I_j}
\reg (\qy_\j-\th1\reg) f(\reg) d\reg
 = 0, \;\; \forall j. 
\label{eq:limitE}
\end{eqnarray}
This condition is sufficiently reasonable for 
the representative number $\qy_\j$ of the subsection $\S_j$ 
(or the corresponding $\I_j$ on $\reg$).  

On the other hand, 
%the cost function
%%the objective function
%$\V[(\tilde{U}^{\rm T}E)_1]$
%is written as:
%\begin{eqnarray}
%\V\left[
%(\tilde{U}^{\rm T}E)_1
%\right]
%& = &
%\E
%\left[
%\left(
%\sum_{\i = 1}^N
%\tilde{\phi}_{1}(\i)
%{e}(\i)
%\right)^2
%\right]
% = 
%\E
%\left[
%\left(
%\sum_{\i = 1}^N
%\tilde{\phi}_{1}(\i)
%{e}(\reg(\i))
%\right)^2
%\right]
%\label{eq:star0}
%.
%\end{eqnarray}
we can derive the following key lemma
for the cost function 
$\V[{U}^{\rm T}E] \left(= V[\tilde{U}^{\rm T}E]\right)$ 
of \rref{eq:star}: 
% for this quantity.  
\begin{lemma}
\label{lemma:vari0}
Subject to the conditions: 
\begin{eqnarray}
&&
\int
\tilde{\phi}_{h}
f(\tilde{\phi}_{1}, 
\dots, 
\tilde{\phi}_{h}, 
\dots,
\tilde{\phi}_{n}
)
d\tilde{\phi}_{h} 
 = 0, \; \forall h = 1, 2, \dots , n
\label{eq:tenkaijyoken}
\\
\mbox{and}
&&
\int
\reg
{e}(\reg)
f(\reg)
d\reg
 = 0,
\label{eq:tenkaijyoken2}
\end{eqnarray}
\begin{equation}
\E
\left[
\left(
\sum_{\i = 1}^N
\tilde{\phi}_{k}(\i)
{e}(\reg(\i))
\right)^2
\right]
 = 
\left\{
\begin{array}{ll}
N
\int
\tilde{\phi}_{1}^2
{e}^2(\reg)
f(\tilde{\phi}_1)
d\tilde{\phi}_1
& \mbox{ for }\; k = 1
\cr
&
\cr
N
\int
\tilde{\phi}_{k}^2
{e}^2(\reg)
f(\tilde{\phi}_1, \tilde{\phi}_{k})
d\tilde{\phi}_1
d\tilde{\phi}_{k}
& \mbox{ for }\; k \not =  1
\end{array}
\right.
\label{eq:tenkai}
\end{equation}
is satisfied.  
\end{lemma}
The proof of this lemma is in Appendix~\ref{appendex:proofs}.  

%\tthrule

Assumption~\ref{assumption:dist} automatically guarantees the condition
\rref{eq:bias100}, i.e. \rref{eq:tenkaijyoken}, and therefore
with the bias-free condition \rref{eq:tenkaijyoken2}, 
\rref{eq:tenkai} 
%for $k = 1$ 
follows from Lemma~\ref{lemma:vari0}. 
With these preliminaries, we formulate the problem considered in this
section: 
\begin{problem}
\label{prob:optimals}
Let $\Mo$ be the number of quantized subsections $\S_j$ 
of $[-\kappa_y, \, \kappa_y] : =  [-\kappa \th1, \, \kappa \th1]$ 
on $y$ (i.e., $\I_j$ of $[-\kappa, \, \kappa]$ on
 $\reg$) where $\Mo \geq 2$.  
For the system \rref{eq:sys} 
with Assumption~\ref{assumption:dist} and a fixed $\Mo$, 
find a quantizer $q$ that minimizes 
\begin{equation}
%\frac{1}{N}
\V
\left[{U}^{\rm T}E
\right]
\left(=
\V
\left[\tilde{U}^{\rm T}E
\right]
\right)
=
%\frac{1}{N}
\sum_{k=1}^n
\E
\left[
\left(
\sum_{\i = 1}^N
\tilde{\phi}_{k}(\i)
{e}(\reg(\i))
\right)^2
\right]
=
N
\int
\sigma^2(\reg)
e^2(\reg)
f(\reg)
d\reg
% = 
%\V
%\left[
%\reg {e}(\reg)
%\right]
% = 
%\int \reg^2 {e}^2(\reg) f(\reg)
%d\reg,
\label{eq:variance2}
\end{equation}
such that 
$\E_{\I_j}\left[
\reg e(\reg)
\right] = 0$ for all $j$.
\end{problem}
The reason for the constraint $\Mo \geq 2$ is described in Note~\ref{note:odd}.

As described in Section~\ref{sec:formulation}, the quantization scheme
of $[-\kappa\tilde{\theta}_1 , \, \kappa\tilde{\theta}_1]$ 
on $y$ is essentially equal to that of $[-\kappa, \, \kappa]$ on $\reg$ 
and it is completely defined by the setting of the subsections
$\I_{-\M}$, $\dots$ , $\I_{-2}$, $\I_{-1}$,
$\I_{0}$, 
$\I_{1}$, $\I_{2}$, $\dots$ , $\I_{\M}$,
where 
\begin{equation}
\M: = 
\left\{
\matrix{
\frac{1}{2}\Mo & \mbox{ for even } \Mo \;\; (\geq 2)
\cr
\vbox to 18pt{}
\frac{1}{2}(\Mo-1) & \mbox{ for odd } \Mo \;\; (\geq 3)
}
\right.,
\label{eq:defofM}
\end{equation}
and the assigned quantized values
\[
q(y)|_{y\in \S_j}
=q(\reg\th1)|_{\reg \in \I_j}
=\qy_\j
\]
for each subsection $\S_j$ or $\I_j$ (see Fig.~\ref{fig:quant}).
Therefore, optimization of the quantization is reduced to 
a minimization problem of 
$\V[U^{\rm T}E]$
%$\V\left[\reg e(\reg) \right]$
of approximately 
%($2(M-1)\times 2$)
2$M$-variables ($d_{-(\M-1)}$, $\dots$, $d_{\M-1}$ 
and $\qy_{\langle -\M \rangle}$, $\dots$, $\qy_{\langle \M \rangle}$, 
note that $d_\M=\kappa \th1$ and $d_{-\M}=-\kappa \th1$).
% and it appears to be an extremely difficult problem. 
%However, we can show that this problem reduces to a feasible problem as follows.  

\vspace{1cm}

\centerline{
\begin{minipage}[c]{7cm}
\centerline{
\epsfxsize = 200pt
\epsfbox{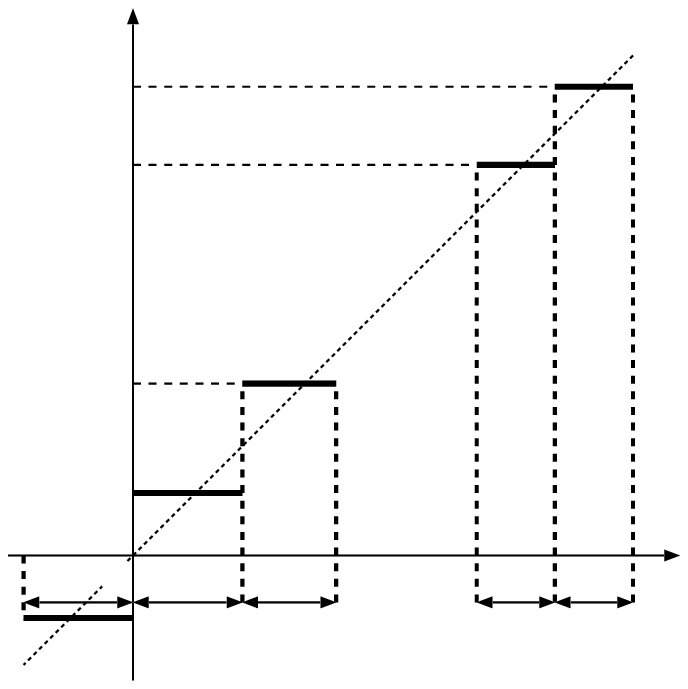}{
\tput(98,13,{$y\, ( = \reg\tilde{\theta}_1)$})
%\tput(98,13,{$\tilde{\phi}_1$})
\tput(34,8,{$d_{1}$})
\tput(48,8,{$d_{2}$})
\tput(67,8,{$d_{j-1}$})
\tput(80,8,{$d_{j}$})
\tput(90,8,{$d_{j+1}$})
\tput(6,-4,{{$\matrix{\S_{-1}\cr \updownarrow \cr \I_{-1}}$}})
\tput(25,-4,{{$\matrix{\S_1\cr \updownarrow \cr \I_1}$}})
\tput(39,-4,{{$\matrix{\S_2\cr \updownarrow \cr \I_2}$}})
\tput(56,27,{\Large {$\cdots$}})
\tput(100,27,{\Large {$\cdots$}})
\tput(73,-4,{{$\matrix{\S_j\cr \updownarrow \cr \I_j}$}})
\tput(82,-4,{{$\matrix{\S_{j+1}\cr \updownarrow \cr \I_{j+1}}$}})
\tput(6,97,{$q(y)$})
\tput(11,28,{{$\qy_{\langle 1 \rangle}$}})
\tput(11,44,{{$\qy_{\langle 2 \rangle}$}})
\tput(11,76,{{$\qy_{\j}$}})
\tput(6,88,{{$\qy_{\langle j+1 \rangle}$}})
\tput(14,15,{{$O$}})
\tput(90,97,{$q(y) = y$})
}
}
%\refstepcounter{figure}
%\label{fig:quant}
%\vspace{5mm}
%\centerline{Fig.\,{\thefigure} \hspace{2mm} The quantization 
\refstepcounter{figure}\label{fig:quant}
\vspace{10mm}
%\vspace{3mm}
\centerline{Fig.\,{\thefigure} \hspace{2mm} The quantization 
scheme of $q$}
\end{minipage}
}

\vspace{5mm}

In this section, we consider the case of even $\Mo$. 
The case of odd $\Mo$, that is, $\S_0 \not= \{0\}$ ($\I_0 \not= \{0\}$), 
is reduced to the even case and 
the reason is explained in Note~\ref{note:odd}.  
We also refer to the positive domain $\S_1$, $\S_2$, $...$ because of 
quantization symmetry. 

It is known that when a subsection $\S_j$ is fixed
(i.e. $d_{j-1}$ and $d_{j}$ are fixed), 
$\qy_\j$ is given by 
the bias-free condition 
$\E_{\I_j}\left[
\reg e(\reg)
\right] = 0$.   
Therefore, the optimization problem is reduced to finding optimal 
$d_{-(\M-1)}$, $\dots$, $d_{\M-1}$.  
Corresponding to $d_j$, 
we introduce key variables, ratios $r_j$ ($j = 1, \, \dots, \frac{1}{2}\Mo
-1$) between $d_j$ and $d_{j+1}$ defined by:
\begin{equation}
d_j = r_j d_{j+1}, \; r_j \in [0, 1].  
\label{eq:rdef}
\end{equation}
Note that determining optimal $d_{-(\M-1)}$, $\dots$, $d_{\M-1}$ is equal to 
determining optimal $r_{-(\M-1)}$, $\dots$, $r_{\M-1}$ and we derive the following result.   

\begin{proposition}
\label{th:mainprop}
%
%\begin{itemize}
%\item[(1)]
The optimal ratios $r_j^o$ for Problem~\ref{prob:optimals}
are given by solving the following recursive optimization problem iteratively.
\begin{eqnarray}
&&
r_j^o = \arg \min_{r \in [0, \, 1]} 
\left(
d_{j+1}^5\psi(r ; \psi_{j-1}^{\min})
+20\kappa_y^2(n-1)d_{j+1}^3
\xi(r ; \xi_{j-1}^{\min})
\right)
%\psi_j(r)
\label{eq:ratio0}
\\
&&
\psi(r ; \alpha) : =  
\alpha r^5  -18(1-r)^5 + 45 (1+r)^2(1-r)^3 + 5(1-r)^7(1+r)^{-2}
\label{eq:function0}
\\
&&
\psi_j^{\min} : =  \psi(r_j^o ; \psi_{j-1}^{\min})
\nonumber
\\
&&
\psi_0^{\min}: = 
32
\nonumber
\\
&&
\xi(r ; \alpha) : =  
\alpha r^3  + 3 (1-r)^3 + \frac{(1-r)^5}{(1+r)^2}
\label{eq:function0_xi}
\\
&&
\xi_j^{\min} : =  \xi(r_j^o ; \xi_{j-1}^{\min})
\nonumber
\\
&&
\xi_0^{\min}: = 
4.
\nonumber
%\label{eq:min0}
\end{eqnarray}
The optimal value of 
\rref{eq:variance2}
is given by
\begin{eqnarray}
\min_q \V\left[{U}^{\rm T}E\right]
\left(
=
\min_q \V\left[\tilde{U}^{\rm T}E\right]
%\min_q \V\left[\reg\cdot e(\reg)\right]
%&=&
\right)
=
\min_q \sum\limits_{j = -\M}^\M
\V_{\I_j}\left[
\tilde{U}^{\rm T}E
%\reg\cdot e(\reg)
\right]
%\nonumber
%\\
&=&
\frac{N}{2160}
\kappaY^4
(
\psi_{\M-1}^{\min} + 20(n-1)\xi_{\M-1}^{\min}
)
\nonumber
\\
&=& 
\frac{N}{2160}
\tilde{\theta}_1^4
\kappa^4
(
\psi_{\M-1}^{\min} + 20(n-1)\xi_{\M-1}^{\min}
).
\label{eq:vari00}
\end{eqnarray}
\end{proposition}
See Appendix~\ref{appendex:proofs} for the proof. 

\begin{note}
\label{note:odd}
\rm
For odd $\Mo$, 
there must not exist a subsection $\S_0$ (i.e. $\I_0$) of nonzero width that contains the
 origin of $y$ (i.e., origin of $\reg$)
 because for any such subsection and setting $\qy_{\langle 0 \rangle}$, 
$\E_{\I_0}
\left[
\reg e(\reg)\right] \not =  0
$.
This means that 
%one of the $\Mo$ subsections, that is 
$\S_0$ (i.e. $\I_0$) should be $\{ 0 \}$ 
%$\emptyset$ 
and consequently
the problem is equal to the case of even $\Mo$ with the setting $\M = \frac{1}{2}(\Mo -1)$.  
\hfill{$\diamondsuit$}
\end{note}

\begin{example}
\rm

Consider the following second-order FIR model as an example of \rref{eq:sys}:
\begin{equation}
y(\i) = \theta_1 u(\i) + \theta_2 u(\i-1),
\label{eq:2ndFIR}
\end{equation}
where $\theta_1 = \frac{\sqrt{3}}{2}$ and $\theta_2 = \frac{1}{2}$ 
and the system is noise free.
We generate 50 sets of I/O data sequences with a length $N=10,000$ for the system 
\rref{eq:2ndFIR} that obey Assumption~\ref{assumption:dist} and
$\kappa = 4$ (i.e., $\kappaY = 4$). 
Fig.~\ref{fig:histogram} is one of the histogram of $10,000$ samples of
$\reg$ from 50 sets. 

%\hspace{.5cm}
%
%\centerline{
%\begin{minipage}[c]{8cm}
%\centerline{
%%\epsfxsize = 170pt \epsfbox{datahistbw.eps}{
%\epsfxsize = 170pt \epsfbox{matlab/datahistbw.eps}{
%%\epsfxsize = 170pt \epsfbox{matlab/I5/datahistbw.eps}{
%\tput(50,-10,{$\tilde{\phi}_1$})
%\tput(-10,32,{\rotatebox{90}{number of $\tilde{\phi}_1$}})
%%\tput(-40,50,{number of $\tilde{\phi}_1$})
%}
%}
%\refstepcounter{figure}\label{fig:histogram}
%\vspace{6mm}
%\centerline{Fig.\,{\thefigure} \hspace{2mm} Histogram of $\reg$}
%\end{minipage}
%}
%
%\vspace{.5cm}

Next, quantize the output data $y$ with the optimal quantizers given by
Proposition~\ref{th:mainprop} and with uniform quantizers, for
comparison, subject to the constraints $\M = 5$ ($\Mo = 10$).
Fig \ref{fig:optm5} shows the step function $q$ for $y$ of the optimal
quantizer for $\M = 5$.  
%Fig \ref{fig:uniformm5} shows the corresponding step function of the
%uniform quantizer.  
Fig \ref{fig:optm5} 
%and Fig
%\ref{fig:uniformm5} 
%show the difference and the former 
indicates a basic
property of the optimal quantizer, that is, it is coarse near the origin
and becomes denser away from the origin.  

The bias term $\frac{1}{N}\sum_{\i = 1}^{N} \reg(\i) e(\i)$
and the quantization error term $\Delta {E}$ were calculated;
Table~\ref{table:m5} shows a summary of the results.  
From Table~\ref{table:m5}, the optimal quantizer, which minimize $\V[U^{\rm T}E]$
attains a lower $\|\Delta {E}\|_2^2$ than that of the uniform quantizer.

\vspace{1cm}

\hfill
\hspace{.5cm}

\begin{minipage}[c]{8cm}
\centerline{
\epsfxsize = 170pt \epsfbox{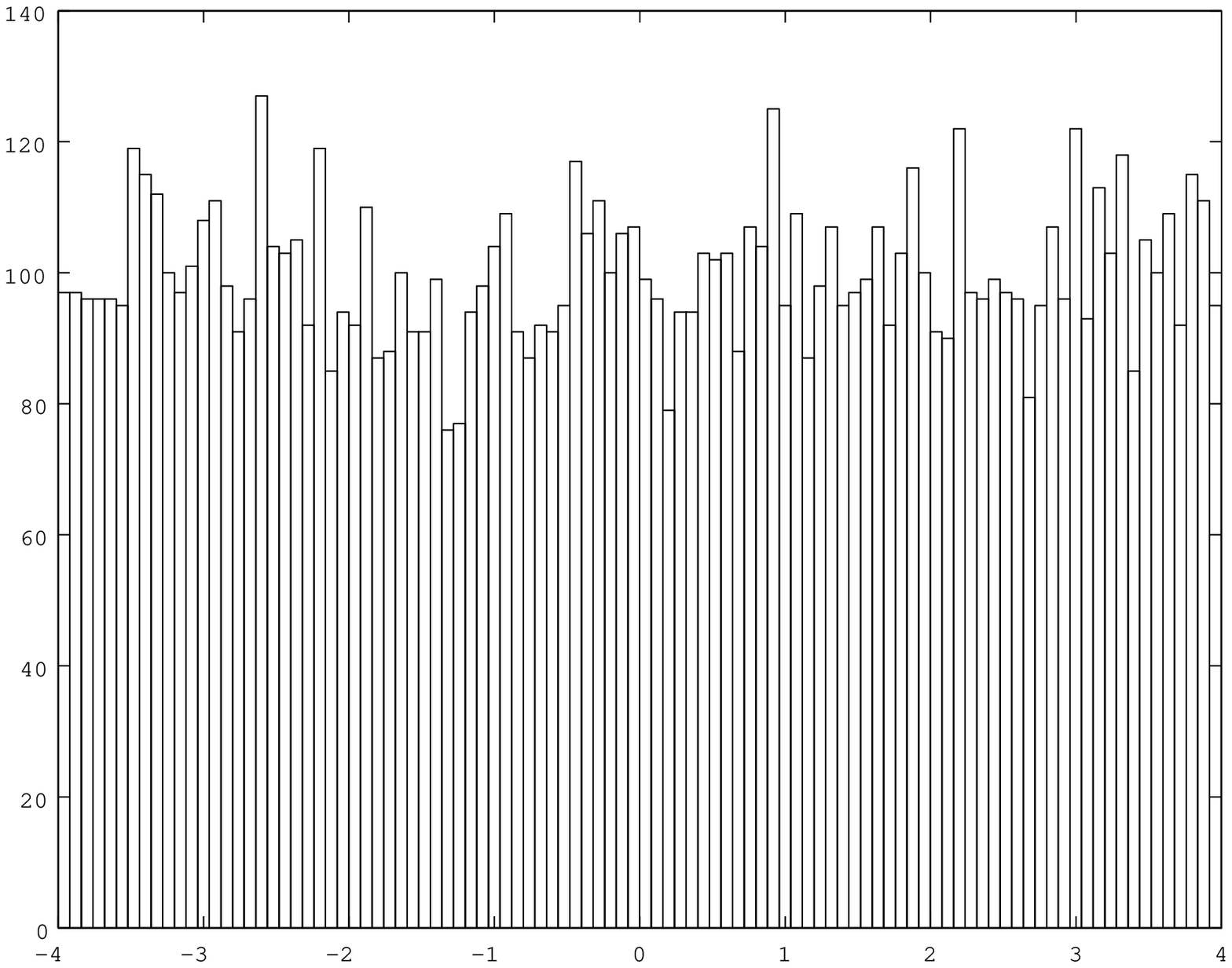}{
\tput(50,-10,{$\tilde{\phi}_1$})
\tput(-10,32,{\rotatebox{90}{number of $\tilde{\phi}_1$}})
%\tput(-40,50,{number of $\tilde{\phi}_1$})
}
}
\refstepcounter{figure}\label{fig:histogram}
\vspace{6mm}
\centerline{Fig.\,{\thefigure} \hspace{2mm} Histogram of $\reg$}
\end{minipage}
\begin{minipage}[c]{8cm}
\centerline{
\epsfxsize = 150pt \epsfbox{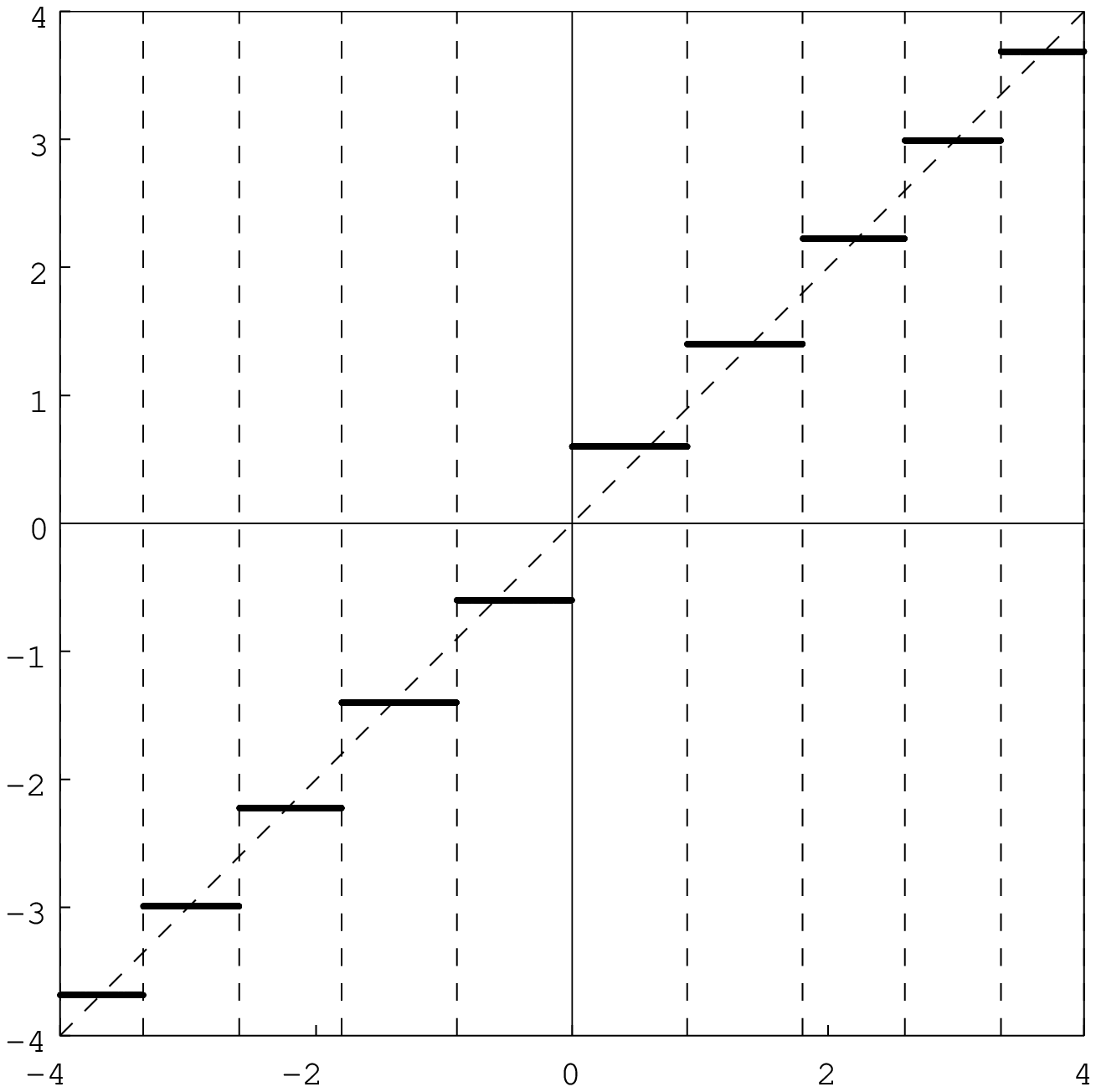}{
\tput(100,50,{$y$})
\tput(50,102,{$q(y)$})
%\tput(100,50,{$\tilde{\phi}_1$})
%\tput(50,102,{$q'(\tilde{\phi}_1)$})
}
}
\refstepcounter{figure}\label{fig:optm5}
\vspace{6mm}
\centerline{Fig.\,{\thefigure} \hspace{2mm} Optimal quantization scheme 
%\mbox{\Qopt} 
for $\M = 5$}
\end{minipage}
%\hfill
%\begin{minipage}[c]{8cm}
%\centerline{
%\epsfxsize = 150pt \epsfbox{I5ichiyoquant.eps}{
%%\epsfxsize = 150pt \epsfbox{matlab/I5/ichiyoquant.eps}{
%\tput(100,50,{$y$})
%\tput(50,102,{$q(y)$})
%%\tput(100,50,{$\tilde{\phi}_1$})
%%\tput(50,102,{$q'(\tilde{\phi}_1)$})
%}
%}
%\refstepcounter{figure}\label{fig:uniformm5}
%\vspace{6mm}
%\centerline{Fig.\,{\thefigure} \hspace{2mm} Uniform quantization scheme for $M = 5$}
%\end{minipage}
%\hfill
%\break

\vspace{.7cm}

\begin{minipage}[t]{16cm}
\refstepcounter{table}\label{table:m5}
\centerline{Table.\,{\thetable} \hspace{2mm}
The ratios of the biases and the squares of errors for $\M = 5$ (averages of 50 sets)
%The ratios of the biases, the squares of errors, and the quantization errors for $M = 5$ 
}
\vspace{.3cm}
\centerline{
\begin{tabular}{|c|c|}
\hline
\vbox to 12pt{}$|$ $\sum_{t=1}^{10000}\tilde{\phi}_1(t) e(t)|$ by 
opt. quant.
%\mbox{\Qopt} 
/
 $|$ $\sum_{t=1}^{10000}\tilde{\phi}_1(t) e(t)|$ by unif. quant. &
0.1107 \\
%\vbox to 12pt{}$|$ave. $\tilde{\phi}_1 \cdot e|$ by \mbox{\Qopt} /
% $|$ave. $\tilde{\phi}_1 \cdot e|$ by unif. quant. & 0.0135 \\
\hline
%\vbox to 12pt{}ave. $\tilde{\phi}_1^2 \cdot e^2$ by \mbox{\Qopt} /
% ave. $\tilde{\phi}_1^2 \cdot e^2$ by unif. quant. & 0.6150 \\
%\hline
\vbox to 12pt{}$\|\Delta E\|_2^2$ by opt. quant.
%\mbox{\Qopt} 
/
 $\|\Delta E\|_2^2$ by unif. quant. & 
0.0132 \\
%\vbox to 12pt{}$|(\Delta \tilde{E})_1|$ by \mbox{\Qopt} /
% $|(\Delta \tilde{E})_1|$ by unif. quant. & 0.0116 \\
\hline
\end{tabular}
}
\end{minipage}

%\begin{minipage}[t]{16cm}
%\refstepcounter{table}\label{table:m5}
%\centerline{Table.\,{\thetable} \hspace{2mm}
%The ratios of the biases and the squares of errors for $\M = 5$ (averages of 50 sets)
%%The ratios of the biases, the squares of errors, and the quantization errors for $M = 5$ 
%}
%\vspace{.3cm}
%\centerline{
%\begin{tabular}{|c|c|c|}
%\hline
%& $n=1$ & $n=2$ \\
%\hline
%\vbox to 12pt{}$|$ $\sum_{t=1}^{10000}\tilde{\phi}_1(t) e(t)|$ by 
%opt. quant.
%%\mbox{\Qopt} 
%/
% $|$ $\sum_{t=1}^{10000}\tilde{\phi}_1(t) e(t)|$ by unif. quant. &
%0.0135 & 0.1107 \\
%%\vbox to 12pt{}$|$ave. $\tilde{\phi}_1 \cdot e|$ by \mbox{\Qopt} /
%% $|$ave. $\tilde{\phi}_1 \cdot e|$ by unif. quant. & 0.0135 \\
%\hline
%%\vbox to 12pt{}ave. $\tilde{\phi}_1^2 \cdot e^2$ by \mbox{\Qopt} /
%% ave. $\tilde{\phi}_1^2 \cdot e^2$ by unif. quant. & 0.6150 \\
%%\hline
%\vbox to 12pt{}$\|\Delta E\|_2^2$ by opt. quant.
%%\mbox{\Qopt} 
%/
% $\|\Delta E\|_2^2$ by unif. quant. & 0.0116 & 0.8155 \\
%%\vbox to 12pt{}$|(\Delta \tilde{E})_1|$ by \mbox{\Qopt} /
%% $|(\Delta \tilde{E})_1|$ by unif. quant. & 0.0116 \\
%\hline
%\end{tabular}
%}
%\end{minipage}

\hfill{$\diamondsuit$}
\end{example}

\vspace{5mm}

Proposition~\ref{th:mainprop} shows that the problem is in a category of the typical dynamic
programming and we can solve it by numerical calculation.  
In general, the computation complexity of this problem is high; 
however, 
%surprisingly, 
the optimization problem \rref{eq:ratio0} can be
solved by very few calculation steps in special cases $n=1$ or
$n \gg 1$, respectively, 
as shown in the following theorem:
\begin{theorem}
\label{th:mainspec}
When $n=1$, the optimal ratios $r_j^o$ for Problem~\ref{prob:optimals}
are given by solving the following optimization problem iteratively.
\begin{eqnarray}
&&
r_j^o = \arg \min_{r \in [0, \, 1]} \psi(r; \psi_{j-1}^{\min})
\label{eq:ratio}
\\
&&
%\psi(r; \alpha) : =  
%\alpha r^5  -18(1-r)^5 + 45 (1+r)^2(1-r)^3 + 5(1-r)^7(1+r)^{-2}
%\label{eq:function}
%\\
%&&
\psi_j^{\min} : =  \psi(r_j^o; \psi_{j-1}^{\min})
\nonumber
\\
&&
\psi_0^{\min}: = 
32.
\label{eq:min}
\end{eqnarray}
The optimal value of 
\rref{eq:variance2}
is given by
\begin{eqnarray}
\min_q \V\left[{U}^{\rm T}E\right]
\left(
=
\min_q \V\left[
\tilde{U}^{\rm T}E
%\reg\cdot e(\reg)
\right]
\right)
%=
%\min_q \sum\limits_{j = -M}^M
%\V_{\I_j}\left[\reg\cdot e(\reg)\right]
%=
%\frac{1}{2160}
%\kappaY^4 \psi_{M-1}^{\min}
= 
\frac{N}{2160}
\tilde{\theta}_1^4\kappa^4 \psi_{\M-1}^{\min}.
\label{eq:vari0}
\end{eqnarray}
Similarly,
when $n \gg 1$, the optimal ratios $r_j^o$ for Problem~\ref{prob:optimals}
converge to the solution of 
the following optimization problem.
\begin{eqnarray}
&&
r_j^o = \arg \min_{r \in [0, \, 1]} \xi(r; \xi_{j-1}^{\min})
\label{eq:ratio00}
\\
&&
%\xi(r; \alpha) : =  
%\alpha r^3  + 3 (1-r)^3 + \frac{(1-r)^5}{(1+r)^2}
%\label{eq:function00}
%\\
%&&
\xi_j^{\min} : =  \xi(r_j^o; \xi_{j-1}^{\min})
\nonumber
\\
&&
\xi_0^{\min}: = 
4.
\label{eq:min00}
\end{eqnarray}
The optimal value of 
\rref{eq:variance2}
converges to
%is given by
\begin{eqnarray}
%\min_q \V\left[\reg\cdot e(\reg)\right]
%=
\frac{N}{108}
\tilde{\theta}_1^4
\kappa^4
(n-1)\xi_{\M-1}^{\min}
.
\label{eq:vari000}
\end{eqnarray}
\end{theorem}
%The proof is straightforward from the result of Proposition~\ref{th:mainprop}.  

\begin{note}
\rm
The definitive difference of the optimization problems 
\rref{eq:ratio0} and \rref{eq:ratio} or \rref{eq:ratio00}
is that in the former case, 
$r_j^o$ depends on $d_{j+1}$ and this requires a complex calculation
such as dynamic programming, on the other hand, 
in the latter cases,  
$r_j^o$ does not depend on $d_{j+1}$ and 
$\{r_j^o\}$ can be given by solving \rref{eq:ratio} or \rref{eq:ratio00}
from $j=1$ to $j=\M-1$ in turn only once.  
This means that the original minimization problem of approximately 
$2M$-variable
function $\V\left[{U}^{\rm T}E\right]$ can be reduced to a
recursive minimization problem of a single one-variable rational
function when $n=1$ or $n \gg 1$.  
Moreover, when $n=1$, 
from Lemma~\ref{lemma:a1} in Appendix~\ref{appendex:proofs}, 
the local minimum of 
%$\psi_j(r)$ 
$\phi(r; \alpha)$, $\alpha >0$, 
in $r \in (0, \; 1)$ is unique. 
Therefore, finding the minimizer does not require a highly complex
 calculation.  
\hfill{$\diamondsuit$}
\end{note}

%\begin{note}
%\rm
%The idea of the proof is on a fact that 
%the optimal ratios of $r_1^o$, $r_2^o$, $\dots$ is independent of the
%$\kappa_y$ because the distribution of $y$ is uniform.   
%Therefore the $j$th optimal ratio $r_j^o$ is given only with 
%the previously fixed optimal rations $r_1^o$, $_2^o$, $\dots$ , $r_{j-1}^o$.  
%The recursive algorithm of minimizing one-variable function originates in this fact. 
%\end{note}

%\begin{note}
%\rm
%The original minimization problem of approximately 
%%$(2(M-1)\times 2)$
%$4M$-variable
%function $\V\left[\reg\cdot e(\reg) \right]$ can be reduced to a
%recursive minimization problem of a single one-variable rational function.  
%Moreover, from Lemma~\ref{lemma:a1} in Appendix~\ref{appendex:proofs}, 
%the local minimum of $\psi_j(r)$ in $r \in (0, \; 1)$ is unique. 
%Therefore, finding the minimizer does not require a highly complex
% calculation.  
%\end{note} 

In the following of this section, 
we focus on the case $n=1$ 
because it is a basic problem and reveals typical property of the
optimal quantization. 
We call the optimal quantization scheme as \Qopt hereafter.

Every optimal ratio $r_j^o$ can be explicitly determined by solving \rref{eq:ratio} --
\rref{eq:min} iteratively; however, the properties of the sequence $r_1^o$, $r_2^o$, $...$  
are not clear from \rref{eq:ratio} -- \rref{eq:min}.   
For the asymptotic characteristics
of the optimal ratios $r_j^o$ $(j = 1, 2, \dots)$ and related quantities,
we derive the following series of Lemma~\ref{lemma:rate} -- \ref{lemma:asympt}.
\begin{lemma}
\label{lemma:rate}
The optimal ratios $r_j^o$ satisfies: 
\begin{eqnarray*}
&
r_j^o < r_{j+1}^o, \; \forall j > 0,
\\
&
r_j^o \rightarrow 1, \; j \rightarrow \infty.
\end{eqnarray*}
\end{lemma}

\begin{lemma}\label{lemma:rate2}
The width of the subsections $\S_j$ or $\I_j$ of \mbox{\rm \Qopt} satisfy: 
\begin{eqnarray*}
|\S_j| > |\S_{j+1}|, \; 
|\I_j| > |\I_{j+1}|, \; \forall j > 0,
\end{eqnarray*}
where $|\cdot|$ denotes the width of the subsection.
\end{lemma}
The proofs of these lemmas are in Appendix~\ref{appendex:proofs}.  

Lemma~\ref{lemma:rate2} shows that the optimal quantization scheme
\Qopt\ has the property that it is coarse near the origin of
$y$ and becomes denser as $y$ tends to the boundaries of $[-\kappaY, \, \kappaY]$.
This property coincides with the results in Section~\ref{section:main} 
and it is also 
%is, in some sense, 
the dual result to that of the quantization
problem for stabilization by \cite{Elia:IEEE01} as mentioned in
Section~\ref{section:main}.  
%; that is, the coarsest
%quantization scheme for stabilization is dense near the origin and
%becomes coarser as distance from the origin increases.
%These observations suggest that there appears to exist a trade-off between
%parameter estimation and stabilization in the quantization scheme for
%a type of adaptive control system.

Next, consider the unboundedness of $\prod_{j = 1}^\infty \frac{1}{r_j^o}$.
If it is bounded and
$\prod_{j = 1}^\infty \frac{1}{r_j^o} = \gamma < \infty$,
then this causes a contradiction as to the optimality of \Qopt, that is,
when a region $[-\gamma, \; \gamma]$ of $\tilde{\phi}_1$ is quantized,
the width of $\I_1$, for example, is never smaller than 1 even if the
number of quantization levels increases to infinity.
Of course, this is not true and $\prod_{j = 1}^\infty \frac{1}{r_j^o}$ is
therefore unbounded.
The next lemma strictly describes this fact.  
Refer to \cite{Tsumura:METR05-04} for the proof.
%See Appendix~\ref{appendex:proofs} for the proof.  
\begin{lemma}
\label{lemma:unbound}
The optimal ratios $r_j^o$ satisfies: 
\begin{equation*}
\prod_{j = 1}^\infty \frac{1}{r_j^o} = \infty
\end{equation*}
\end{lemma}
From Lemma~\ref{lemma:rate} to Lemma~\ref{lemma:unbound}, we know the
outline of the quantization of the region $[-\kappaY, \, \kappaY]$.

Next, 
to clarify the profile of 
$\V\left[U^{\rm T}E\right]$ 
%$\V\left[\reg \cdot e(\reg)\right]$ 
with respect to $\M$, the following lemma confirms the
asymptotic characteristics of $\psi_\M^{\min}$.
\begin{lemma}
\label{lemma:asympt}
The minimized quantity $\psi_j^{\min}$ of \rref{eq:function0} at $j=\M$ converges as 
\[
\psi_\M^{\min}
\rightarrow
\Psi_{a}^{b}(\M),
\; \M \rightarrow \infty,
\]
where $a = -5 \cdot 3^{-\frac{5}{2}}$ and $b = \frac{3}{2}$,
and $\Psi_a^b(m)$ is a function of integer $m$ defined as the solution of the
following recurrence formula with an appropriate initial number
$\psi(0) =  \psi_o$:
\begin{equation}
\hat{\psi}(m) - \hat{\psi}(m-1)
 = 
a \hat{\psi}^b(m-1).
\label{eq:diff100}
\end{equation}
\end{lemma}
The proof is in Appendix~\ref{appendex:proofs}.

Note that the recurrence formula \rref{eq:diff100} is from
\rref{eq:symmain} in Appendix~\ref{appendex:proofs}, and 
it can be approximated
% $\hat{\psi}$ 
by $\tilde{\psi}$, which is a solution of 
a differential equation:
\[
\frac{d\tilde{\psi}(m)}{dm}
 =  (a + \nu) \tilde{\psi}^b(m) \geq
a\tilde{\psi}^b(m) + o(\tilde{\psi}^b(m))
 = a\tilde{\psi}^b(m) + O(\tilde{\psi}^2(m)) = {\cal
 P}(\tilde{\psi}(m)), \; m \in {\cal R}
,
\]
where ${\cal P}(\bullet)$ is defined in \rref{eq:symmain} and 
$\nu > 0$ is an appropriate constant number satisfying $a+\nu < 0$ and
the above inequality (such $\nu$ always exists).  
We can show $\tilde{\psi}(m) \geq \hat{\psi}(m)$ at sufficiently large
integer $m$ when $\tilde{\psi}(0) \geq \hat{\psi}(0)$ in
Lemma~\ref{lemma:recurrentdiff}.
Then, we obtain the solution
\begin{equation}
\tilde{\psi}(m)  = 
\left\{
(-b+1)(a+\nu)m
+B
\right\}^{\frac{1}{-b+1}}
\label{eq:approxconvex}
\end{equation}
for an appropriate constant $B$.  
From \rref{eq:vari0} and \rref{eq:approxconvex}, 
we obtain
\begin{eqnarray}
\min_q 
\V\left[U^{\rm T}E\right]
%\V \left[\reg\cdot e(\reg)\right]
&\leq&
\frac{N}{2160}
\kappa^4 ((-3/2+1)((-5 \cdot 3^{-\frac{5}{2}}+\nu) (\M-1)+B))^{\frac{1}{-3/2+1}}
\nonumber
\\
& = &
A \kappa^4 (\M-B^\star)^{-2}
\nonumber
\\
A&: = &
\frac{N}{540}
\left(
5 \cdot 3^{-\frac{5}{2}}
-\nu
\right)^{-2}, \;\; 
%\nonumber
%\\
B^\star: = (5 \cdot 3^{-\frac{5}{2}}-\nu)^{-1}B.
\label{eq:variapp}
\end{eqnarray}
This \rref{eq:variapp} approximately shows the relationship between the optimized
quantization error 
$\min_q \V\left[U^{\rm T}E\right]$
%\V \left[\reg\cdot e(\reg)\right]
and the number of quantization levels.
%In Section~\ref{sec:eval},
%this result is also used to evaluate 
%the magnitude of $\Delta E$.

\begin{example}
\rm

Consider the following first-order FIR model for verifying the above results:
\begin{equation}
y(\i) = \theta u(\i),
\label{eq:1stFIR}
\end{equation}
where $\theta = 2$ and the system is noise free.
We also generate 50 sets of I/O data sequences with a length $N=10,000$ for the system 
\rref{eq:1stFIR} that obey Assumption~\ref{assumption:dist} and
$\kappa = 4$ (i.e., $\kappaY = 8$). 
%Fig.~\ref{fig:histogram2} is one of the histogram of $10,000$ samples of
%$\reg$ from 50 sets. 
%
%\hspace{.5cm}
%
%\centerline{
%\begin{minipage}[c]{8cm}
%\centerline{
%%\epsfxsize = 170pt \epsfbox{datahistbw.eps}{
%\epsfxsize = 170pt \epsfbox{matlab/I5/datahistbw.eps}{
%\tput(50,-10,{$\tilde{\phi}_1$})
%\tput(-10,32,{\rotatebox{90}{number of $\tilde{\phi}_1$}})
%%\tput(-40,50,{number of $\tilde{\phi}_1$})
%}
%}
%\refstepcounter{figure}\label{fig:histogram2}
%\vspace{6mm}
%\centerline{Fig.\,{\thefigure} \hspace{2mm} Histogram of $\reg$}
%\end{minipage}
%}
%
%\vspace{.5cm}

Next, quantize the output data $y$ with the optimal quantizers given by
Theorem~\ref{th:mainspec} 
and with uniform quantizers, for comparison, subject to the constraints $\M = 5$
($\Mo = 10$).
Fig \ref{fig:optm55} shows the step function $q$ for $y$ of the optimal
quantizer for $\M = 5$. 
%Fig \ref{fig:uniformm55} shows the corresponding
%The step function of the uniform quantizer is the same in Fig \ref{fig:uniformm5}.  
From the comparison with Fig \ref{fig:optm5},
Fig \ref{fig:optm55} more clearly shows the property of the optimal
quantizer, that is, it is coarse near the origin and becomes denser away from the origin. 

%???????????????????

Table~\ref{table:m55} shows comparison of 
the bias term $\frac{1}{N}\sum_{\i = 1}^{N} \reg(\i) e(\i)$
and the quantization error term $\Delta {E}$.
From Table~\ref{table:m55}, the optimal quantizer, which minimize
$\V[U^{\rm T}E]$ attains a lower $\|\Delta {E}\|_2^2$ than those of the
uniform quantizer.

\vspace{1cm}

%\hfill
\centerline{
\begin{minipage}[c]{8cm}
\centerline{
\epsfxsize = 150pt \epsfbox{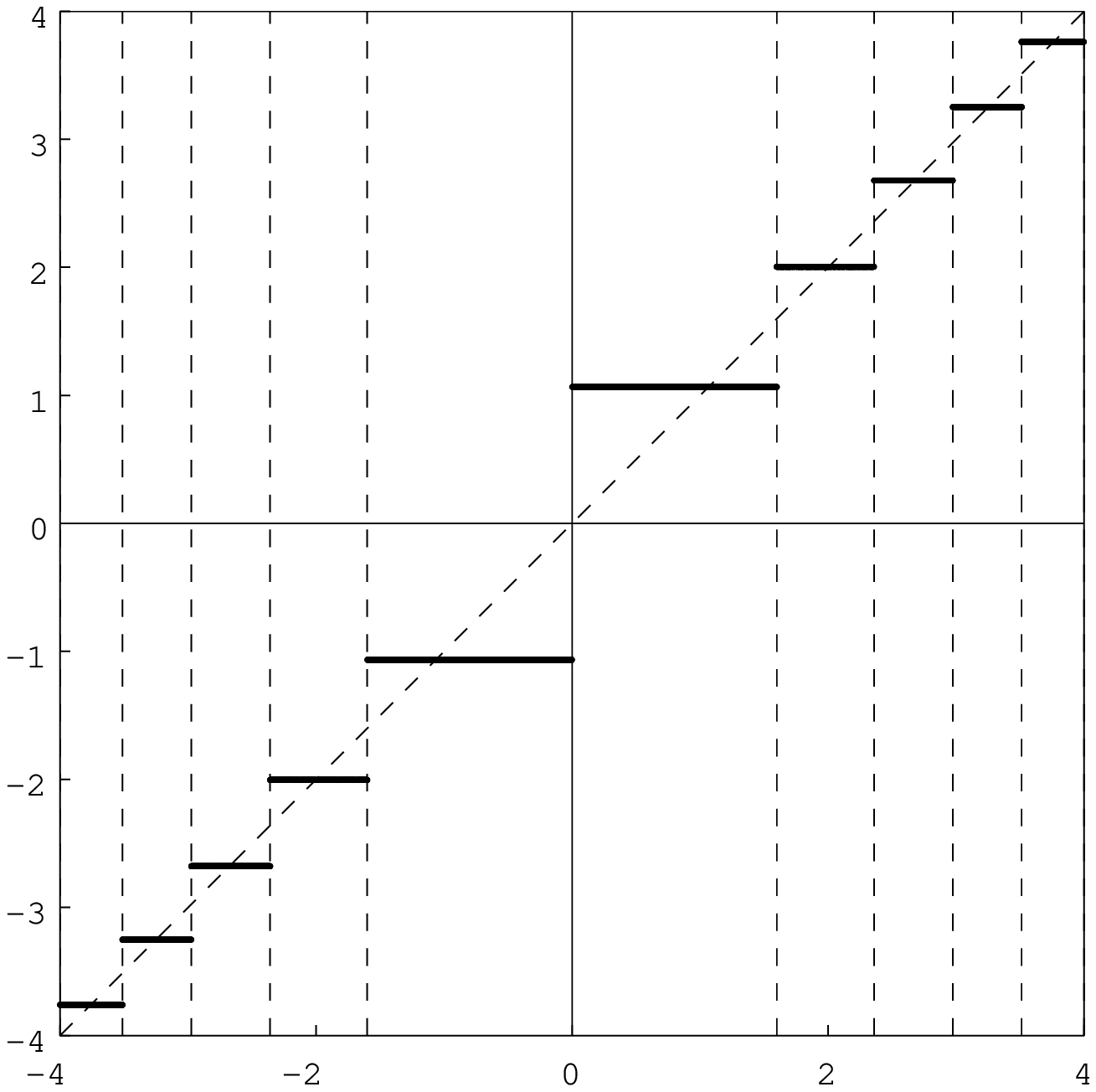}{
\tput(100,50,{$y$})
\tput(50,102,{$q(y)$})
%\tput(100,50,{$\tilde{\phi}_1$})
%\tput(50,102,{$q'(\tilde{\phi}_1)$})
}
}
\refstepcounter{figure}\label{fig:optm55}
\vspace{6mm}
\centerline{Fig.\,{\thefigure} \hspace{2mm} Optimal quantization scheme \mbox{\Qopt} for $\M = 5$}
\end{minipage}
}

%\hfill
%\begin{minipage}[c]{8cm}
%\centerline{
%\epsfxsize = 150pt \epsfbox{I5ichiyoquant.eps}{
%%\epsfxsize = 150pt \epsfbox{matlab/I5/ichiyoquant.eps}{
%\tput(100,50,{$y$})
%\tput(50,102,{$q(y)$})
%%\tput(100,50,{$\tilde{\phi}_1$})
%%\tput(50,102,{$q'(\tilde{\phi}_1)$})
%}
%}
%\refstepcounter{figure}\label{fig:uniformm55}
%\vspace{6mm}
%\centerline{Fig.\,{\thefigure} \hspace{2mm} Uniform quantization scheme for $M = 5$}
%\end{minipage}
%\hfill
%\break

\vspace{.7cm}

\begin{minipage}[t]{16cm}
\refstepcounter{table}\label{table:m55}
\centerline{Table.\,{\thetable} \hspace{2mm}
The ratios of the biases and the squares of errors for $\M = 5$ (averages of 50 sets)
}
\vspace{.3cm}
\centerline{
\begin{tabular}{|c|c|}
\hline
\vbox to 12pt{}$|$ $\sum_{t=1}^{10000}\tilde{\phi}_1(t) e(t)|$ by \mbox{\Qopt} /
 $|$ $\sum_{t=1}^{10000}\tilde{\phi}_1(t) e(t)|$ by unif. quant. & 0.0135 \\
\hline
\vbox to 12pt{}$\|\Delta E\|_2^2$ by \mbox{\Qopt} /
 $\|\Delta E\|_2^2$ by unif. quant. & 0.0116 \\
\hline
\end{tabular}
}
\end{minipage}

\hfill{$\diamondsuit$}
\end{example}

\vspace{1cm}

%\input{high.tex}

%\input{noise.tex}

%#!platex paper

\section{Resolution of the Quantizer and I/O Data Length}\label{sec:eval}

In the system identification of \rref{eq:sys}, it is important to
clarify the relationship between the estimation error and the amount of
signal data used for the estimation.  
The amount of signal data is the resolution of the quantization 
multiplied by the length of signal sequence.  
Using the results in the previous sections, we evaluate the magnitudes
of the error term $\Delta \tilde{E}$ 
%in Section~\ref{sec:prev}
and $\Delta \tilde{W}$ based on the approach in
\cite{Tsumura:CDC99} and compare the effects of the resolution of
quantizers and the length of signal sequence.
%For the cases described in Section~\ref{section:main}, 
%similar results can be derived; however they are omitted here.   

First, the evaluation of the magnitude
of $(\tilde{U}^{\rm T}\tilde{U})^{-1}$.
\begin{lemma}
\label{lemma:uinv}{\rm \cite{Tsumura:CDC99}}
Assume that 
$\tphi$
%$\tilde{u}(\i)$ 
satisfies Assumption~\ref{assumption:0} and \ref{assumption:smooth}
%Assumption~\ref{assumption:dist}
with
$\V[\reg(\i)] = \sigma_{\reg}^2$, $\V[\reg^2(\i)] =  \eta$.
%$\V[\tilde{u}(\i)] = \sigma_{\tilde{u}}^2$,  $\V[\tilde{u}^2(\i)] =  \eta$.
Then, for any reliability index $\beta_1 >0$, where  $1-\beta_1 > 0$, and
$\sigma_{\reg}^2 N-
n\sqrt{\frac{N}{\beta_1}}
\left(
\sqrt{\eta} + (n-1)\sigma_{\reg}^2
\right)>0$,
the following inequality is satisfied.
\begin{eqnarray}
&&
{\mbox{\rm{Prob}}}
\left(
\|(\tilde{U}^{\rm T}\tilde{U})^{-1}\|_1
\geq
\epsilon_1
\right)
\le \beta_1
\nonumber
\\
&&
\epsilon_1: = 
\frac{1}{\sigma_{\tilde{u}}^2 N-
n\sqrt{\frac{N}{\beta_1}}
\left(
\sqrt{\eta} + (n-1)\sigma_{\reg}^2
\right)
}
\label{eq:1epsilon1}
\end{eqnarray}
\end{lemma}

%When $\reg(\i)$ has a uniform distribution:
%$\reg(\i)$ $\in$ $[-\kappa, \kappa]$, that is,
%$\sigma_{\reg}^2 = \frac{1}{3}\kappa^2$, $\eta = \frac{4}{45}\kappa^4$,
%then,
%\[
% \epsilon_1  = 
%\frac{1}{
%\kappa^2
%\left(
%\frac{1}{3}N -
%n
%\left(
%\sqrt{\frac{4}{45}}+\frac{1}{3}(n-1)
%\right)
%\sqrt{\frac{N}{\beta_1}}
%\right)
%}.
%\]

Using Lemma \ref{lemma:uinv}, we evaluate 
$\|\Delta \tilde{E}\|_\infty$ in the following theorem.
\begin{theorem}
\label{lemms:quantterm}
For the system \rref{eq:sys} with the optimal quantizer  
$q(y)$ defined by \rref{eq:quant10} -- \rref{eq:dd},
\rref{eq:opthigh}, 
%\rref{eq:ratio} -- \rref{eq:min}, 
assume Assumption~\ref{assumption:0} and \ref{assumption:smooth}.  
%assume Assumption~\ref{assumption:dist}. 
Then, for the reliability indices $\beta_1$, $\beta_2 > 0$, a length of
data $N$ and the number of quantization levels $M$, 
% in $[-\th1\kappa, \; \th1\kappa]$, 
where $1-\beta_1-\beta_2  > 0$, 
%$M \gg B^\star$, $B^\star$ is defined by \rref{eq:variapp}, 
and 
$\sigma_{\reg}^2 N- n\sqrt{\frac{N}{\beta_1}}
\left(
\sqrt{\eta} + (n-1)\sigma_{\reg}^2
\right)
> 0$, 
%$\sigma_{\reg}: = \frac{1}{3}\kappa^2$,
the following inequality asymptotically holds at $\Delta y \rightarrow 0$:
\begin{eqnarray}
&&
{\mbox{\rm{Prob}}}\left(
\|\Delta \tilde{E}\|_\infty \leq
\epsilon_1\epsilon_2
\right)
\geq 1-\beta_1-\beta_2
\label{eq:theoremnormal}
\\
&&
\epsilon_1: = 
\frac{1}
{\sigma_{\reg}^2 N-
n\sqrt{\frac{N}{\beta_1}}
\left(
\sqrt{\eta} + (n-1)\sigma_{\reg}^2
\right)
},
\;
\epsilon_2 : = 
\frac{1}{M}
\sqrt{\frac{1}{12}\tilde{\theta}_1^2 D^3}
\sqrt{\frac{nN}{\beta_2}}.
%\epsilon_2 : = 
%\frac{A^{\frac{1}{2}}\kappa^2}{M-B^\star}
%\sqrt{\frac{N}{\beta_2}}.
\label{eq:1epsilon2}
\end{eqnarray}
%where $A$ is defined by \rref{eq:variapp}.
\end{theorem}
The proof is in Appendix~\ref{appendex:proofs}.  

From this theorem, we know that the convergence rate of the error
term 
$\|\Delta \tilde{E}\|_\infty$ 
%$|(\Delta \tilde{E})_1|$ 
has an order of $M^{-1}$ for sufficiently large $M$ and
of $N^{-\frac{1}{2}}$.
Approximately, the total amount of information in the quantized
output transmitted from identified systems to the observers is approximately
$N\log_2 M  = : {\cal K}$ using binary coding. Therefore, subject to a constraint of such a
total amount of information, it is known that a large $M$ is preferable
to a large $N$ to reduce the estimation error by observing: 
\[
M^{-1}N^{-\frac{1}{2}} = M^{-1}\left(\frac{\cal K}{\log_2 M}\right)^{-\frac{1}{2}}
 = 
{\cal K}^{-\frac{1}{2}} M^{-1} \left(\log_2 M\right)^{\frac{1}{2}}
\;\mathop{\longrightarrow}^{M\rightarrow \infty}\;
0.
\]
Of course, this is valid only for the error term 
$\|\Delta \tilde{E}\|_\infty$
%$(\Delta \tilde{E})_1$ 
and the situation is different for the noise error term $\Delta W$.
We introduce the result for $\Delta \tilde{W}$ in the following proposition.
\begin{proposition}
\label{lemms:noiseterm}{\rm \cite{Tsumura:CDC99}}
Assume that $\tphi$ satisfies Assumption~\ref{assumption:0} and \ref{assumption:smooth}
%Assume that $\reg(\i)$ satisfies Assumption~\ref{assumption:dist}
and $w(\i)$ is i.i.d. random variable with
%${\mbox{\rm E}}[u(i)] = 0$, 
${\mbox{\rm V}}[\reg(\i)] =  \sigma_{\reg}^2$,
and ${\mbox{\rm V}}[w(\i)] =  \sigma_w^2$, respectively.
Then, for reliability indices $\beta_1$, $\beta_2 > 0$, and a length of
data $N$, where $1-\beta_1-\beta_2  > 0$, and
\break
$\sigma_{\reg}^2 N-
n\sqrt{\frac{N}{\beta_1}}
\left(
\sqrt{\eta} + (n-1)\sigma_{\reg}^2
\right)
> 0$,
the following inequality is satisfied.
\begin{eqnarray}
&&
{\mbox{\rm{Prob}}}
\left(
\|\Delta \tilde{W}\|_\infty \leq
\epsilon_1\epsilon_2
\right)
\geq 1-\beta_1-\beta_2
\label{eq:theoremnormal2}
\\
&&
\epsilon_1: = 
\frac{1}
{\sigma_{\reg}^2 N-
n\sqrt{\frac{N}{\beta_1}}
\left(
\sqrt{\eta} + (n-1)\sigma_{\reg}^2
\right)
},
\;
\epsilon_2 : =  \sigma_{\reg}\sigma_w\sqrt{\frac{nN}{\beta_2}}
\label{eq:1epsilon3}
\end{eqnarray}
\end{proposition}
This result shows that a large $N$ is preferable for reducing $\Delta \tilde{W}$.  
By combining Theorem~\ref{lemms:quantterm} and
Proposition~\ref{lemms:noiseterm}, it can be seen that there exists
a trade-off between $\Delta \tilde{E}$ and $\Delta \tilde{W}$ 
(also $\Delta {E}$ and $\Delta {W}$)
for reducing the total
identification error subject to the constraint on the amount of information
transmitted from the identified systems to the estimators.

%#!platex paper

\section{Conclusion}

In this paper, we show that the optimal quantizers for system
identification can be derived analytically and 
their essential properties investigated with a simple FIR model.
The results of this paper are summarized as follows:

\begin{itemize}
\item[(1)] 
General cases of the distribution of regressor vectors can be treated
for high resolution quantizers by introducing the concept of the
density of quantization subsections (Section~\ref{section:main}). 

\item[(2)] 
The optimization problems in (1) are reduced to minimizations of 
functionals and the solutions can be found by solving Euler--Lagrange
differential equations (Section~\ref{section:main}). 

\item[(3)] 
When the regressor vector has a form of uniform distribution, the
optimal quantization problem is reduced to a recursive minimization,
which can be solved by a dynamic programming (Section~\ref{sec:prev}).  

\item[(4)] 
In usual situation, the optimal quantizer is coarse near the origin of
the output signals and tends to be dense away from the origin
(Section~\ref{section:main} and Section~\ref{sec:prev}).  

\item[(5)] 
Subject to a limitation on the total quantity of information in the
quantized I/O data, there exists a trade-off between the magnitudes of
the quantization error and noise error (Section~\ref{sec:eval}).  
\end{itemize}

In this paper, we restrict the model to a SISO FIR model. 
For more realistic situations, we must extend the results to: 
a) ARX models, or MIMO systems, 
b) quantized input signal,
and 
c) online system identification and adaptive control. 
These remain for future study. 

\noindent{\bf\Large Acknowledgement}

\noindent
The author expresses deep gratitude for discussion and help 
to Professor Jan Maciejowski, University of Cambridge.

\setlength{\baselineskip}{11pt}
\bibliographystyle{plain}
\bibliography{control}

\setlength{\baselineskip}{16.6pt}

\appendix

%#!platex paper

\section{Appendix}
\label{appendex:proofs}

\noindent
{\bf Slutsky's Theorem} (e.g. \cite{Grimmett:2001})

For sequences of stochastic variables $X(i)$, $Y(i)$, 
assume that ${\rm plim}_{i\rightarrow \infty}[X(i)]$ and ${\rm
plim}_{i\rightarrow \infty}[Y(i)]$ converge to constants.  
Then, 
\[
\mathop{\rm plim}_{i\rightarrow \infty}[X(i)^{-1}Y(i)] 
=
\left(
\mathop{\rm plim}_{i\rightarrow \infty}[X(i)]
\right)^{-1}
\mathop{\rm plim}_{i\rightarrow \infty}[Y(i)] 
\]
holds. 

\noindent
{\bf Proof of Lemma~\ref{lemma:seisai}}

The outline of the proof is similar to that of Lemma~\ref{lemma:vari0}
 and we evaluate the value of:
$
\E
\left[
\rregh
{e}(\rregi)
\rregj
{e}(\rregk)
\right]
$
for possible cases in \rref{eq:conv10}:
$a\not = b \not = c \not = d$,
$a = b \not = c \not = d$,
$a = b \not = c  = d$,
$a = b = c = d$,
and
$a = c \not = b = d$ (the other possible cases in \rref{eq:conv10} are
 essentially identical to these cases). 

Let $\Ih$, $\Ii$, $\Ij$, or $\Ik$ be 
a quantized subsection of the axis of $\rregh$, $\rregi$, $\rregj$, or
$\rregk$, respectively, and consider a subset 
$\Ih \times \Ii \times \Ij \times \Ik$
in the space of $\tphi$.
Moreover, let ${\rregh}'$, ${\rregi}'$, ${\rregj}'$,
and ${\rregk}'$ be the quantized values, which are midpoints of
$\Ih$,
$\Ii$,
$\Ij$, and $\Ik$, respectively.  
The partial integral of 
$
\E
\left[
\rregh
{e}(\rregi)
\rregj
{e}(\rregk)
\right]
$
restricted to this subset is 
\begin{eqnarray*}
\int_{\II}
\rregh
{e}(\rregi)
\rregj
{e}(\rregk)
f(\rregh, 
\rregi, 
\rregj,
\rregk
)
d\rregh
d\rregi
d\rregj
d\rregk. 
\end{eqnarray*}
Let $2\Delta \tilde{\phi}$ be the width of the largest side of the
possible hyperrectangular parallelepiped regions in $\tilde{\phi}$ 
given by quantization,
then, when $a \not =  b \not =  c \not =  d$:
%then, in the case of $h \not =  i \not =  j \not =  k$,
\begin{eqnarray}
\lefteqn{
\int_{\II}
\rregh
{e}(\rregi)
\rregj
{e}(\rregk)
f(\rregh, 
\rregi, 
\rregj,
\rregk
)
d\rregh
d\rregi
d\rregj
d\rregk
}
\nonumber
\\
& = &
\int_{\II}
\rregh
{e}(\rregi)
\rregj
{e}(\rregk)
\nonumber
\\
&&
\mbox{}
\times
\left(
\den
\right)
d\rregh
d\rregi
d\rregj
d\rregk
\nonumber
\\
%& = &
%\int_{\rregh \in \Ih, \rregj \in \Ij}
%\rregh
%\rregj
%\left(
%\int_{\rregi \in \Ii}
%({\rregi}'-\rregi)
%(\H_b+\K_b ({\rregi}'-\rregi) + O(({\rregi}'-\rregi)^2))
%d\rregi
%\right.
%\nonumber
%\\
%&&
%\left.
%\mbox{}
%\times
%\int_{\rregk \in \Ik}
%({\rregk}'-\rregk)
%(\H_d+\K_d ({\rregk}'-\rregk) + O(({\rregk}'-\rregk)^2))
%d\rregk
%\right)
%\nonumber
%\\
%&&
%\mbox{}
%\times
%(\H_a+\K_a ({\rregh}'-\rregh) + O(({\rregh}'-\rregh)^2))
%(\H_c+\K_c ({\rregj}'-\rregj) + O(({\rregj}'-\rregj)^2))
%d\rregh
%d\rregj
%\nonumber
%\\
& = &
{\rregh}'
{\rregj}'
\H_{bd}
\frac{2^4}{3^2}
\Delta \tphi^8 + O(\Delta \tphi^9),
%\\
%& = &
%{\rregh}'
%{\rregj}'
%\H_a\H_c
%\K_b\K_d
%\frac{2^4}{3^2}
%\Delta \tphi^8 + O(\Delta \tphi^9),
\label{eq:deltay1}
\end{eqnarray}
%
%\noindent
and similarly, 
when $a = b \not =  c \not = d$:
\begin{eqnarray}
\int_{\II}
\rregh
{e}(\rregh)
\rregj
{e}(\rregk)
f(\rregh, 
\rregi, 
\rregj,
\rregk
)
d\rregh
d\rregi
d\rregj
d\rregk
 = 
({\rregh}'
{\rregj}'
\H_{ad}+{\rregj}'\H_{d})
\frac{2^4}{3^2}
\Delta \tphi^8 + O(\Delta \tphi^9),
\label{eq:deltay22}
\end{eqnarray}
%\begin{eqnarray}
%\int_{\IIII}
%\rregh
%{e}(\rregh)
%\rregj
%{e}(\rregk)
%f(\rregh, 
%%\rregi, 
%\rregj,
%\rregk
%)
%d\rregh^2
%%d\rregi
%d\rregj
%d\rregk
% = 
%%{\rregh}'
%%%{\rregi}'
%%{\rregj}'
%%%{\rregk}'
%(2{\rregh}'\K_a-\H_a)
%{\rregj}'
%\H_a
%\H_c\K_d
%%\H_b\H_d
%%\K_a\K_c
%\frac{2^4}{3^2}
%\Delta \tphi^8 + O(\Delta \tphi^9),
%\label{eq:deltay22}
%\end{eqnarray}
and when $a = b \not =  c = d$:
\begin{eqnarray}
\lefteqn{
\int_{\II}
\rregh
{e}(\rregh)
\rregj
{e}(\rregj)
f(\rregh, 
\rregi, 
\rregj,
\rregk
)
d\rregh
d\rregi
d\rregj
d\rregk
}
\nonumber
\\
&=&
(
{\rregh}'{\rregj}'\H_{ac}
+
{\rregh}'\H_a
+
{\rregj}'\H_c
+
\H_{0}
)
\frac{2^4}{3^2}
\Delta \tphi^8 + O(\Delta \tphi^9).
\label{eq:deltay2}
\end{eqnarray}

%\begin{eqnarray}
%\int_{\III}
%\rregh
%{e}(\rregh)
%\rregj
%{e}(\rregj)
%f(\rregh, 
%%\rregi, 
%\rregj
%%,%\rregk
%)
%d\rregh^2
%%d\rregi
%d\rregj^2
%%d\rregk
% = 
%(2{\rregh}'\K_a-\H_a)
%(2{\rregj}'\K_c-\H_c)
%\H_a
%\H_c
%%{\rregh}'
%%%{\rregi}'
%%{\rregj}'
%%%{\rregk}'
%%(2\H_a\K_a-\H_a^2)
%%(2\H_c\K_c-\H_c^2)
%%\H_b\H_d
%%\K_a\K_c
%\frac{2^4}{3^2}
%\Delta \tphi^8 + O(\Delta \tphi^9).
%\label{eq:deltay2}
%\end{eqnarray}

\noindent
Alternatively, when $a = b = c = d$:
\begin{eqnarray}
\int_{\II}
\rregh
{e}(\rregh)
\rregh
{e}(\rregh)
f(\rregh, 
\rregi, 
\rregj,
\rregk
)
d\rregh
d\rregi
d\rregj
d\rregk
=
{\rregh}'^2
\H_0
\frac{2^4}{3}
\Delta \tphi^6
+
O(\Delta \tphi^7)
\label{eq:deltay3}
\end{eqnarray}
%\begin{eqnarray}
%\lefteqn{
%\int_{\VV}
%\rregh
%{e}(\rregh)
%\rregh
%{e}(\rregh)
%f(\rregh
%%, 
%%\rregi, 
%%\rregj,
%%\rregk
%)
%d\rregh^4
%%d\rregi
%%d\rregj
%%d\rregk
%}
%\nonumber
%\\
%& = &
%\int_{
%\VVI
%%\rregi \in \Ii, \rregj \in \Ij, \rregk \in \Ik
%}
%\left(
%\int_{\rregh \in \Ih}
%{\rregh}^2
%e^2(\rregh)
%(\H_a+\K_a ({\rregh}'-\rregh) + O(({\rregh}'-\rregh)^2))^4
%d\rregh
%\right)
%%\nonumber\\
%%&&
%%\mbox{}
%%\times
%%(\H_b+\K_b ({\rregi}'-\rregi) + O(({\rregi}'-\rregi)^2))
%%(\H_c+\K_c ({\rregj}'-\rregj) + O(({\rregj}'-\rregj)^2))
%%\nonumber
%%\\
%%&&
%%\mbox{}
%%\times
%%(\H_d+\K_d ({\rregk}'-\rregk) + O(({\rregk}'-\rregk)^2))
%d\rregh^3
%%d\rregj
%%d\rregk
%\nonumber
%\\
%& = &
%{\rregh}'^2
%%{\rregi}
%%{\rregj}
%%{\rregk}
%\H_a^4
%%\H_b
%%\H_c
%%\H_d
%\frac{2^4}{3}
%\Delta \tphi^6
%+
%O(\Delta \tphi^7)
%\label{eq:deltay3}
%\end{eqnarray}
and similarly, 
when $a  = c \not =  b = d$:
\begin{eqnarray}
\int_{\II}
\rregh
{e}(\rregi)
\rregh
{e}(\rregi)
f(\rregh, 
\rregi, 
\rregj,
\rregk
)
d\rregh
d\rregi
d\rregj
d\rregk
 = 
{\rregh}'^2
\H_0
\frac{2^4}{3}
\Delta \tphi^6
+
O(\Delta \tphi^7).
\label{eq:deltay4}
\end{eqnarray}

%\begin{eqnarray}
%\int_{\VVII}
%\rregh
%{e}(\rregi)
%\rregh
%{e}(\rregi)
%f(\rregh, 
%\rregi
%%, 
%%\rregj,
%%\rregk
%)
%d\rregh
%d\rregi
%d\rregh
%d\rregi
% = 
%{\rregh}'^2
%%{\rregj}
%%{\rregk}
%\H_a^2
%\H_b^2
%%\H_c
%%\H_d
%\frac{2^4}{3}
%\Delta \tphi^6
%+
%O(\Delta \tphi^7).
%\label{eq:deltay4}
%\end{eqnarray}

The above show that, when $\Delta \tphi$ $\rightarrow$ 0, 
the rate of convergence of \rref{eq:deltay1} -- \rref{eq:deltay2} to 0 is faster than 
that of \rref{eq:deltay3} and \rref{eq:deltay4}. 
Therefore, we have the following: 
\begin{equation*}
\E
\left[
\left(
\sum_{\i = 1}^N
\tilde{\phi}_{k}(\i)
{e}(\reg(\i))
\right)^2
\right]
\;
\mathop{\rightarrow}_{\Delta y_{\rm max} \rightarrow 0}
\;
N\,
\E
\left[
\tilde{\phi}_{k}^2
{e}^2(\reg)
\right].
\vspace*{-12pt}
\end{equation*}
\hfill{$\square$}

\noindent
{\bf Derivation of eq.~\rref{eq:vari1001app}}
\begin{eqnarray*}
\rref{eq:vari1001}/N
&\stackrel{(i)}{=}&
\int
\sigma^2(\reg)
e^2(\reg)
f(\reg)
%\cdot
d\reg
%}
%\nonumber
%\\
%& = & 
 = 
\sum_{j}
\int_{\I_j}
\sigma^2(\reg)
(\qy_\j - \th1\reg)^2
%\cdot e^2(\reg) 
f(\reg)
d\reg
%\\
%& = &
%\sum_i \int_{I_i} e^2 \cdot \tilde{\phi}_1^2 f(\tilde{\phi}_1) d\tilde{\phi}_1
%\\
\nonumber
\\& = &
\sum_{j} \int_{\regjp - \frac{1}{2}g_j^{-1}}^{\regjp + \frac{1}{2}g_j^{-1}} (\th1\regjp-\th1 \reg)^2 \cdot \sigma^2(\reg) f(\reg) d\reg
%+ O(\Delta \tilde{\phi}^3)
\nonumber
\\
%& = &
%\sim
&
\stackrel{\rm (i)}{=}
& 
\sum_{j} \int_{\regjp - \frac{1}{2}g_j^{-1}}^{\regjp + \frac{1}{2}g_j^{-1}} (\th1\regjp-\th1 \reg)^2 \sigma^2(\regjp) f_j d\reg + O(\Delta \tilde{\phi}^3)
\nonumber
\\
%& = & 
%%\sim 
%\th1^2
%\sum_{j}
%\frac{1}{12}g_j^{-3}
%\sigma^2(\regjp) f_j
%+
%O(\Delta \tilde{\phi}^3)
%\nonumber
%\\
& = & 
% = 
%\sim
\th1^2
\sum_j
\int_{\regjp - \frac{1}{2}g_j^{-1}}^{\regjp + \frac{1}{2}g_j^{-1}}
\frac{1}{12}g_j^{-2}
\sigma^2(\regjp) f_j d\reg
+
O(\Delta \tilde{\phi}^3)
\nonumber
\\
& 
\stackrel{\rm (ii)}{=}
& 
\th1^2
\sum_j
\int_{\regjp - \frac{1}{2}g_j^{-1}}^{\regjp + \frac{1}{2}g_j^{-1}}
\frac{1}{12}g(\tilde{\phi}_1)^{-2}
\sigma^2(\regjp) f_j d\tilde{\phi}_1
+
O(\Delta \tilde{\phi})
%O(\Delta \tilde{\phi}^3)
\nonumber
\\
&
\stackrel{\rm (iii)}{=}
& 
%= 
\th1^2
\sum_j
\int_{\regjp - \frac{1}{2}g_j^{-1}}^{\regjp + \frac{1}{2}g_j^{-1}}
\frac{1}{12}g(\tilde{\phi}_1)^{-2}
\sigma^2(\tilde{\phi}_{1}) f(\tilde{\phi}_1) d\tilde{\phi}_1
+
O(\Delta \tilde{\phi})
\\
&=&
%& 
%= 
\th1^2
\int
\frac{1}{12}g(\tilde{\phi}_1)^{-2}
\sigma^2(\tilde{\phi}_{1}) f(\tilde{\phi}_1) d\tilde{\phi}_1
+
O(\Delta \tilde{\phi}),
%O(\Delta \tilde{\phi}^3),
%\label{eq:vari1001app}
\end{eqnarray*}
where 
$\regjp$ is the midpoint of $\I_j$, 
(i) is by Assumption~\ref{assumption:smooth}.1, 
(ii) is by Assumption~\ref{assumption:smooth}.2, 
and 
(iii) is by Assumption~\ref{assumption:smooth}.1.
%where $\regjp$ is an assigned value for $\I_j$ satisfying 
%$\regjp \in \I_j$.  
\hfill{$\square$}

\noindent
{\bf Proof of Theorem~\ref{theorem:fixed-rate}}

\noindent
The optimal solution can be given 
by using a similar technique to that in
\cite{Bennett:BSTJ48,Lloyd:TIT82}. 
With the calculus of variations, 
the following Euler--Lagrange equation:
\[
\frac{d}{d\tilde{\phi}_1} \left(
\frac{\partial {\cal F}}{\partial g}
\right)
-
\frac{\partial {\cal F}}{\partial G}
 = 0,
\] 
where
\[
G(\tilde{\phi}_1) : =  \int_{-\infty}^{\reg} g(\reg) d\reg,
%G(\tilde{\phi}_1) : =  \int_{\reg^{\min}}^{\reg} g(\reg) d\reg,
\]
gives a differential equation:
\begin{equation*}
\frac{d}{d\tilde{\phi}_1}
\left(
-2 g(\tilde{\phi}_1)^{-3}
\sigma^2(\tilde{\phi}_1) f(\tilde{\phi}_1)
\right)  = 0, 
\end{equation*}
and the solution is:
\begin{equation*}
g(\tilde{\phi}_1) = K
\sigma^{\frac{2}{3}}(\tilde{\phi}_1)
f^{\frac{1}{3}}(\tilde{\phi}_1), \;\; K: \mbox{\rm constant}.
\end{equation*}
The constant number $K$ is directly calculated by the condition \rref{eq:opt-n-const},
and the value of the objective function is derived as follows. 
\begin{eqnarray*}
\int {\cal F}(g_{\rm f}(\tilde{\phi}_1)
) d\tilde{\phi}_1
& = & 
\int  
\frac{1}{12}\th1^2
(K \sigma^{\frac{2}{3}}(\tilde{\phi}_1)
f^{\frac{1}{3}}(\tilde{\phi}_1))^{-2} \sigma^2(\tilde{\phi}_1) f(\tilde{\phi}_1)
d\reg
\nonumber
\\
& = & 
\int  
\frac{1}{12}\th1^2
K^{-2}
\sigma^{\frac{2}{3}}(\reg)
f^{\frac{1}{3}}(\reg)
d\reg
 =  
\frac{1}{12}\th1^2 K^{-2}D
 = 
\frac{1}{12}\th1^2
D^3 M^{-2}
\vspace*{-12pt}
\end{eqnarray*}
\hfill$\square$\par

\noindent
{\bf Proof of Theorem~\ref{theorem:variable}}

\noindent
We use a similar technique to that in \cite{Gish:TIT68,Berger:TIT72}.
Let $\lambda$ be a Lagrange multiplier and consider the minimization
 of the following quantity. 
\begin{eqnarray*}
%&&
%\lefteqn{
\int {\cal F}(g(\tilde{\phi}_1)
%, \; G(\tilde{\phi}_1)
) d\tilde{\phi}_1
+ \lambda H_{\rm d}(f,g)
%}
%\nonumber
%\\
& = &
\int
\frac{1}{12}\th1^2
\left(
\frac{1}{g(\tilde{\phi}_1)}
\right)^2
\sigma^2(\reg)
f(\tilde{\phi}_1) 
%\nonumber
%\\
%&&\mbox{}
- \lambda
f(\tilde{\phi}_1) \log \left(g^{-1}(\tilde{\phi}_1) \right)
d\tilde{\phi}_1
+ \lambda H_{\rm d}(f)
\nonumber
\\
& = &
\int
\frac{1}{12}\th1^2
f(\tilde{\phi}_1) 
\left(
g^{-2}(\tilde{\phi}_1)
\sigma^2(\reg)
+ \lambda
\log g(\tilde{\phi}_1)
\right)
d\tilde{\phi}_1
%\nonumber\\
%&&\mbox{}
+ \lambda H(f)
\end{eqnarray*}
By applying the calculus of variations, 
%Euler--Lagrange differential equation, 
we obtain:
\begin{eqnarray*}
\frac{\partial}{\partial g}
\left(
g^{-2}
\sigma^2(\reg)
+ \lambda
\log g
\right) 
& = &
-2g^{-3}
\sigma^2(\reg)
+ \lambda
g^{-1}
%\nonumber\\
%& = &
 = 
{\mbox{ constant }}.
\end{eqnarray*}
Fix the constant to be zero, then, 
\begin{eqnarray*}
g = \left(\frac{2}{\lambda}\right)^{\frac{1}{2}} \sigma(\reg),
\end{eqnarray*}
and by substituting this for $H(f, \, g)$, we obtain:
\begin{eqnarray*}
H(f, \, g) 
%& = &
 = 
\int -f \log g^{-1} f d\reg
%\nonumber
%\\
 = 
%& = &
\log \left(\frac{2}{\lambda}\right)^{\frac{1}{2}} + \int -f \log \frac{f}{\sigma(\reg)} 
d\reg 
%\nonumber
%\\
%& = &
 = 
\log M. 
\end{eqnarray*}
Therefore, 
\begin{eqnarray*}
\left(\frac{2}{\lambda}\right)^{\frac{1}{2}}  = 
\exp 
\left(
\int f \log \frac{f}{\sigma(\reg)} d\reg + \log M
\right),
\end{eqnarray*}
and \rref{eq:gentroopt} is derived.
By substituting $\goptv$ for the objective integral, 
the following is derived. 
\begin{eqnarray*}
\int 
\frac{1}{12}\th1^2
g_{\rm v}^{-2}(\reg)\sigma^2(\reg) f(\reg) d\reg
& = & \frac{1}{12}\th1^2
\frac{\lambda}{2}
%\nonumber
%\\
%& = &
 = 
\frac{1}{12}\th1^2
K^{-2} M^{-2}
\end{eqnarray*}
\hfill$\square$\par

\noindent
{\bf Proof of Lemma~\ref{lemma:vari0}}

\noindent
The left hand side of \rref{eq:tenkai} is extended:
\begin{eqnarray}
%\lefteqn{
\E
\left[
\left(
\sum_{\i = 1}^N
\tilde{\phi}_{k}(\i)
{e}(\reg(\i))
\right)^2
\right]
%}
%\nonumber
%\\
& = &
\E
\left[
\sum_{\i = 1}^N
\regk^2(\i)
{e}^2(\reg(\i))
\right]
+
2\E
\left[
\sum_{\i = 1}^{N-1}
\regk(\i)
{e}(\reg(\i))
\regk(\i+1)
{e}(\reg(\i+1))
\right]
+
\cdots
\nonumber
\\
& = &
N\,
\E
\left[
\regk^2
{e}^2(\reg)
\right]
+
2(N-1)
\E
\left[
\regk
{e}(\reg)
\tilde{\phi}_{k+1}
{e}(\tilde{\phi}_{2})
\right]
+ \cdots .
\label{eq:tenkaiappen}
\end{eqnarray}
In \rref{eq:tenkaiappen}, 
terms of the form 
$
\E
\left[
\rregh
{e}(\rregi)
\rregj
{e}(\rregk)
\right]
$
appear and in general, when \rref{eq:tenkaijyoken} and 
\rref{eq:tenkaijyoken2} are satisfied, 
$
\E
\left[
\rregh
{e}(\rregi)
\rregj
{e}(\rregk)
\right]
$
can be calculated according to the combinations of $a$, $b$, $c$ and $d$ as follows.  

%\noindent
When $a \not =  b \not =  c \not =  d$,
%In the case of $h \not =  i \not =  j \not =  k$,
\begin{eqnarray*}
&&
\E
\left[
\rregh
{e}(\rregi)
\rregj
{e}(\rregk)
\right]
=
\int
\rregh
{e}(\rregi)
\rregj
{e}(\rregk)
f(\rregh, 
\rregi, 
\rregj,
\rregk
)
d\rregh
d\rregi
d\rregj
d\rregk
\nonumber
\\
&&
= 
\int
{e}(\rregi)
\rregj
{e}(\rregk)
\left(
\int
\rregh
f(\rregh, 
\rregi, 
\rregj,
\rregk
)
d\rregh
\right)
d\rregi
d\rregj
d\rregk
%\nonumber
%\\
%& 
\stackrel{\rref{eq:tenkaijyoken}}{=}
%&
\int
{e}(\rregi)
\rregj
{e}(\rregk)
\times 0
\times
d\rregi
d\rregj
d\rregk
%\nonumber
%\\
%& = &
 = 0,
\end{eqnarray*}
and similarly, when $a = b \not =  c \not =  d$,
%and also in the case of $h = i \not =  j \not =  k$,
\begin{eqnarray*}
&&
\E
\left[
\rregh
{e}(\rregi)
\rregj
{e}(\rregk)
\right]
=
%& = &
\int
\rregh
{e}(\rregh)
\rregj
{e}(\rregk)
f(\rregh, 
\rregj,
\rregk
)
d\rregh
d\rregj
d\rregk
\nonumber
\\
&& =
%& = &
\int
\rregh
{e}(\rregh)
{e}(\rregk)
\left(
\int
\rregj
f(\rregh, 
\rregi, 
\rregj,
\rregk
)
d\rregj
\right)
d\rregh
d\rregk
%\nonumber
%\\
%& 
\stackrel{\rref{eq:tenkaijyoken}}{=} 
%&
%& = &
\int
\rregh
{e}(\rregh)
{e}(\rregk)
\times 0
\times
d\rregh
d\rregk
%\nonumber
%\\
%& = &
 = 0,
\end{eqnarray*}
%
%\noindent
and when $a = b \not = c = d$,
\begin{eqnarray*}
&&
\E
\left[
\rregh
{e}(\rregi)
\rregj
{e}(\rregk)
\right]
=
%& = &
\int
\rregh
{e}(\rregh)
\rregj
{e}(\rregj)
f(\rregh, 
\rregj
%,
%\rregk
)
d\rregh
d\rregj
%d\rregh^2
%d\rregj^2
%d\rregk
\nonumber
\\
&&=
%& = &
\int
\rregh
{e}(\rregh)
\left(
\int
\rregj
{e}(\rregj)
f(\rregh, 
%\rregi, 
\rregj
%,
%\rregk
)
d\rregj
\right)
d\rregh
%d\rregh^2
%d\rregj
%\nonumber
%\\
%& 
\stackrel{\rref{eq:tenkaijyoken2}(\rm i.e. \rref{eq:sim})}{=} 
%&
%& = &
\int
\rregh
{e}(\rregh)
%{e}(\rregk)
\times 0
\times
d\rregh
%d\rregh^2
%d\rregj
%\nonumber
%\\
%& = &
 = 0.
\end{eqnarray*}
%
%\rref{eq:tenkaijyoken2}
On the other hand, 
there is no term when 
$a = c \not =  b \not =  d$ or $b = d \not =  a \not =  c$ 
in \rref{eq:tenkaiappen}. 
Finally, when $a = c$, $b = d$,
\begin{eqnarray*}
\E
\left[
\rregh
{e}(\rregi)
\rregj
{e}(\rregk)
\right]
& = &
\E
\left[
\rregh^2
{e}^2(\rregi)
\right].
\end{eqnarray*}
The other cases are essentially equivalent to one of the above cases (for
example, $a = d \not =  b \not =  c$ is equivalent to $a = b
\not =  c \not =  d$).  

From the above, it follows that:
\begin{eqnarray*}
\E
\left[
\left(
\sum_{\i = 1}^N
\regk(\i)
{e}(\reg(\i))
\right)^2
\right]
& = &
N
\E
\left[
\regk^2
{e}^2(\reg)
\right].
\end{eqnarray*}
\hfill{$\square$}

\noindent
{\bf Proof of Proposition~\ref{th:mainprop}}

\noindent
Consider $\S_1 = (0, \; d_1]$ (equivalently $\I_1$ on $\reg$)
and 
$\S_2 = (d_1, \; d_2]$ (equivalently $\I_2$ on $\reg$)
where their boundaries $d_1$, $d_2$ have the relationship:
\begin{equation}
d_1 = r_1 d_2, \; r_1 \in [0, 1]
\label{eq:rdef1}
\end{equation}
with an appropriate ratio $r_1$. 
The quantized values $\qyone$ and $\qytwo$ for the subsections
$\S_1$ on $y$ (or $\I_1$ on $\reg$) and $\S_2$ (or $\I_2$) satisfying the bias-free condition:
\[
\E_{\I_j}\left[\reg \cdot e(\reg)\right] = 0, \; j = 1, 2
\]
are given as follows.
Let $\qyone = \frac{d_1}{2}+h_1$, where $h_1$ is an offset from the
center of $\S_1$, then,
%aaaaaaaaaaaaaaaaaaa
\begin{eqnarray*}
\E_{\I_1}\left[\reg \cdot e(\reg)\right]
 = 
\int_{-k_1}^{k_1} -\left(\frac{r_1d_2}{2}+z\right)
(z-h_1)
\frac{1}{2\kappaY}dz
 =  - \frac{1}{2\kappaY} \left(\frac{2}{3}k_1^3 - r_1d_2h_1k_1\right),
\; k_1 : =  \frac{d_1}{2},
\nonumber
\end{eqnarray*}
and therefore,
\begin{equation*}
h_1 = \frac{2}{3}\frac{k_1^2}{r_1d_2}
 =  \frac{1}{6}r_1d_2.
\end{equation*}
Similarly, let $\qytwo : =  \frac{(1+r_1)d_2}{2} + h_2$, where $h_2$ is
the offset, then,
\begin{eqnarray*}
&&
\E_{\I_2}\left[\reg \cdot e(\reg)\right]
 = 
\int_{-k_2}^{k_2}
-
\left(
\frac{d_2+r_1d_2}{2}+z
\right)(z-h_2)\frac{1}{2\kappaY}dz
 = -\frac{1}{2\kappaY}
\left(
\frac{2}{3}k_2^3 - (d_2+r_1d_2)h_2 k_2
\right), 
\nonumber
\\
&&
k_2 : = \frac{d_2(1-r_1)}{2},
\nonumber
\end{eqnarray*}
and therefore,
\begin{equation*}
h_2 = \frac{2}{3} k_2^2 \frac{1}{d_2(1+r_1)}
 =  \frac{1}{6}\frac{(1-r_1)^2}{(1+r_1)}d_2.
\end{equation*}
%Hereafter in this section, the quantized values $\qy_\j$ are selected
%with these offsets.  

By using these $\qyone$ and $\qytwo$, the variances of
$\reg e(\reg)$ in each subsection can be calculated as follows.
Let $\V_{\I_j}\left[\sigma(\reg) e(\reg)\right]$ denote the quantity:
\begin{equation}
\V_{\I_j}\left[\sigma(\reg) e(\reg)\right]
: =  \int_{\I_j} 
\sigma^2(\reg)e^2(\reg) 
f(\reg)d\reg,
\label{eq:limv}
\end{equation}
where
\[
\sigma^2(\reg) = \tilde{\phi_1}^2 + \frac{1}{3}\kappaY^2(n-1),
\]
then,
for even $\Mo$,
\begin{eqnarray*}
\V_{\I_1}\left[
\sigma(\reg) e(\reg)
\right]
&=& 
\int_{-k_1}^{k_1}
\left\{
\left(
\frac{r_1d_2}{2}+z
\right)^2
+
\frac{1}{3}\kappaY^2(n-1)
\right\}
(z-h_1)^2
\frac{1}{2\kappaY}
dz
\\
&=&
\frac{1}{2160}
\frac{1}{2\kappaY}
d_2^5
\left(
32 r_1^5
\right)
+
\frac{1}{27}
\frac{1}{2\kappaY}
\kappaY^2(n-1)d_2^3r_1^3
\nonumber
\end{eqnarray*}
and similarly
\begin{eqnarray*}
\V_{\I_2}\left[
\sigma(\reg) e(\reg)
\right]
& = &
\int_{-k_2}^{k_2}
\left\{
\left(
\frac{d_2(1+r_1)}{2}+z
\right)^2
+\frac{1}{3}\kappaY^2(n-1)
\right\}
(z-h_2)^2
\frac{1}{2\kappaY}
dz
\nonumber
\\
& = &
\frac{1}{2160}
\frac{1}{2\kappaY}
d_2^5
\left\{
-18(1-r_1)^5 + 45 (1+r_1)^2(1-r_1)^3 + 5(1-r_1)^7(1+r_1)^{-2}
\right\}
\nonumber
\\
&& \mbox{}+
\frac{1}{108}\frac{1}{2\kappaY}\kappaY^2(n-1)d_2^3
\left\{
3(1-r_1)^3+\frac{(1-r_1)^5}{(1+r_1)^2}
\right\}
%\label{eq:V2}
\end{eqnarray*}
Therefore, the sum of 
$\V_{\I_1}\left[\sigma(\reg)e(\reg)\right]$ and
$\V_{\I_2}\left[\sigma(\reg)e(\reg)\right]$ is:
\begin{eqnarray}
\V_{\I_1}\left[\sigma(\reg)e(\reg)\right]
+
\V_{\I_2}\left[\sigma(\reg)e(\reg)\right] 
&=& 
\frac{1}{2160}
\frac{1}{2\kappaY}
\left(
d_2^5 \psi(r_1; 32)
+
20\kappaY^2(n-1)d_2^3
\xi(r_1; 4)
\right),
\nonumber
\\
\psi(r_1;32) &:=& 
32 r_1^5 -18(1-r_1)^5 + 45 (1+r_1)^2(1-r_1)^3 + 5(1-r_1)^7(1+r_1)^{-2},
\nonumber
\\
\xi(r_1;4) &:=& 
4r_1^3+3(1-r_1)^3+\frac{(1-r_1)^5}{(1+r_1)^2}.
\label{eq:f1}
\end{eqnarray}
The minimizer $r_1^o$ of this sum is given by:
\begin{eqnarray*}
r_1^o & = & \arg \min_{r_1 \in [0, 1]} 
\left(
d_2^5 \psi(r_1; 32)
+
20\kappaY^2(n-1)d_2^3
\xi(r_1; 4)
\right)
%\psi_1(r_1)
\nonumber
%\label{eq:r1min1}
\\
\psi_1^{\min} &:=& \psi(r_1^o;32),
\nonumber
\\
\xi_1^{\min} &:=& \xi(r_1^o;4),
\nonumber
%\label{eq:r1min2}
\end{eqnarray*}
and
\begin{eqnarray*}
\left.
\left(\V_{\I_1}\left[
\sigma(\reg)e(\reg)\right]
+
\V_{\I_2}\left[
\sigma(\reg)e(\reg)\right]
\right)
\right|_{r_1 = r_1^o}
 = 
\frac{1}{2160}
\frac{1}{2\kappaY}
\left(
d_2^5 \psi_1^{\min}
+
20\kappaY^2(n-1)d_2^3
\xi_1^{\min}
\right).
%d_2^5 \psi_1^{\min}.
\end{eqnarray*}
Note that the optimal $r_1^o$ is independent of the value of $d_2$, which is
the upper boundary of $\S_2$.

Next, we successively consider another subsection $\S_3$ on $y$ (or $\I_3$ on $\reg$) together with
$\S_1$ (or $\I_1$) 
%$S_{-1}$ (or $I_{-1}$) 
and
$\S_2$ (or $\I_2$).
Assume the relation between $d_2$ and $d_3$ is:
\begin{equation*}
d_2 = r_2 d_3,
\end{equation*}
where $r_2$ is an appropriate number in $[0, \, 1]$.
Similar to the case of 
$\S_1$ and $\S_2$, the offset $h_3$ of
$\qythree$ for the subsection $\S_3$ on $y$ (or $\I_3$ on $\reg$) satisfying
$\E_{\I_3}\left[\reg e(\reg)\right] = 0$ is
\begin{eqnarray*}
h_3 &=& \frac{2}{3} k_3^2 \frac{1}{d_3(1+r_2)}
 =  \frac{1}{6}\frac{(1-r_2)^2}{(1+r_2)}d_3, \; k_3: = \frac{d_3(1-r_2)}{2}
\end{eqnarray*}
and $\V_{\I_3}\left[\sigma(\reg)e(\reg)\right]$ can be
given as
\begin{eqnarray*}
\V_{\I_3}\left[\sigma(\reg)e(\reg)\right]
& = &
\int_{-k_3}^{k_3}
\left\{
\left(
\frac{d_3(1+r_2)}{2}+z
\right)^2
+\frac{1}{3}\kappaY^2(n-1)
\right\}
(z-h_3)^2
\frac{1}{2\kappaY}
dz
\nonumber
\\
& = &
\frac{1}{2160}
\frac{1}{2\kappaY}
d_3^5
\left\{
-18(1-r_2)^5 + 45 (1+r_2)^2(1-r_2)^3 + 5(1-r_2)^7(1+r_2)^{-2}
\right\}
\nonumber
\\
&& \mbox{}+
\frac{1}{108}\frac{1}{2\kappaY}\kappaY^2(n-1)d_3^3
\left\{
3(1-r_2)^3+\frac{(1-r_2)^5}{(1+r_2)^2}
\right\}
\end{eqnarray*}
Therefore, the optimal $r_2^o$ that minimizes
$\V_{\I_1}\left[\sigma(\reg)e(\reg)\right]
+
\V_{\I_2}\left[\sigma(\reg)e(\reg)\right]
+
\V_{\I_3}\left[\sigma(\reg)e(\reg)\right]$
is found by solving the following minimization problem:
\begin{eqnarray}
r_2^o&: = &
\arg\min_{r_2} 
\left(
\V_{\I_1}\left[\sigma(\reg)e(\reg)\right]
+
\V_{\I_2}\left[\sigma(\reg)e(\reg)\right]
+
\V_{\I_3}\left[\sigma(\reg)e(\reg)\right]
\right)
\nonumber
\\
& = &
\arg \min_{r_2}
\frac{1}{2160}
\frac{1}{2\kappaY}
\left(
d_3^5 \psi(r_2; \psi_1^{\min})
+
20\kappaY^2(n-1)d_3^3
\xi(r_2; \xi_1^{\min})
\right)
%\frac{1}{2160}
%\frac{1}{2\kappaY}
%d_3^5 \psi_2(r_2)
\nonumber
\\
\psi(r_2 ; \psi_1^{\min})
&: = &
\psi_1^{\min}r_2^5  -18(1-r_2)^5 + 45 (1+r_2)^2(1-r_2)^3 + 5(1-r_2)^7(1+r_2)^{-2},
\nonumber
\\
\xi(r_2 ; \xi_1^{\min}) &:=& 
\xi_1^{\min}r_2^3+3(1-r_2)^3+\frac{(1-r_2)^5}{(1+r_2)^2}.
\label{eq:f2}
\end{eqnarray}
By repeating the above process, we obtain the result.
\hfill{$\square$}

\begin{lemma}
\label{lemma:a1}
A rational function
\[
\psi(r) : =  \alpha r^5  -18(1-r)^5 + 45 (1+r)^2(1-r)^3 + 5(1-r)^7(1+r)^{-2}
\]
has only one local minimum in $r \in (0, \; 1)$ when $\alpha > 0$.
\end{lemma}
Refer to \cite{Tsumura:METR05-04} for the proof.

\noindent
{\bf Slutsky's theorem}
\[
\mathop{\rm plim}_{i\rightarrow \infty}[X(i)^{-1}Y(i)] 
=
(\mathop{\rm plim}_{i\rightarrow \infty}[X(i)])^{-1}
\mathop{\rm plim}_{i\rightarrow \infty}[Y(i)] 
\]
subject to that ${\rm plim}_{i\rightarrow \infty}[X(i)]$ 
and ${\rm plim}_{i\rightarrow \infty}[Y(i)]$ exist.

\noindent
{\bf Proof of Lemma~\ref{lemma:rate}}

\noindent
From Lemma~\ref{lemma:a1}, it is known that $\psi(r,\psi_0^{\min}=32)$ has only one local
minimum in $r \in (0, 1)$.
Moreover, from
\[
\psi(0; \alpha) = 32, \; \forall \alpha > 0, \;
\psi(1; \psi_{j-1}^{\min}) = \psi_{j-1}^{\min},  \; \psi_0^{\min} = 32,
%\psi_j(0) = 32, \; \forall j, \;
%\psi_j(1) = \psi_{j-1}^{\min},  \; \psi_0^{\min} = 32,
\]
the minimum value $\psi_1^{\min}$ satisfies
\[
\psi_1^{\min} < 32.
\]
Next, $\psi(r; \psi_1^{\min})$ satisfies
\[
\psi(0; \psi_1^{\min}) = 32, \; \psi(1; \psi_1^{\min}) = \psi_1^{\min} < 32,
\]
and also $\psi(r; \psi_1^{\min})$ has only one local minimum in $r \in (0, 1)$.
This means 
\[
\psi_1^{\min} > \psi_2^{\min}. 
\]
The difference between $\psi(r; \psi_0^{\min})$ and $\psi(r; \psi_1^{\min})$ is only the
coefficient of the term $r^5$ and 
%Moreover, the terms $r_1^5$ and $r_2^5$ are 
$r^5$ is a strictly increasing function in $(0, 1]$.
Therefore, with $\psi_0^{\min} > \psi_1^{\min}$, 
\begin{equation*}
r_1^o < r_2^o < 1.
\end{equation*}
By repeating the same process, we finally obtain:
\begin{equation*}
r_1^o < r_2^o < r_3^o < \cdots < 1.
\end{equation*}
\vspace{-.5cm}

Next to show $\lim_{j\rightarrow \infty} r_j^o = 1$. 
Let $\lim_{j\rightarrow \infty} r_j^o = r_\infty$.
Then, $r_\infty$ satisfies:
\begin{eqnarray}
&&
r_\infty : =  \arg \min_{r \in [0, 1]} \psi(r; \psi_\infty^{\min})
\nonumber
\\
&&
%\psi(r; \psi_\infty^{\min}) : =  \psi_{\infty}^{\min}r^5  -18(1-r)^5 + 45 (1+r)^2(1-r)^3 + 5(1-r)^7(1+r)^{-2}
%\label{eq:finftyr}
%\\
%&&
\psi_\infty^{\min} =  \psi(r_\infty; \psi_\infty^{\min}).
\nonumber
\end{eqnarray}
Note that if $\psi_\infty^{\min} > 0$, $\psi(r; \psi_\infty^{\min})$ 
also has only one local minimum in $r \in (0, \; 1)$.  
On the other hand, when $\psi_\infty^{\min} = 0$, 
it is also known that $\psi(r; \psi_\infty^{\min})$ 
is a decreasing function in $r \in [0, \; 1]$
from the proof of Lemma~\ref{lemma:a1} and 
$\min_r \psi(r; \psi_\infty^{\min}) = \psi(1; \psi_\infty^{\min})$.   
From \rref{eq:function0}, $\psi(1; \psi_\infty^{\min}) = \psi_{\infty}^{\min}$, 
%From \rref{eq:finftyr}, $\psi(1; \psi_\infty^{\min}) = \psi_{\infty}^{\min}$, 
and the minimum is at $r = 1$.  
This means $r_{\infty} = 1$ (and $\psi_\infty^{\min} = 0$). 
\hfill{$\square$}

\noindent
{\bf Proof of Lemma~\ref{lemma:rate2}}

\noindent
On the subsections $\I_j$ ($\S_j$) and $\I_{j+1}$ ($\S_{j+1}$), 
i.e., the general case for \rref{eq:rdef1} -- \rref{eq:f2}, 
from:
\begin{eqnarray*}
\int_{-k_j}^{k_j}
\left(
\frac{d_j+d_{j+1}}{2}
+z
\right)
\left(
z - h_j
\right) dz
 = 
\frac{2}{3}k_j^3 - (d_j + d_{j+1}) h_j k_j,
\end{eqnarray*}
the offsets $h_j$ and $h_{j+1}$ such that
$\E_{\I_j}\left[\reg e(\reg)\right]  = 0$ and 
$\E_{\I_{j+1}}\left[\reg e(\reg)\right]  = 0$
are given by:
\begin{eqnarray*}
&&
h_j = \frac{2}{3}\frac{1}{d_j + d_{j+1}} k_i^2, \;
k_j : =  \frac{d_{j+1} - d_j}{2},
\;
h_{j+1} = \frac{2}{3}\frac{1}{d_{j+1} + d_{j+2}} k_{j+1}^2, \;
k_{j+1} : =  \frac{d_{j+2} - d_{j+1}}{2}.
\end{eqnarray*}
On the other hand, $\V_{\I_j}\left[\reg e(\reg)\right]$ is calculated by:
\begin{eqnarray*}
\V_{\I_j}\left[\reg e(\reg)\right]
 = 
%& = &
\int_{-k_j}^{k_j}
\left(
\frac{d_j+d_{j+1}}{2}
+
z
\right)^2
\left(
z - h_j
\right)^2 dz
%\nonumber
%\\
%& = &
 = 
A
\left(
d_{j+1} - d_j
\right)^5
+
B
\left(
d_j + d_{j+1}
\right)^2
\left(
d_{j+1} - d_j
\right)^3,
\end{eqnarray*}
where
\begin{equation*}
A : =  \frac{1}{5\cdot 2^4} - \frac{1}{3^2\cdot 2^3} < 0,
\;
B : =  \frac{1}{3\cdot 2^4} > 0.
\end{equation*}
Therefore:
\begin{eqnarray}
%&&
\V_{\I_j}\left[\reg e(\reg)\right]
+
\V_{\I_{j+1}}\left[\reg e(\reg)\right]
%\nonumber
%\\
& = &
A (d_{j+1} - d_{j})^5 + B(d_{j+1}+d_j)^2(d_{j+1}-d_j)^3
\nonumber
\\
&&
\mbox{} 
+ A (d_{j+2} - d_{j+1})^5 +
B(d_{j+2}+d_{j+1})^2(d_{j+2}-d_{j+1})^3
\nonumber
\\
&=:& Z(d_{j+1}).
\label{eq:v1pv2}
\end{eqnarray}
For given $d_j$ and $d_{j+2}$, 
consider which side the minimum point of $Z(d_{j+1})$ is on from the
center of $d_j$ and $d_{j+2}$.  
From $A < 0$ and $B > 0$ and 
the symmetric structure of 
$Z(d_{j+1})$, except for the terms 
$(d_{j+1}+d_j)^2$ and $(d_{j+2}+d_{j+1})^2$ where
$B(d_{j+1}+d_j)^2  < B(d_{j+2}+d_{j+1})^2$, 
it is known that $Z(d_{j+1})$ 
has its minimum at $d_o > \frac{d_j + d_{j+2}}{2}$.
This means $|\I_j| > |\I_{j+1}|$, that is, $|\S_j| > |\S_{j+1}|$.
The same applies for arbitrary sections $\I_j$ and
$\I_{j+1}$, 
and we can conclude the statement is true.
\hfill $\square$

\noindent
{\bf Proof of Lemma~\ref{lemma:asympt}}

\noindent
From Lemma~\ref{lemma:rate} and its proof, it is known that when $j$ $\rightarrow$ $\infty$,
$r_j^o$ and $\psi_j^{\min}$ converge to 1 and 0, respectively.
Therefore, by employing the Taylor series expansion, 
$\psi(r; \psi_{j-1}^{\min})$ 
can be represented by:
%Therefore, by employing the Taylor series expansion, $\psi_{j}(r)$ can be represented by:
\begin{equation*}
\psi(r; \psi_{j-1}^{\min}) = \psi_{j-1}^{\min} (1-5(1-r)+10(1-r)^2-10(1-r)^3)
+ 45\cdot 2^2(1-r)^3 + O((1-r)^4)
%\psi_{j}(r) = \psi_{j-1}^{\min} (1-5(1-r)+10(1-r)^2-10(1-r)^3)
%+ 45\cdot 2^2(1-r)^3 + O((1-r)^4)
\end{equation*}
near $r = 1$ at sufficiently large $j$.  
By applying a variable transformation $1-r  = : \epsilon$, we obtain
\begin{equation}
\psi(\epsilon; \psi_{j-1}^{\min}) 
= \psi_{j-1}^{\min} (1-5\epsilon +10\epsilon^2-10\epsilon^3)
+ 180 \epsilon^3 + O(\epsilon^4)
%\psi_{j}(\epsilon) = \psi_{j-1}^{\min} (1-5\epsilon +10\epsilon^2-10\epsilon^3)
%+ 180 \epsilon^3 + O(\epsilon^4)
\label{eq:poly0}
\end{equation}
at $\epsilon \rightarrow 0$.
Denote the local minimum of $\psi(\epsilon; \psi_{j-1}^{\min})$ as
%Denote the local minimum of $\psi_{j}(\epsilon)$ as
$\epsilon_{j}$, then $\epsilon_j$ must satisfy: 
\begin{equation}
\psi_{j-1}^{\min} (-5 +20\epsilon_j-30\epsilon_j^2)
+ 540 \epsilon_j^2 + O(\epsilon_j^3)
 = 0.
\label{eq:orderofepsilon0}
\end{equation}
From 
\rref{eq:orderofepsilon0}, 
it is simple to verify that:
\begin{equation}
\epsilon_j = \left(\frac{1}{108}\psi_{j-1}^{\min}\right)^{1/2} 
+ 
o
\left(
\left(
\psi_{j-1}^{\min}
\right)^{1/2}
\right) 
\label{eq:order2}
\end{equation}
at $\psi_{j-1}^{\min} \rightarrow 0$. 
On the other hand, from \rref{eq:poly0}, 
$\psi_{j}^{\min}$ is represented by:
\begin{equation}
\psi_{j}^{\min}
 =  \psi_{j-1}^{\min} (1-5\epsilon_j +10\epsilon_j^2-10\epsilon_j^3)
+ 180 \epsilon_{j}^3 + O(\epsilon_j^4),
\label{eq:hatf}
\end{equation}
and with \rref{eq:order2}, 
we obtain:
\begin{eqnarray}
\psi_{j}^{\min}- \psi_{j-1}^{\min} 
& = &
-5
\left(
\frac{1}{108}
\right)^{1/2}
{\psi_{j-1}^{\min}}^{3/2}
+
180
\left(
\frac{1}{108}
\right)^{3/2}
{\psi_{j-1}^{\min}}^{3/2}
+
O({\psi_{j-1}^{\min}}^{2}
)
\nonumber
\\
& = &
-5\cdot 3^{-\frac{5}{2}}
%\left(
{\psi_{j-1}^{\min}}
%\right)
^{\frac{3}{2}}
+
O({\psi_{j-1}^{\min}}^{2})=:{\cal P}(\psi_{j-1}^{\min}).
\label{eq:symmain}
%\label{eq:hatf2}
\end{eqnarray}
With the convergence $\psi_j^{\min} \rightarrow 0$, 
we derive the statement of the lemma. 
\hfill$\square$

\noindent
\begin{lemma}\label{lemma:recurrentdiff}
$\tilde{\psi}(m) \geq \hat{\psi}(m)$ at $m = 0$, $1$, ... , 
when $\tilde{\psi}(0) \geq \hat{\psi}(0)$. 
\end{lemma}
%Refer to \cite{Tsumura:METR05-04} for the proof.
\begin{proof}
First define $\hat{\psi}'(m)$ for $m \in {\cal R}$, which is a simple linear interpolation of 
$\hat{\psi}(m)$ at $m = 0$, 1, ... , 
and the gradient between 
$\hat{\psi}'(m-1)$ and $\hat{\psi}'(m)$ ($m = 1$, $2$, ...) 
is a constant ${\cal P}(\hat{\psi}'(m-1))=a \hat{\psi}'^b(m-1) + 
o(\hat{\psi}'^b(m-1))$ ($< 0$). 
Assume that $\tilde{\psi}(m)$ crosses $\hat{\psi}'(m)$ downward at
 $m = m'$ between $m-1$ and $m$. 
Note that $\tilde{\psi}(m') < \hat{\psi}'(m-1) = \hat{\psi}(m-1)$,
 therefore, 
\begin{eqnarray*}
\left.
\frac{d\tilde{\psi}(m)}{dm}
\right|_{m=m'}
=
(a+\nu)\tilde{\psi}^b(m')
\geq
{\cal P}(\tilde{\psi}(m'))
>
{\cal P}(\hat{\psi}'(m-1))
=
a \hat{\psi}'^b(m-1) + o(\hat{\psi}'^b(m-1)).
\end{eqnarray*}
This contradicts the assumption $\tilde{\psi}(m')$ crosses
$\hat{\psi}'(m')$ downward.  
\end{proof}

\noindent
{\bf Proof of Theorem~\ref{lemms:quantterm}}
% {\rm \cite{Tsumura:CDC99}}

\noindent
First evaluate the magnitude of $\tilde{U}^{\rm T}E$.
%Its first element $(\tilde{U}^{\rm T}E)_1$ is of form
%\[
%\reg(1)e(1) + \reg(2)e(2) + \cdots + \reg(N)e(N).
%\]
%From the independence of $\reg(\i)$ and \rref{eq:variapp}, the expectation
%and the square of $(\tilde{U}^{\rm T}E)_1$ are given by:
From 
\rref{eq:vari1001app}, \rref{eq:objectivef}, and \rref{eq:minv}, 
\[
\E\left[\tilde{U}^{\rm T}E\right] = 0,\;
\V\left[\tilde{U}^{\rm T}E \right] = 
\frac{1}{12}\th1^2 D^3 M^{-2} N.
%\E\left[(\tilde{U}^{\rm T}E)_1\right] = 0,\;
%\V\left[(\tilde{U}^{\rm T}E)_1\right] \leq
%N A \kappa^4 (M-B^\star)^{-2}.
\]
Then by Chebyshev's inequality,
% \rref{eq:chebyshev}, 
we obtain:
\[
\mbox{\rm Prob}\left(\|\tilde{U}^{\rm T}E\|_\infty \geq
%\mbox{\rm Prob}\left(|(\tilde{U}^{\rm T}E)_1| \geq
\sqrt{\frac{n}{\beta_2}\frac{1}{12}\th1^2 D^3 M^{-2}N}
%\sqrt{\frac{NA\kappa^4}{\beta_2(M-B^\star)^2}
\right) \leq \beta_2,
\]
for a reliability index $\beta_2$.
Combine $(\tilde{U}^{\rm T}\tilde{U})^{-1}$ and $\tilde{U}^{\rm T}E$ using the norm inequality:
\[
\|(\tilde{U}^{\rm T}\tilde{U})^{-1}\tilde{U}^{\rm T}E\|_\infty \leq
\|(\tilde{U}^{\rm T}\tilde{U})^{-1}\|_1
\|\tilde{U}^{\rm T}E\|_\infty,
\]
and this gives:
\[
\mbox{\rm Prob}\left(
\|(\tilde{U}^{\rm T}\tilde{U})^{-1}\tilde{U}^{\rm T}E\|_\infty
\leq \epsilon_1\epsilon_2
\right)
\geq
\mbox{\rm Prob}\left(
\|(\tilde{U}^{\rm T}\tilde{U})^{-1}\|_1 \leq \epsilon_1 \mbox{ and }
\|\tilde{U}^{\rm T}E\|_\infty \leq \epsilon_2
\right).
\]
Therefore we have proved the statements.
\hfill$\square$

\end{document}